\documentclass[aos,reqno]{imsart}
\usepackage{amsmath,amssymb,amsfonts,amsbsy,amsthm,epsfig,subfig,graphicx,rotating,mathrsfs}
\usepackage{fancyhdr}

\RequirePackage[OT1]{fontenc}
\RequirePackage{amsthm,amsmath}
\RequirePackage[numbers]{natbib}
\RequirePackage[colorlinks,citecolor=blue,urlcolor=blue]{hyperref}

\startlocaldefs
\numberwithin{equation}{section}
\theoremstyle{plain}

\endlocaldefs

\newcommand{\noi}{\noindent}

\newcommand{\beq}{\begin{eqnarray*}}
\newcommand{\eeq}{\end{eqnarray*}}
\newcommand{\beqn}{\begin{eqnarray}}
\newcommand{\eeqn}{\end{eqnarray}}

\newcommand{\var}{{\rm var}}
\newcommand{\bi}{\begin{itemize}}
\newcommand{\ei}{\end{itemize}}

\newcommand{\ka}{\kappa}
\newcommand{\be}{\begin{equation}}
\newcommand{\ee}{\end{equation}}
\newcommand{\nn}{\nonumber}

\newcommand{\con}{\mbox{const.}}
\newcommand{\PI}{\rm{PI}}
\newcommand{\ISE}{\mbox{ISE}}

\newcommand{\ignore}[1]{}{}

\newcommand{\sn}{\sum_{i=1}^n}

\newcommand{\De}{\Delta}

\newcommand{\s}{\sqrt}

\newcommand{\br}{\mathbb{R}}

\newcommand{\LL}{{\rm LL}}
\newcommand{\NW}{{\rm NW}}

\newcommand{\argmin}{\mathop{\rm arg\min}}

\numberwithin{equation}{section}
\theoremstyle{plain}

\newtheorem{theorem}{Theorem}[section]

\newtheorem{remark}{Remark}[section]

\def\T{^{{\rm T}}}
 
\allowdisplaybreaks

\usepackage{xr}

\begin{document}

\begin{frontmatter}
\title{Nonparametric covariate-adjusted regression}
\runtitle{Covariate-adjusted regression}

\begin{aug}
  \author{\fnms{Aurore}  \snm{Delaigle}\corref{}\thanksref{t2}\ead[label=e1]{A.Delaigle@ms.unimelb.edu.au}},
  \author{\fnms{Peter} \snm{Hall}\thanksref{t2}\ead[label=e2]{halpstat@ms.unimelb.edu.au}}
\and
  \author{\fnms{Wen-Xin}  \snm{Zhou}\corref{}\thanksref{t2}\ead[label=e3]{wenxinz@princeton.edu}}

  \thankstext{t2}{Research supported by grants and fellowships from the Australian Research Council.}

  \runauthor{A. Delaigle, P. Hall and W.-X. Zhou}

  \affiliation{University of Melbourne and Princeton University}

 \address{School of Mathematics and Statistics \\ 
 	University of Melbourne \\ 
 	Parkville, Victoria 3010 \\
    Australia \\ 
\printead{e1}\\
\phantom{E-mail:\ }\printead*{e2}\\
\phantom{E-mail:\ }\printead*{e3}}

 \address{Department of Operations Research \\
			  and Financial Engineering \\
			   Princeton University \\
			   Princeton, New Jersey 08544 \\
			   USA \\  \printead{e3}}

\end{aug}

\begin{abstract}
We consider nonparametric estimation of a regression curve when the data are observed with multiplicative distortion which depends on an observed confounding variable. We suggest several estimators, ranging from a relatively simple one that relies on restrictive assumptions usually made in the literature, to a sophisticated piecewise approach that involves reconstructing a smooth curve from an estimator of a constant multiple of its absolute value, and which can be applied in much more general scenarios. We show that, although our nonparametric estimators are constructed from predictors of the unobserved undistorted data, they have the same first order asymptotic properties as the standard estimators that could be computed if the undistorted data were available. We illustrate the good numerical performance of our methods on both simulated and real datasets.
\end{abstract}


\begin{keyword}
\kwd{discontinuities}
\kwd{local linear estimator}
\kwd{multiplicative distortion}
\kwd{Nadaraya-Watson estimator}
\kwd{nonparametric smoothing}
\kwd{predictors}
\end{keyword}

\end{frontmatter}

\pagenumbering{arabic}

\section{Introduction}
We consider nonparametric estimation of a regression curve $m(x)=E(Y|X=x)$ when $X$ and $Y$ are observed with multiplicative distortion induced by an observed confounder~$U$. 
Specifically, we observe $\tilde{X}$, $\tilde{Y}$ and $U$, where  $\tilde{Y} = \psi(U)\, Y$, $\tilde{X} = \varphi(U)\, X$, $\psi$ and $\varphi$ are unknown functions and $U$ is independent of $X$ and $Y$.
This model  is known as a covariate-adjusted regression model. It was introduced by \c{S}ent\"urk and M\"uller~(2005a) to generalize an approach commonly employed in medical studies, where
the effect of a confounder $U$, for example body mass index, is often removed by dividing by $U$. Motivated by the fibrinogen data on haemodialysis patients, where $\tilde Y$ was fibrogen level, $\tilde X$ was serum transferrin level, and $U$ was body mass index, \c{S}ent\"urk and M\"uller~(2005a) pointed that although it is often reasonable to assume that the effect of $U$ is multiplicative, it does not need to be proportional to $U$, and a more flexible model is obtained by allowing for distortions represented by the functions $\varphi$ and $\psi$. More generally, this model is useful to describe the relationship between variables that are influenced by a confounding variable, and see if this relationship still exists once the effect of the confounder has been removed.

A number of authors have suggested estimators of the curve $m$ in various parametric settings. 
Linear regression models were considered by \c{S}ent\"urk and M\"uller~(2005a, 2006) and \c{S}ent\"urk and Nguyen~(2006), who generalized them to varying coefficient models (\c{S}ent\"urk, ~2006) and generalized linear models (\c{S}ent\"urk and M\"uller,~2009).
A more general nonlinear regression model was suggested by Cui et al.~(2009) and Zhang et al.~(2012), and in Zhang et al.~(2013), the authors considered a partially linear model, where the linear part is observed with multiplicative distortions.

In this work, we propose more flexible nonparametric estimators of the regression function $m$, which not only relax the parametric assumptions imposed in the existing literature, but also significantly weaken some of the strong assumptions on the curves $\varphi$ and $\psi$ and on the distribution of the data made by previous authors. In particular, we propose estimators which, unlike in the previous studies, can be applied if $EX$ and $EY$ vanish, and even if the functions $\psi$ and $\varphi$ are not strictly positive. Our procedures involve estimating the functions $\varphi$ and $\psi$, deduce from there  predictors of $X$ and $Y$, and construct nonparametric estimators of $m$ using those predictors. 
We show that, under the restrictive assumptions made in the existing literature, this is relatively straightforward to do, whereas under the much weaker assumptions we also consider, we need to use a sophisticated approach.

This paper is organized as follows. We describe the covariate-adjusted model and discuss the model restrictions in the existing literature in Section~\ref{sec:model}. We propose several nonparametric estimators in Section~\ref{np.estimation.sec}, ranging from the most basic estimators which can be applied under similar restrictions as those imposed in the existing literature, to the most sophisticated ones which rely on much milder assumptions. We derive theoretical properties of our estimators in Section~\ref{theo.sec}, where we show that they have the same first order asymptotic properties as the nonparametric estimators that could be computed if $X$ and $Y$ were observed directly. More surprisingly, in some particular cases, our new estimators can even achieve faster convergence rates than the standard estimators based on direct observations from $(X,Y)$. 
We discuss practical implementation of our methods in Section~\ref{sec:implem}, where we also investigate their performance on simulated data, and apply them to analyze two real datasets studied in \c{S}ent\"urk and M\"uller~(2005b) and \c{S}ent\"urk and Nguyen~(2006). We discuss multivariate extensions in Section~\ref{sec:multivariate}. Our proofs are provided in Section~\ref{sec:proofs} and in a supplementary file.

\section{Model and data}\label{sec:model}

We observe independent and identically distributed (i.i.d.) triplets $\{ (\tilde{X}_i, \tilde{Y}_i, U_i)\}_{i=1}^n$ generated by the covariate-adjusted model of \c{S}ent\"urk and M\"uller~(2005a), where
\begin{equation} 
	 Y  = m(X) + \sigma(X) \, \varepsilon,  \  \    
    \tilde{Y} = \psi(U) \, Y,   \ \ 
      \tilde{X} = \varphi(U) X , 
      \label{basic.model}
\end{equation}
with  $m(x) = E( Y| X=x)$ an unknown regression curve that we wish to estimate nonparametrically, $\sigma^2(x)=\var(Y|X=x)$ an unknown variance function, and $\varphi$ and $\psi$ unknown smooth functions. The random variables $U, X$ and $\varepsilon$ are mutually independent, $Y$ and $U$ are independent, $E(\varepsilon)=0$ and $\var(\varepsilon)=1$. We use $f_X$ and $f_U$ to denote the densities of $X$ and $U$, respectively. As in \c{S}ent\"urk and M\"uller~(2005), to make the problem identifiable, we assume that
\be 
	E \{ \varphi(U) \} = E\{ \psi(U) \} = 1.  \label{SM.cond1}
\ee
In other words, on average there is no distorting effect, which is similar to the standard condition imposed in the related classical measurement error problems (Carroll and Hall,~1988; Fan and Truong,~1993), where one observes $W=X+U$ with $X$ and $U$ independent, and the measurement error $U$ is assumed to have zero mean.

As mentioned in the introduction, several parametric estimators of  $m$ have been suggested in the literature. There, it is commonly assumed that
\be \label{positivity.ass}
(a)\  \varphi(u), \, \psi(u) >0    \ \mbox{ for all } u\in I_U\,,
(b)\  \mbox{$E (X)\neq 0$~~and~~$E(Y) \neq 0$\,,}
\ee
where $I_U\equiv  [u_L, u_R]$ denotes the compact support of $U$. Without loss of generality, we assume that $I_U=[0,1]$ throughout the paper. 

An approach used by some authors is based on constructing predictors of the $(X_i,Y_i)$'s, which can be obtained from the data $(\tilde X_i,\tilde Y_i,U_i)$, $i=1,\ldots,n$, on noting that
\be
	\varphi_0(U_i)\equiv  E(  \tilde{X}_i  | U_i) = \varphi(U_i)\,E (X) ,  \psi_0(U_i)\equiv E(  \tilde Y_i | U_i ) = \psi(U_i)\,E (Y). \label{equiv.mod.ass}  
\ee
Now, $\varphi$ and $\psi$ can easily be estimated nonparametrically, say by $\hat\varphi$ and $\hat\psi$, which motivates Cui et al.'s~(2009) predictors
$\hat Y_i =  \{ \hat \psi(U_i) \}^{-1} \tilde Y_i$ and $\hat X_i = \{ \hat \varphi(U_i)\}^{-1}  \tilde X_i$, and shows that \eqref{positivity.ass} is  needed by those authors to avoid dividing by zero. 
In the next section, we shall see that it is possible to construct consistent nonparametric estimators of $m$, and that this can be done under much less restrictive conditions than \eqref{positivity.ass}.

\section{Methodology}
\label{np.estimation.sec}
\subsection{Different methods under different conditions}\label{sec:methods}

The parametric methods developed in the  literature crucially rely  on   assumption \eqref{positivity.ass}, and the examples considered there
are always such that $\varphi$, $\psi$, $EX$ and $EY$ are far from zero. We wish to construct nonparametric estimators of $m$ that are  consistent even if those assumptions do not hold. 
Let $\mathbf{e}_1=(1,0)\T$, and, for any pairs of random variables $(Q,R)$ and $(Q_i,R_i)$, $i=1,\ldots,n$, let  
${\mathbf{S}}_{Q,n}(x;K,h)= n^{-1}\sn K_h(Q_i-x) \mathbf{w}\{h^{-1}(Q_i-x)\} \mathbf{w}\{h^{-1}(Q_i - x)\}\T\in \br^{2\times 2}$ and ${\mathbf{T}}_{Q,R,n}(x;K,h)=n^{-1}\sn R_i K_h(Q_i-x)\mathbf{w}\{h^{-1}(Q_i-x)\}$, with $\mathbf{w}(s)=(1,s)\T$ and where $K$ is a kernel function, $h=h_n>0$ is a bandwidth and, for every $t\in \br$, $K_h(t)=h^{-1}K(t/h)$. 

If the $(X_i,Y_i)$'s were available, we could estimate $m(x)$ nonparametrically by a standard local polynomial estimator constructed from the $(X_i, Y_i)$'s, the two most popular versions of which are the Nadaraya-Watson and the local linear estimators, defined~by
\be 
	\tilde  m_{\NW}(x) =\frac{ \sn  Y_i K_h(x-X_i)    }{\sn K_h(x-X_i)} \, , \,  \tilde  m_{\LL}(x)   =     \mathbf{e}_1\T {\mathbf{S}}_{X,n}^{-1}(x;K,h) {\mathbf{T}}_{X,Y,n}(x;K,h), \label{standardLP}
\ee
respectively. In our case, the $(X_i,Y_i)$'s are not observed and these standard estimators cannot be computed. We develop new nonparametric estimators that can be computed from the $(\tilde X_i,\tilde Y_i,U_i)$'s, and whose complexity depends on whether \eqref{positivity.ass}(a) and (b) are satisfied or not.
The simplest situation is the one where \eqref{positivity.ass}(a) holds. There, we can estimate $m$ by standard nonparametric estimators based on predictors of the $(X_i,Y_i)$'s that are similar to, but less restrictive than, those used by Cui et al.~(2009); see Section~\ref{one-step}.
The case where we do not assume \eqref{positivity.ass}(a) requires more elaborate techniques:  
in Section~\ref{two-stage.estiamtion.sec}, we suggest a method that can be used when \eqref{positivity.ass}(b) is satisfied;  
we handle the most general case in Section~\ref{sec:fourth}, where we develop a sophisticated method which is valid regardless of whether \eqref{positivity.ass}(a) and (b) hold or not. It involves computing estimators of unknown constant multiples of $|\varphi|$ and $|\psi|$, estimate the zeros of those functions, construct piecewise estimators of unknown constant multiples of $\varphi$ and $\psi$, estimate these constants and finally deduce estimators of $\varphi$ and $\psi$.

\subsection{Basic method}\label{one-step}
We start by deriving simple nonparametric estimators of $m$ that can be computed when \eqref{positivity.ass}(a) holds, and which form the basis of the more sophisticated methods we introduce in the subsequent sections. The idea is similar to the one used in the parametric context by Cui at al.~(2009):  replace the unobserved $(X_i,Y_i)$'s by predictors $(\hat X_i,\hat Y_i)$.
Under~\eqref{positivity.ass}, motivated by \eqref{equiv.mod.ass} and  since $EX=E\tilde X$ and $E Y=E\tilde Y$, Cui et al.~(2009) 
take $\hat Y_i =  \{ \hat \psi(U_i) \}^{-1} \tilde Y_i$   and $\hat X_i = \{ \hat \varphi(U_i)\}^{-1}  \tilde X_i$, where $\hat\varphi$ and $\hat\psi$ denote  Nadaraya-Watson estimators of $ \varphi_0$ and $\psi_0$, divided by, respectively,  $\widehat {EX} =n^{-1} \sum_{i=1}^n \tilde{X}_i$ and $\widehat {EY} =  n^{-1} \sn \tilde Y_i$. 

It is because of this division that Cui at al.~(2009)  assume \eqref{positivity.ass}(b), but the latter can be avoided and replaced by $E|X|, E|Y|\neq0$ (which holds for all non-degenerate random variables), by better exploiting \eqref{positivity.ass}(a). 
Specifically, under \eqref{positivity.ass}(a), $|\psi|=\psi$, $|\varphi|=\varphi$,  and
\be
	\varphi_0^+(U_i)\equiv  E\big(  |\tilde{X}_i|   \big|   U_i\big) = \varphi(U_i)\,E |X|\, ,   \psi_0^+(U_i)\equiv E\big(  |\tilde Y_i | \big| U_i \big) = \psi(U_i)\,E |Y|\,. 	\label{equiv.mod.assB}  
\ee
Motivated by this, we propose to estimate $\psi$ and $\varphi$ by
\be 
  \hat \varphi_{\LL}(u) =  \hat{\varphi}_{0, \LL}^+(u)/\widehat{ E|X|} \  \mbox{ and }  \   \hat \psi_{\LL}(u) = \hat{\psi}_{0, \LL}^+(u)/\widehat{ E|Y|}\,,   \label{hatphi1}
\ee
where $\widehat{ E|X|} =n^{-1} \sum_{i=1}^n |\tilde{X}_i|$, $\widehat{ E|Y|} =  n^{-1} \sn |\tilde Y_i|$, and where 
$
	 \hat  \varphi_{0,\LL}^+(u)   =    \mathbf{e}_1\T \mathbf{S}_{U,n}^{-1}(u;L,g_1) \allowbreak\mathbf{T}_{U,|\tilde{X}|,n}(u;L,g_1)$ and $\hat \psi_{0,\LL}^+(u)    =    \mathbf{e}_1\T \mathbf{S}_{U,n}^{-1}(u;L,g_2) \mathbf{T}_{U,|\tilde{Y}|,n}(u;\allowbreak L,g_2)  \label{phi.LL.est}
$
are local linear estimators of $\varphi_0^+$ and $\psi_0^+$ computed with a kernel function $L$ and bandwidths $g_1$ and $g_2$.

Then, we predict $Y_i$ and $X_i$ by taking
 \be 
 	\hat Y_i =  \{ \hat \psi_{\LL}(U_i) \}^{-1} \tilde Y_i    \quad  \mbox{ and } \quad \hat X_i = \{ \hat \varphi_{\LL}(U_i)\}^{-1}  \tilde X_i \,. \label{generated.response.predictor}
 \ee 
Finally, replacing $(X_i,Y_i)$ by $(\hat X_i,\hat Y_i)$ in \eqref{standardLP}, 
we obtain the following estimators of $m(x)$:
\be
\hat  m_{\NW}(x) =\frac{ \sn \hat Y_i  K_h(x-\hat X_i)   }{\sn K_h(x-\hat{X}_i)} , 
 \hat  m_{\LL}(x)   =     \mathbf{e}_1\T {\mathbf{S}}_{\hat X,n}^{-1}(x;K,h) {\mathbf{T}}_{\hat X,\hat Y,n}(x;K,h)\,.  \label{LL.m} 
\ee

\begin{remark}\label{rem:simpler}
{\rm
Using  $E(\tilde Y_i |X_i)=E(Y_i |X_i)=m(X_i)$, simpler estimators of $m$ can also be defined  by  
$	\hat m_{\NW,0}(x) = \sn  \tilde Y_i  K_h(x-\hat X_i)   /K_h(x-\hat X_i)$  and $\hat  m_{\LL,0}(x)   =     \mathbf{e}_1\T {{\mathbf{S}}}_{\hat X,n}^{-1}(x;K,h) \allowbreak {\mathbf{T}}_{\hat X,\tilde Y,n}(x;K,h). 
$ 
Since they require predicting only the $X_i$'s, these estimators seem more attractive than those in \eqref{LL.m}. However, it can be proved that their asymptotic ``variance'' is larger than that of the estimators in \eqref{LL.m}.  Moreover, they cannot be adapted simply to the case where $\varphi$ does not satisfy \eqref{positivity.ass}(a); see Remark~\ref{rem:simpler_fail} in Section~\ref{two-stage.estiamtion.sec}.
}
\end{remark}

\subsection{Refined procedure}\label{two-stage.estiamtion.sec}
As their parametric counterparts developed in the covariate-adjusted literature, the methods introduced in Section~\ref{one-step} can only be computed if \eqref{positivity.ass}(a) holds.  However, in practice, there is no  reason why $\varphi$ and $\psi$ would always be positive, and even if they are, their estimators may vanish or get close to zero, which can cause numerical problems.
In this section, we suggest a refined approach which can overcome these difficulties when \eqref{positivity.ass}(b) holds. The more complex case where \eqref{positivity.ass}(b) is violated will be dealt with in Section~\ref{sec:fourth}.

As in Section~\ref{one-step}, to estimate $m$, the first step is to construct predictors $\hat X_i$ and $\hat Y_i$, and thus estimators of $\varphi$ and $\psi$. Recall the notation in \eqref{equiv.mod.ass}. Since we assume \eqref{positivity.ass}(b) but not \eqref{positivity.ass}(a), instead of \eqref{hatphi1} we  take
$\hat Y_i =  \{ \hat \psi_{\LL}(U_i) \}^{-1} \tilde Y_i$   and $\hat X_i = \{ \hat \varphi_{\LL}(U_i)\}^{-1}  \tilde X_i$, where
$\hat \varphi_{\LL}(u) =  \hat{\varphi}_{0, \LL}(u)/\widehat {EX}$ and $\hat \psi_{\LL}(u) = \hat{\psi}_{0, \LL}(u)/\widehat {EY}$,
and the local linear estimators
\begin{align}
	 &\hat  \varphi_{0,\LL}(u)   =    \mathbf{e}_1\T \mathbf{S}_{U,n}^{-1}(u;L,g_1) \mathbf{T}_{U,\tilde{X},n}(u;L,g_1)\, ,\notag\\
	 &\hat \psi_{0,\LL}(u)    =    \mathbf{e}_1\T \mathbf{S}_{U,n}^{-1}(u;L,g_2) \mathbf{T}_{U,\tilde{Y},n}(u;L,g_2)  \label{phi.LL.est2}
\end{align}
of $\varphi_0$ and $\psi_0$ computed with a kernel function $L$ and bandwidths $g_1$ and~$g_2$.

To derive consistent estimators of $m$ without imposing \eqref{positivity.ass}(a), recall that, for each $i$, $X_i$ and $Y_i$ are independent of $U_i$. As a consequence, for any subset $\cal S\subseteq \br$, we have $E(Y_i|X_i=x,U_i\in {\cal S})=E(Y_i|X_i=x)$. In particular, if $X_i$, $Y_i$, $\varphi$ and $\psi$ were known, then letting
$ \mathcal{C}_n(\rho_1,\rho_2) = \{ 1\leq i\leq n:  |\varphi_0 (U_i)|  \geq \rho_1, |\psi_0(U_i)| \geq \rho_2 \}$, with $\rho_1 , \rho_2>0$ denoting two  small numbers, the following modification of $\tilde m_\NW(x)$ at \eqref{standardLP} would be  consistent: 
$$
	 \tilde{m}_\NW(x;\rho_1,\rho_2) = {\sum_{i \in   {\mathcal{C}}_n(\rho_1,\rho_2 )}  Y_i K_h(x-X_i) }\big   /  {\sum_{i \in   {\mathcal{C}}_n(\rho_1,\rho_2 )} K_h(x-X_i) }\,,
$$
and a similar consistent version $\tilde{m}_\LL(x;\rho_1,\rho_2)$ of $\tilde{m}_\LL(x)$  at \eqref{standardLP} could be constructed by replacing, in the definition of $\tilde{m}_\LL(x)$, sums over all $i$ by sums over $i \in   {\mathcal{C}}_n(\rho_1,\rho_2 )$ as above. The advantage of this approach is that it enables us to exclude  the data for which $\psi(U_i)$ or $\varphi(U_i)$ are small, and thus it can be applied even if \eqref{positivity.ass}(a) does not hold.

Motivated by this discussion, in the case that interests us, where $X_i$, $Y_i$, $\varphi$ and $\psi$ are unknown, we suggest estimating $m$ as follows. First, let
$\hat{\mathcal{C}}_n(\rho_1,\rho_2) = \big\{ i =1,\ldots, n:  |\hat{ \varphi}_{0, \LL}(U_i)|  \geq \rho_1, \ |\hat{ \psi}_{0, \LL}(U_i)| \geq \rho_2    \big\}\,.$ (The choice of $\rho_1$ and $\rho_2$ will be discussed in Section~\ref{sec:implem}.) We define a Nadaraya-Watson estimator of $m(x)$, valid even if  \eqref{positivity.ass}(a) does not hold, by
\be 
	\hat m_\NW(x;  \rho_1,\rho_2 ) = {\sum_{i \in  \hat{\mathcal{C}}_n(\rho_1,\rho_2 )} \hat Y_i K_h(x-\hat X_i) }\big /{\sum_{ i\in   \hat{\mathcal{C}}_n (\rho_1,\rho_2)}  K_h(x-\hat X_i) } \,.  	\label{two-step.est.m}
\ee 
Similarly, we define a local linear estimator  $\hat m_\LL(x;  \rho_1,\rho_2 )$ in the same way as $\hat m_\LL$ in \eqref{LL.m}, replacing there, and in the definitions of  ${\mathbf{S}}_{\hat X,n}(x;K,h)$ and  ${\mathbf{T}}_{\hat X,\hat Y,n}(x;K,h)$, the indices $i=1,\ldots,n$ by  the indices $i \in  \hat{\mathcal{C}}_n (\rho_1,\rho_2 )$.

\begin{remark}\label{rem:removepoints}
{\rm
While we shall prove in Section~\ref{theo.sec} that these estimators are consistent and have the same first order asymptotic properties as their counterparts at \eqref{standardLP} based on undistorted data,
in practice performance can be further improved by excluding a small fraction (say  5\%) of the observations corresponding to the $U_i$'s such that a kernel density estimator $\hat f_U(U_i)$ of $f_U(U_i)$ is the smallest. (Indeed, we know from standard
properties of kernel regression estimators that, at points $u$ where $f_U(u)$ is small, $\hat\varphi(u)$ and $\hat\psi(u)$ are more variable.) Doing this corresponds to enlarging the set ${\cal S}$ slightly, which does not affect consistency and convergence rates, again due to the fact that the $U_i$'s are independent of the $(X_i,Y_i)$'s, 
}
\end{remark}

\begin{remark}\label{rem:simpler_fail}
{\rm
It is not possible to directly use this approach to modify the estimator discussed in Remark~\ref{rem:simpler} for the case where $\varphi$ has zeros, because $\tilde Y_i$ and $U_i$ are dependent. Particularly, we note that in general $E(\tilde Y_i|X_i=x,U_i\in {\cal S})$ and $E(\tilde Y_i|X_i=x)$ are not equal. 
}
\end{remark}

\subsection{Elaborate procedure for the most general case}\label{sec:fourth}
Finally we construct estimators of $m$ that rely on neither part of \eqref{positivity.ass}. As before, we start by deriving  predictors of the $(X_i,Y_i)$'s. 
Constructing predictors $\hat X_i$ (resp., $\hat Y_i$) without assuming  \eqref{positivity.ass} requires  to derive an estimator of $\varphi$ (resp., $\psi$) without this assumption, which, unlike the methods used in the previous sections, turns out to be a challenging task.
Our procedure is based on the fact that, from \eqref{basic.model},
$\varphi^*(u)\equiv E( |\tilde X|\, | U=u ) =|\varphi(u)|\, E|X|$ (resp., $\psi^*(u)\equiv E( |\tilde Y|  \,|U \allowbreak =u ) =|\psi(u)|\, E|Y|$)\,,
which implies that we can estimate $\varphi^*$ (resp., $\psi^*$) by a standard local linear estimator $\hat\varphi^*_\LL$  (resp., $\hat\psi^*_\LL$) with kernel $L$ and bandwidth $g_1$ (resp., $g_2$) constructed from the $(U_i,|\tilde X_i|)$'s (resp.,  the $(U_i,|\tilde Y_i|)$'s). In what follows, we explain how to deduce an estimator of $\varphi$ from $\hat\varphi^*_\LL$. The same procedure can be applied to derive an estimator of $\psi$ from $\hat\psi^*_\LL$.

Since $\varphi^*$ is proportional to $|\varphi|$, to extract an estimator of $\varphi$ from $\hat\varphi^*_\LL$, we need to estimate the zeros of $\varphi$, say $\tau_1,\ldots , \tau_M$ for some finite $M$,  at which $\varphi$ changes sign.
To do this we assume that, for each $j$, $\varphi''(\tau_j)\neq 0$. Then, it is straightforward to see that the  first derivative of $\varphi^*$ has jump discontinuities at the $\tau_j$'s. Moreover, the zeros of $\varphi$ coincide with those of $\varphi^*$, so that, at the $\tau_j$'s,  $\varphi^*$ reaches its minimum value, 0. Therefore, the $\tau_j$'s can be estimated using procedures for detecting discontinuities in derivatives of a regression curve, such as those in Gijbels et al.~(1999) and Gijbels and Goderniaux~(2005), combined with the fact that the $\hat\tau_j$'s need to correspond to local minima of $\hat\varphi^*_\LL$; see Section~\ref{sec:param} for details of implementation.  For $j=1,\ldots,M$, let $\hat\tau_j$ denote the resulting estimator of $\tau_j$, and let  $I_0=(-\infty,\hat\tau_1)$, $I_M=[\hat\tau_M,\infty)$, and, for $j=1,\ldots,M-1$, $I_j=[\hat\tau_j, \hat\tau_{j+1})$.

Our next target is to construct an estimator of $\varphi$. Recall the notation $\varphi_0^+= \varphi\cdot E |X|$ in \eqref{equiv.mod.assB}. Recalling that $\varphi$ changes sign at each $\tau_j$, we can obtain a consistent estimator of either $\varphi_0^+$ or $-\varphi_0^+$ (we'll see below how to distinguish these two cases) by taking
$\hat\varphi_{\pm,0}^+(x)=\sum_{j=0}^M (-1)^j \,\hat\varphi^*_{j,\LL}(x)\cdot I(x\in I_j)\,,$
where, for each $j$,  $\hat\varphi^*_{j,\LL}$ denotes the local linear estimator of $\varphi^*$ constructed using only the $(U_i,|\tilde X_i|)$'s for which $U_i\in I_j$. Here we use a different local estimator in each $I_j$ because, under our assumptions, the first derivative of $\varphi^*=|\varphi|\cdot E|X|$ is discontinuous at the $\tau_j$'s. It can be shown using standard kernel smoothing arguments that in this case the bias near the $\tau_j$'s is reduced by using this piecewise approach.

Our next step is to extract from  $\hat\varphi_{\pm,0}^+$ an estimator of $\varphi_0^+$ (recall that $\hat\varphi_{\pm,0}^+$ is an estimator of $\varphi_0^+$ or  $-\varphi_0^+$, but we can't know of which one). To do this, recall that $E\{\varphi(U)\}=1$, which implies that $E\{\varphi_0^+(U)\}>0$. This fact motivates us to estimate $\varphi_0^+(x)$ by
$
\hat\varphi_0^+(x)=\hat\varphi_{\pm,0}^+(x) \, \big/ \,\textrm{sign} \big\{\sum_{i=1}^n\hat\varphi_{\pm,0}^+(U_i)\big\}.
$
Since $\varphi_0^+(x)=\varphi(x)\, E|X|$, once we have done this, to estimate $\varphi$ it remains to construct an estimator of $E|X|$. 

Noting that $E\{\varphi_0^+(U)\}=E\{\varphi(U)\}\, E|X|=E|X|$,  we can estimate $E|X|$ by
$\widehat{E|X|}=n^{-1}\sum_{i=1}^n \hat\varphi_0^+(U_i)=\big| n^{-1}\sum_{i=1}^n \hat\varphi_{\pm,0}^+(U_i) \big|$. Finally we estimate $\varphi(x)$ by
$\hat\varphi(x)=\hat\varphi_0^+(x) \, / \widehat{E|X|}\,.$
Then, we can predict the $X_i$'s  by taking  $\hat X_i = \{ \hat \varphi(U_i)\}^{-1}  \tilde X_i$. We can proceed similarly to construct predictors $\hat Y_i$ of the $Y_i$'s.
As in Section~\ref{two-stage.estiamtion.sec}, since, to obtain these predictors, we divide by $\hat\varphi(U_i)$ and $\hat\psi(U_i)$, when constructing our estimator of $m$ we cannot use the $(\hat X_i,\hat Y_i)$'s for which $|\hat\varphi(U_i)|$ or $|\hat\psi(U_i)|$ is too small. Therefore, to estimate $m$  we use the  estimators $\hat m_\NW(x;  \rho_1,\rho_2 )$   and  $\hat m_\LL(x;  \rho_1,\rho_2 )$ defined in Section~\ref{two-stage.estiamtion.sec}, but with the predictors $\hat X_i$ and $\hat Y_i$ constructed above.

\section{Theoretical properties} \label{theo.sec}
\setcounter{equation}{0}

We start by establishing theoretical properties of the estimators $\hat m_{\NW}$ and $\hat m_{\LL}$ from Section~\ref{one-step}.
While these estimators seem intuitively natural, because they are computed using variables  obtained through nonparametric prediction, checking whether they are consistent, and deriving detailed asymptotic properties, are quite difficult. Recently, Mammen et al.~(2012) gave a deep account of nonparametric estimators computed from nonparametrically generated covariates, but our estimators do not fall into the class of settings they consider, not least because in our case, not only the covariate $X$, but also the dependent variable $Y$, are nonparametrically generated, which makes the problem even more complex than theirs. In addition to the basic model assumptions introduced in the first paragraph of Section~\ref{sec:model}, we make the following regularity assumptions:

\begin{itemize}
\baselineskip=14pt
\item[(B1)] $ E|X| \neq 0$, $E|Y| \neq 0$ and 
$
\inf_{u\in I_U }  \varphi(u)>0 , \, \inf_{u\in I_U } \psi(u) >0. $
\\[-.5cm]
 
\item[(B2)] $0<   \inf_{u\in I_U}f_U(u) \leq \sup_{u\in I_U}f_U(u)  <\infty$;  $f_U$, $\varphi$ and $\psi$ are twice differentiable, and their second derivatives are uniformly continuous and bounded.\\[-.5cm]

\item[(B3)] (a) $f_X$ is continuous, $\sup_{x \in \br} f_X(x)<\infty$, and $E \{ \exp(c_1 |X|) \}<\infty$ for some constant $c_1>0$; (b) $m$ and $f_X$ are twice differentiable and their second derivatives are uniformly continuous and bounded; (c) $\sigma$ is continuous and bounded.\\[-.5cm]

\item[(B4)] $E(\varepsilon)=0$, $E(\varepsilon^2)=1$ and $E\{ \exp(c_2 |\varepsilon|) \} < \infty$ for some $c_2>0$.\\[-.5cm]

\item[(B5)] $K$ and $L$ are twice continuously differentiable, symmetric density functions, and are compactly supported on $[-1,1]$. Moreover,  $ \int_0^1 t^2 L(t)\, dt  >  2 \{ \int_0^1  t L(t)\, dt   \}^2$. \\[-.5cm]

\item[(B6)] The bandwidths $(h, g_1 , g_2  )=(h_n , g_{1n}, g_{2n} )$ are such that $h \asymp n^{-\alpha_0}$ and $g_1 \asymp n^{-\beta_1}$ and $g_2\asymp n^{-\beta_2}$ for some $0<\alpha_0, \beta_1, \beta_2 < 1/3$.

\end{itemize}

Condition (B1) is a relaxed version of assumption \eqref{positivity.ass} often assumed in the covariate-adjusted regression literature. See, for example, \c{S}ent\"urk and M\"uller~(2005a, 2006) and Cui et al.~(2009). Condition (B2) includes standard regularity and smoothness assumptions for the asymptotic results of kernel-type nonparametric regression estimation. In (B3), we relax the conventional boundedness condition on the covariates used by \c{S}ent\"urk and M\"uller~(2005a, 2006) and Mammen et al.~(2012), and assume instead that $X$ has a finite exponential moment (for example this is satisfied if the distribution of $X$ comes from the exponential family or is compactly supported). Condition (B4), which requires exponentially light tails of $\varepsilon$, is similar in spirit to Assumption~1. (iv) in Mammen et al.~(2012). Like them, we need this technical assumption to employ an argument based on empirical processes.
Condition (B5) is standard in the context of kernel regression, and is easy to satisfy since we can choose the kernels. Condition (B6) states the required range of magnitude of the bandwidths, and is easy to satisfy in practice.

The next two theorems  establish uniform consistency and asymptotic normality of our estimators $\hat m_{\NW}$ and $\hat m_{\LL}$ defined in Section~\ref{one-step}. Their proof can be found in Section~\ref{sec:proofs} and in Section~D in the supplementary file.

\begin{theorem}\label{thm.1}
Assume that \eqref{SM.cond1} and Conditions~\rm{(B1)--(B6)} hold and let $[a,b] \subseteq I_X \equiv \{x: f_X(x)>~0\}$.
\begin{itemize} 
\item[(i)] 
If $h \asymp g_1 \asymp g_2 \asymp(\log n)^{1/5} n^{-1/5}$, then $\hat{m}_{\NW}$ at \eqref{LL.m} satisfies
$
  \max_{x\in [a,b]}\allowbreak |\hat m_{\NW}(x) - m(x)  | = O_P \{ (\log n)^{2/5} n^{-2/5}  \} . 
$
\item[(ii)] 

If $ \beta_1 \geq 1/5$ and $0<\alpha_0 <  1/2 - \beta_1$, then for every $x\in [a,b]$,
\begin{align}
	 \hat{m}_{\NW}(x) - m(x) =\sqrt{ V(x)  }  \,N(x) + B_{0}(x) + \tilde{B}(x) + R_0(x),  \label{AN.mNW} 
\end{align}
where $N(x)   \xrightarrow {\mathscr{D}} N(0,1)$ as $n\rightarrow \infty$, 
$ V(x)  =  \{ nh f_X(x) \}^{-1} \sigma^2(x)\int K^2$, 
$B_{0}(x)  =    \{m''(x) + 2m'(x) f'_X(x)/f_X(x)  \} \mu_{K,2} \, h^2/2,$ 
$\tilde B(x)=\tilde B_\varphi(x)+\tilde B_\psi(x)$ with
$\tilde{B}_\varphi(x) =   x m'(x) \allowbreak E\{\varphi''(U)/\varphi(U)\}\mu_{L,2} \, g_1^2 /2$, 
$\tilde B_\psi(x)= -m(x) \allowbreak E\{\psi''(U)/\psi(U) \} \mu_{L,2} \, g_2^2 /2,$ 
and the remainder $R_0$ is such that $|R_0(x)| =  o_P\{ g_1^2+g_2^2+ h^2+ (nh)^{-1/2} \}$.

\end{itemize}
\end{theorem}

\begin{theorem} \label{thm.2}
Assume that \eqref{SM.cond1} and Conditions~\rm{(B1)--(B6)} hold and let $[a,b] \subseteq I_X$. 
\begin{itemize}
\item[(i)] 
If $h \asymp g_1 \asymp g_2 \asymp(\log n)^{1/5} n^{-1/5}$, then  $\hat{m}_{\LL}$ at \eqref{LL.m} satisfies
$
  \max_{x\in [a,b]} \allowbreak | \hat m_{\LL}(x) - m(x)  | =O_P \{ (\log n)^{2/5} n^{-2/5}  \}. 
$
\item[(ii)] If $\beta_1 \geq 1/5$ and $0<\alpha_0<1/2-\beta_1$, then for every $x\in [a,b]$,
\begin{align}
 \hat{m}_{\LL}(x) - m(x) =\sqrt{ V(x)  }  \,N(x) + B_1(x) + \tilde{B}(x) + R_1(x),  \label{AN.mLL}
\end{align}
where $N(x)   \xrightarrow {\mathscr{D}} N(0,1)$ as $n\rightarrow \infty$, $   B_1(x)  =  m''(x) \mu_{K,2} \, h^2/2 , 
$ $V$ and $\tilde{B}$ are as in part (ii) of Theorem~\ref{thm.1}, 
and $R_1$ is such that $|R_1(x)| =  o_P\{ g_1^2+g_2^2+ h^2+ (nh)^{-1/2} \}$.
\end{itemize}
\end{theorem}

We deduce from the theorems that, although they are constructed from distorted data, when computed with appropriate bandwidths, our estimators  $\hat m_{\NW}$ and $\hat m_{\LL}$ defined in Section~\ref{one-step} have the same uniform convergence rates as the standard estimators  in \eqref{standardLP} used when the $(X_i,Y_i)$'s are available. This  contrasts with the errors-in-variables models studied by Fan and Truong~(1993) and Delaigle et al.~(2009), where convergence rates  are significantly degraded by the measurement errors.
The conclusions arising from the asymptotic distribution of our estimators are also  interesting. Abusing terminology, we refer to  $V$ and $B_0+\tilde B$ (resp., $B_1+\tilde B$) as the asymptotic variance and bias and  of our estimator  $\hat m_{\NW}$ (resp., $\hat m_{\LL}$), and we call asymptotic mean squared error (AMSE) the sum of the asymptotic variance and  squared bias.  We use similar terminology for the  standard estimators of~$m$.

We learn from part (ii) of both theorems that, if we choose $g_1$ and $g_2$ of order $o(h)$, the asymptotic bias and variance of our estimators are identical to those of standard estimators, and there, as in the standard case,  it is optimal to take $h\asymp n^{-1/5}$, so that ${\rm AMSE}\asymp n^{-4/5}$. 
Perhaps more surprisingly, in cases where  $B_0$ (resp., $B_1$  for $\hat m_{\LL}$), $B_\varphi$ and $B_\psi$ do not all have the same sign, it is possible to choose $h$ and $g_1$ or $g_2$ an order of magnitude slightly larger than $n^{-1/5}$ such that the asymptotic bias $B_0+\tilde B$ (resp., $B_1+\tilde B$) vanishes and the  AMSE our estimator is of order $o(n^{-4/5})$, thus smaller than the AMSE of the standard estimator (similar results can be established for the integrated AMSE). However, while it is theoretically interesting, we were not able to  exploit this result in practice to make our estimator outperform the standard one, despite several attempts. In part this is because to benefit from this result we need  to choose the bandwidths in a very specialized way that requires estimating too many unknowns, and we found that the simpler bandwidths choice suggested in Section~\ref{sec:param} almost always worked better.

Next, we develop theoretical properties of our estimator defined in Section~\ref{two-stage.estiamtion.sec}.
We start by rewriting $ \hat{\mathcal{C}}_n(\rho_1,\rho_2)$ as $\hat{\mathcal{C}}_n(\rho_1,\rho_2)   = \{ 1\leq i\leq n: U_i \in \hat{\mathcal{L}}_n(\rho_1,\rho_2)  \}$, where $\hat{\mathcal{L}}_n(\rho_1,\rho_2) = \{ u\in I_U :  |\hat{\varphi}_{0, \LL}(u)|  \geq \rho_1, |\hat{\psi}_{0, \LL}(u)| \geq \rho_2  \}$. We can rewrite the estimator at \eqref{two-step.est.m} as
\beq 
	\hat m_\NW(x;  \rho_1,\rho_2 ) = \frac{\sn \hat Y_i  K_h(x-\hat X_i) I\{U_i \in \hat{\mathcal{L}}_n(\rho_1,\rho_2)\} }{\sn K_h(x-\hat X_i) I\{U_i \in \hat{\mathcal{L}}_n(\rho_1,\rho_2)\}} \,.
\eeq

To emphasize the main idea while avoiding repetitive arguments, here we present the theoretical result only  for this estimator, assuming that only $\varphi$ may have zeros, and therefore we take $\rho_2=0$ throughout this section. A straightforward adaptation of the arguments used to prove Theorem~\ref{two-stage.thm} below leads to similar results in the more general case where $\varphi$ has zeros and $\rho_2>0$, and for the local linear estimator $\hat m_\LL(x;  \rho_1,\rho_2 )$.

When $\rho_2=0$,   $\hat{\mathcal{L}}_n(\rho_1,\rho_2)$ depends only on $\rho_1$; to simplify notation we rewrite it as  $\hat{\mathcal{L}}_n(\rho_1)= \{ u\in I_U :  |\hat{\varphi}_{0, \LL}(u)|  \geq \rho_1  \}$. Likewise,  we rewrite $\hat m_\NW(x;\rho_1,0) $ as $\hat m_\NW(x;\rho_1)$.
Under certain regularity conditions on $\varphi$, the random set $\hat{\mathcal{L}}_n(\rho_1)$ is a consistent estimator of $	\mathcal{L}( \rho_1)  = \{  u \in I_U  : |\varphi_0(u)  |  \geq \rho_1\}$. Recalling that $\varphi_0(u)=E(X)  \, \varphi(u)$, this  suggests taking $\rho_1$ to be some value between $0$ and $M_0 \equiv  |E(X)| \max_{u\in I_U} | \varphi(u)|$. For $0\leq t\leq M_0$, let $\partial \mathcal{L}(t) = \{ u\in I_U : |\varphi_0(u)| = t \}$. We will need the following assumptions:
\begin{itemize}
\baselineskip=14pt
\item[(C1)] $E(X) , E(Y)  \neq 0$ and  
$
	\inf_{u\in I_U }    \psi(u) >0 .
$\\[-.5cm]

\item[(C2)]  $\varphi $ is such that the set $ \Theta =  \big\{ t\in (0,M_0) :   \partial \mathcal{L}(t)$ consists of finitely many points located in the interior of $I_U$ and $\min_{u \in  \partial \mathcal{L}(t) } |\varphi'(u)| >0  \big\} $ is non-empty.

\end{itemize}

The next theorem establishes uniform consistency and asymptotic normality  of $\hat m_\NW(x;\rho )$.
See Section~E in the supplementary file for its proof.

\begin{theorem} \label{two-stage.thm}
Assume that \eqref{SM.cond1}, Conditions~\rm{(B2)--(B5), (C1) and (C2)} hold and that $\rho \in  ( 0, M_0 )$ in \eqref{two-step.est.m} is such that $\rho  \in \Theta$. Let $[a,b]\subseteq I_X$.
\begin{itemize}
\item[(i)] 
If $ g_1\asymp g_2 \asymp h \asymp  (\log n)^{1/5}n^{-1/5}$, then $\hat m_\NW(x;\rho ) \equiv \hat m_\NW(x;\rho , 0 )$ at \eqref{two-step.est.m} satisfies $\max_{x\in [a,b]}  | \hat m_\NW(x;\rho )  -m(x)  | = O_P\{ (\log n)^{2/5}  n^{-2/5} \} .$

\item[(ii)] If $\beta_1 \geq 1/5$ and $0<\alpha_0<1/2-\beta_1$, then for every $x\in [a,b]$,
\begin{equation}
\hat m_\NW(x;\rho )  - m(x) =\sqrt{ V(x; \rho)  }  \,N(x) + B_0(x ) + \tilde{B}(x;\rho) + R_2(x;\rho),  \label{2nd.AN}
\end{equation}
where $N(x)   \xrightarrow {\mathscr{D}} N(0,1)$ as $n\rightarrow \infty$, $V(x;\rho) = V(x)/P\{ U \in \mathcal{L}(\rho) \}$, $V, B_0$ are as in part (ii) of Theorem~\ref{thm.1},
$\tilde{B}(x;\rho)   =   x m'(x)  E[\varphi''(U )\allowbreak I\{  U \in \mathcal{L}(\rho)\}/\varphi(U)   ]\mu_{L,2} \, g_1^2 /2 
  -  m(x) E[ \psi''(U)  I\{ U \in \mathcal{L}(\rho) \} /\psi(U) ] \mu_{L,2} \,\allowbreak g_2^2 /2 , 
$ and $R_2$ is such that $|R_2(x; \rho ) | =  o_P\{ g_1^2+g_2^2+ h^2+ (nh)^{-1/2} \}$.
\end{itemize}
\end{theorem}

We deduce from the theorem that our estimator defined in Section~\ref{two-stage.estiamtion.sec} has the same uniform convergence rate as the standard Nadaraya-Watson estimator  in \eqref{standardLP}, used when the data $(X_i,Y_i)$ are available. Moreover, as long as we choose $g_1$ and $g_2$ of order $o(h)$, the asymptotic ``bias'' and ``variance'' of our estimator from Section~\ref{two-stage.estiamtion.sec} are equal to those of the standard Nadaraya-Watson estimator, where $i\in  \{ 1\leq j\leq n: U_j \in{\mathcal{L}}_n(\rho)\}$.  As we already indicated below Theorems~\ref{thm.1} and \ref{thm.2}, in theory in some cases it is possible to choose the bandwidths in such a way that the AMSE of our estimator tends to zero faster than that of the standard estimator, but it seems very hard to find a way to exploit this in practice. Similar results can be established for the local linear estimator $\hat m_\LL(x;  \rho_1,\rho_2 )$.

Establishing theoretical results for the more general procedure described in Section~\ref{sec:fourth} is particularly challenging. Recall that this method combines a change point detection algorithm and the ridge-parameter based method introduced in Section~\ref{two-stage.estiamtion.sec}. The complex nature of this approach implies that deriving its theoretical properties rigorously requires long and tedious arguments. Since our paper is already very long, and even the proofs for our simpler methods are fairly tedious, we leave such rigorous derivations for future work. However, our preliminary calculations already indicate that the procedure from Section~\ref{sec:fourth} should have asymptotic properties similar to those described  in Theorem~\ref{two-stage.thm}. In particular, these calculations indicate that estimating the $\tau_j$'s and the sign of $\varphi$ and/or $\psi$ has no first order asymptotic effect on the properties of our estimators of $m$.

\section{Numerical results}\label{sec:implem}
\subsection{Which method to use}\label{sec:method}
The approach in Section~\ref{sec:fourth} can be applied in essentially all cases, but since the methods from Sections~\ref{one-step} and \ref{two-stage.estiamtion.sec} are simpler, the user might prefer to use these if all parts of \eqref{positivity.ass} hold. While \eqref{positivity.ass} can be verified by standard tests of hypothesis applied to the observed data (see Remark~\ref{Htest} below), when these conditions are needed, it is because the techniques employed involve dividing by estimators of $\psi$, $\varphi$, $EX$ or $EY$. Therefore, in practice, to avoid numerical issues, we suggest using the method from Section~\ref{two-stage.estiamtion.sec}, and to use instead the method from Section~\ref{sec:fourth} if the absolute values of estimators of $EX$ or $EY$ are small, the extent of which depends on the magnitude of other quantities involved and the precision of the software employed. This is generally rather easy to determine by examining the data, but if unsure the user can just apply the method of Section~\ref{sec:fourth}, which  is valid in the most general case.

We note too that one does not necessarily need to predict the $X_i$'s and the $Y_i$'s with the same method. For example, if one is confident that $EX$ is far from zero, but is not sure about $EY$, then  the $X_i$'s could be predicted using the approach from Section~\ref{two-stage.estiamtion.sec}, and the predictors of the $Y_i$'s could be obtained from the approach suggested in Section~\ref{sec:fourth}.

\begin{remark}\label{Htest}
{\rm The assumption at \eqref{positivity.ass} can be tested in several ways. For example, since $E\tilde X=EX$, we can first test the sign of $EX$ by a standard test of hypothesis for the mean applied to the data $\tilde X_1,\ldots,\tilde X_n$, and then test the sign of the function $\varphi_0=\varphi\cdot EX$ at \eqref{equiv.mod.ass}, using for example tests such as those in D\"umbgen and Spokoiny~(2001), Chetverikov~(2012) and Lee et al.~(2013), applied to the observed data. 
}
\end{remark}

\subsection{Details of implementation}\label{sec:param}
As in the case where the $(X_i,Y_i)$'s are available, in practice we recommend using the local linear versions of our estimators, and in this section we suggest ways of choosing the parameters required to compute them.  Similar ideas can be used for the Nadaraya-Watson estimators. 
We know from Section~\ref{theo.sec} that, while we have to choose $h$ with care, we have more flexibility for the bandwidths $g_1$ and $g_2$, which can take a large range of values. If we take $h$ to be of the standard size for nonparametric regression, and $g_1=o(h)$ and $g_2=o(h)$, then our estimators have the same first order asymptotic properties as the  estimators at~\eqref{standardLP}.

Motivated by this, for the estimators in Section~\ref{one-step}, we take $g_1=n^{-0.1} g_{1,\PI}$, $g_2=n^{-0.1} g_{2,\PI}$ and $h=h_{\PI}$, where the subscript $\PI$ means that we use a standard plug-in bandwidth for local linear estimators (Ruppert et al.,~1995) constructed based on, respectively, the data $(U_i,|\tilde X_i|)$, $(U_i,|\tilde Y_i|)$ and $(\hat X_i,\hat Y_i)$.
For the estimators in Sections~\ref{two-stage.estiamtion.sec} and \ref{sec:fourth},  we take $g_1=n^{-1/10} g_{1,\PI}$ and $g_2=n^{-1/10} g_{2,\PI}$, where  $g_{1,\PI}$ and $g_{2,\PI}$ denote standard plug-in bandwidths  for local linear estimators constructed based on, respectively, the data $(U_i,\tilde X_i)$ and  $(U_i,\tilde Y_i)$. Then, in Section~\ref{two-stage.estiamtion.sec}, we choose $\rho_1=\max(0.1,\rho_1^*)$ and $\rho_2=\max(0.1,\rho_2^*)$, where $\rho_1^*$ (resp., $\rho_2^*$) denotes the square root of an estimator of the asymptotic ``mean squared error'' of $\hat\varphi_\LL$ (resp., $\hat\psi_\LL$), integrated over the set of $x$-values where $|\hat\varphi_\LL(x)|$ (resp., $|\hat\psi_\LL(x)|$) take its smallest values; see Appendix~A in the supplementary file for details. 
We do the same for the method from Section~\ref{sec:fourth}, except that we use the estimators $\hat\varphi$ and $\hat\psi$ of $\varphi$ and $\psi$ derived there.
Finally, we take $h=h_{\PI}$, a standard plug-in bandwidth for local linear estimators computed from the data $(\hat X_i,\hat Y_i)$, $i\in \hat{\mathcal{C}}_n(\rho_1,\rho_2)$.

The estimators from Section~\ref{sec:fourth} also require to estimate the zeros $\tau_1,\ldots,\tau_M$ at which $\varphi$ changes sign, and the same is required for $\psi$ if the method in that section is used to compute predictors of the $Y_i$'s. We proceed as follows. First, since the $\tau_j$'s all correspond to a local minimum of $\varphi^*$, we find all the points at which $\hat\varphi^*_\LL$ has local minima. Then, among those points we keep only those which are close to the discontinuity points of the derivative $\varphi^*$ detected by the method of Gijbels and Goderniaux~(2005). Here we define ``close'' by less than $2h$ away, where $h$ is the bandwidth in Section~2.2.1 of Gijbels and Goderniaux~(2005).
Finally, to slightly improve numerical performance, we implement Remark~\ref{rem:removepoints} and remove the data corresponding to the 5\% smallest $\hat f_U(U_i)$'s.

\subsection{Simulations} \label{sim.sec}
We applied our methods to a variety of simulated examples, ranging from the simplest ones in which  $\psi>0 $ and $\varphi> 0$, where we can use the method from Section~\ref{one-step}, to more complex ones in which $EX=0$ and both $\psi$ and $\varphi$  oscillate between positive and  negative values, where  we need to use the  sophisticated approach suggested in Section~\ref{sec:fourth}.  

We generated data $(\tilde X_i,\tilde Y_i,U_i)$, $i=1,\ldots,n$, from model \eqref{basic.model} for $n=100$, $200$, $500$ and $1000$, and considered various combinations of $m$, $\varphi$, $\psi$ and $\sigma$, and  various distributions of $X_i$ and $U_i$. We took $\varepsilon_i\sim N(0,1)$, and considered shifted versions of three regression curves $m$, denoted by $m_1$, $m_2$ and $m_3$ and defined as
$m_1(x)=\sin\{\pi(x-1)/2\} / \big[{\{1+2(x-1)^2\} \{{\rm sign}(x-1)+1\}}\big]$,
$m_2(x)=x^2 \phi_{0,1}(x)$,
and $m_3(x)=2x+\phi_{0.5,0.1}(x)$,
where $\phi_{\mu,\theta}$ denotes the density of a $N(\mu,\theta^2)$.  In all cases below, the generic constant $\con$ was chosen so that $E\{\varphi(U)\}=E\{\psi(U)\}=1$.

First, we considered models where the local linear estimators from Sections~\ref{one-step} to \ref{sec:fourth} could all be applied:
(i.a) $m=m_1$, $X_i\sim N(1,1.5^2)$, $\sigma(x)=0.3$;
(ii.a) $m=m_2$, $X_i\sim N(1,1.5^2)$, $\sigma(x)=0.05$;
(iii.a) $m=m_3$, $X_i\sim N(0.5,0.75^2)$, $\sigma(x)=0.55$;
(i.b) $m(\cdot)=m_1(\cdot-1)+2$, $X_i\sim N(2,1.5^2)$, $\sigma(x)=0.3$;
(ii.b) $m(\cdot)=m_2(\cdot-1)$, $X_i\sim N(2,1.5^2)$, $\sigma(x)=0.05$;
(iii.b) $m(\cdot)=m_3(\cdot-1)$, $X_i\sim N(1.5,0.75^2)$, $\sigma(x)=0.55$.
Each time we took  $U_i\sim\beta(2,5)$, $\psi(u)=\con \,(u+0.5)^2$ and  $\varphi(u)=\con \,(u+0.25)^2$.

\begin{figure}[htbp]
  \centering
  \includegraphics[width=4.9in]{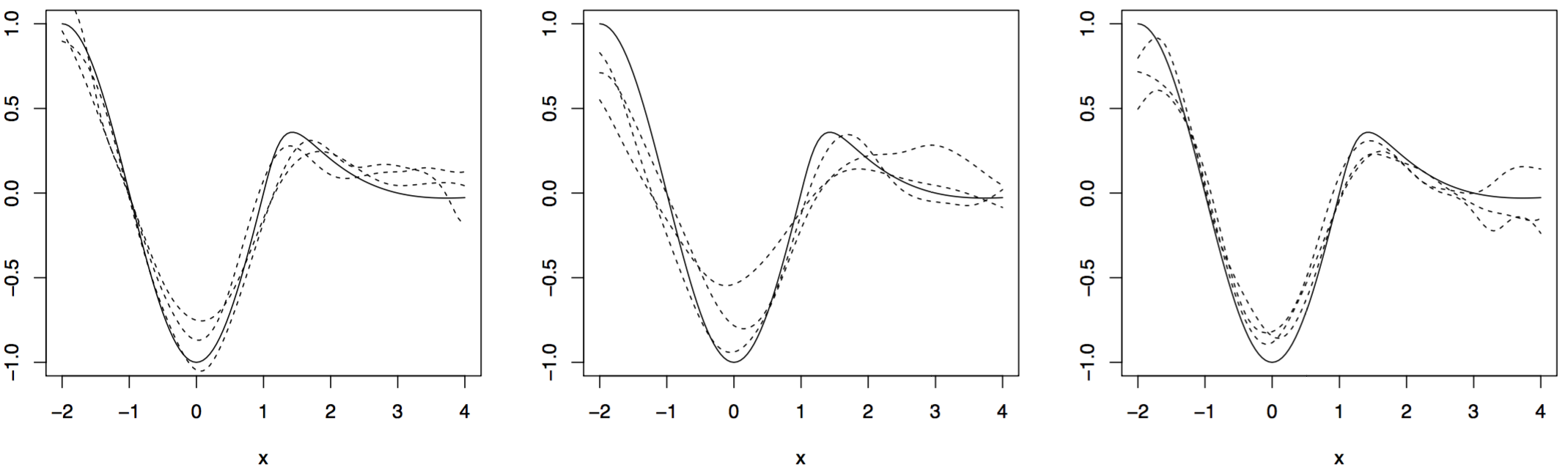}
 \caption{\small\baselineskip=12pt $\hat m_\LL$ from Section~\ref{one-step} (left), $\hat m_\LL(\cdot ; \rho_1,\rho_2)$ from Section~\ref{two-stage.estiamtion.sec} (center) and estimator  $\hat m_\LL(\cdot;\rho_1,\rho_2)$ from Section~\ref{sec:fourth} (right)  for three samples coming from model (i.a) with $n=200$, and corresponding to the  1st, 2nd and 3rd quartiles of the ISEs. The continuous line depicts the true~$m$.}\label{F:model2}
\end{figure}


Next, we considered models (i.c)--(iii.c) and (i.d)--(iii.d), where we took $m$, $X_i$ and $\sigma$ as in models (i.a)--(iii.a) and (i.b)--(iii.b), respectively, but took $U_i\sim \beta(3,5)$ and $\varphi(\cdot)=\psi(\cdot)=\con\,m_1(5\cdot-2)$. Here $\varphi$ and $\psi$ have zeros and change signs, so that the method from Section~\ref{one-step} cannot be applied.  
Finally, in our last models,  $\varphi$ and $\psi$ change signs and have several zeros and  $E(X_i)=0$, so that we can apply only the method from Section~\ref{sec:fourth}:
(iv.a) $m(\cdot)=m_1(\cdot+1)$, $X_i\sim N(0,1.5^2)$, $\sigma(x)=0.3$;
(v.a) $m(\cdot)=m_2(\cdot+1)$, $X_i\sim N(0,1.5^2)$, $\sigma(x)=0.05$;
(vi.a) $m(\cdot)=m_3(\cdot+0.5)$, $X_i\sim N(0,0.75^2)$, $\sigma(x)=0.55$;
(iv.b) $m(\cdot)=m_1(\cdot)$, $X_i\sim\{\chi^2(4)-4\}/2$, $\sigma(x)=0.3$;
(v.b) $m(\cdot)=m_2(\cdot)$, $X_i\sim\{\chi^2(4)-4\}/2$, $\sigma(x)=0.05$;
(vi.b) $m(\cdot)=m_3(\cdot)$, $X_i\sim\{\chi^2(4)-4\}/3.5$, $\sigma(x)=0.55$;
Each time we took $U_i\sim \beta(3,5)$ and $\varphi(\cdot)=\psi(\cdot)=\con\,m_1(5\cdot-2)$.
Heteroscedastic versions of these models gave similar results; see Appendix~B in the supplementary file.

\begin{figure}[htbp]
  \centering
  \includegraphics[width=4.9in]{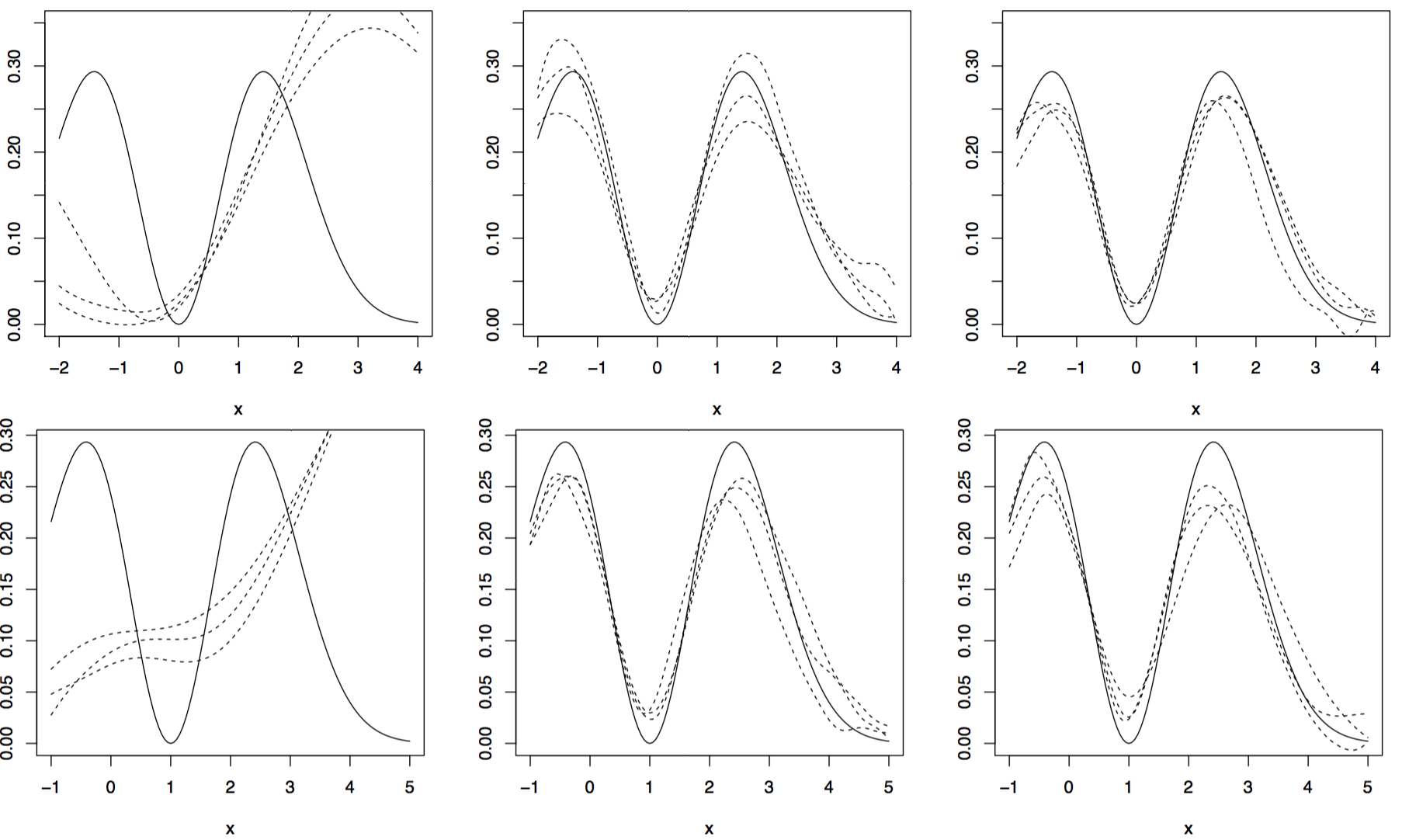}
 \caption{\small\baselineskip=12pt naive estimator $\hat m_{\LL,{\rm naive}}$ (left), $\hat m_\LL(\cdot;\rho_1,\rho_2)$ from Section~\ref{two-stage.estiamtion.sec} (center), and estimator $\hat m_\LL(\cdot;\rho_1,\rho_2)$ from Section~\ref{sec:fourth} (right) for three samples coming from model (ii.c) (top) and model (ii.d) (bottom)  with $n=500$, and corresponding to the 1st, 2nd and 3rd quartiles of the ISEs.  The continuous line depicts the true~$m$.}\label{F:model1}
\end{figure}


We compared each of our estimators with the ideal estimator $\tilde  m_{\LL}$ at \eqref{standardLP} computed from the $(X_i,Y_i)$'s, which are not available in real data applications but are available when we simulate data, and with the inconsistent naive estimator $\hat m_{\LL,{\rm naive}}$, which is the standard local linear estimator computed from the contaminated  $(\tilde X_i,\tilde Y_i)$'s.
For each $n$ and each model, we generated 1000 samples and constructed each estimator for each sample. Let $\hat m$ denote any one of the estimators considered below. To summarise the performance of $\hat m$, we computed, for each sample, the integrated squared error $\ISE=\int_a^b \{\hat m(x)-m(x)\}^2\,dx$, where, in each case, $a$ and $b$ were the quantiles $0.025$ and $0.975$ of the distribution of $X$.

In Tables 1 to 4 in Appendix~B in the supplementary file, for each method we report the first, second and third quartiles of the resulting 1000 ISEs. 
See Appendix~B for a detailed discussion of the simulation results. In summary, we found that, as expected, when $\varphi$, $\psi$, $EX$ and $EY$ were different from zero, but $EX$ and/or $EY$ were relatively close to zero, the estimator that worked best was the one from Section~\ref{one-step}, but the most complex estimator from Section~\ref{sec:fourth} worked well. When $EX$ and $EY$ were far from zero, all three estimators worked well, with the simplest one from Section~\ref{one-step} giving the best results and the one from Section~\ref{sec:fourth} working the worst. When  $\varphi$ and/or $\psi$ had zeros, the estimator from Section~\ref{one-step} could not be applied, and when $EX$ and $EY$ were close to zero, the best results were obtained with the estimator  from Section~\ref{sec:fourth}, whereas when $EX$ and $EY$ were far from zero, the estimator from Section~\ref{two-stage.estiamtion.sec} worked best. Finally, we found that our approach also performed well when the errors were heteroscedastic.

In all cases, our estimators  performed considerably better than the naive estimator, but were of course outperformed by the oracle estimator. As expected, the performance of our estimators improved as sample size increased.
In all our simulation settings, the estimator from Section~\ref{sec:fourth} gave reasonable results. However, if $\varphi$ and $\psi$ were far from zero, we got better results by using the simplest estimator from Section~\ref{one-step}, and if  $EX$ and $EY$ were far from zero, we got better results using the estimator from Section~\ref{two-stage.estiamtion.sec}.

\begin{figure}[htbp]
  \centering
  \includegraphics[width=4.8in]{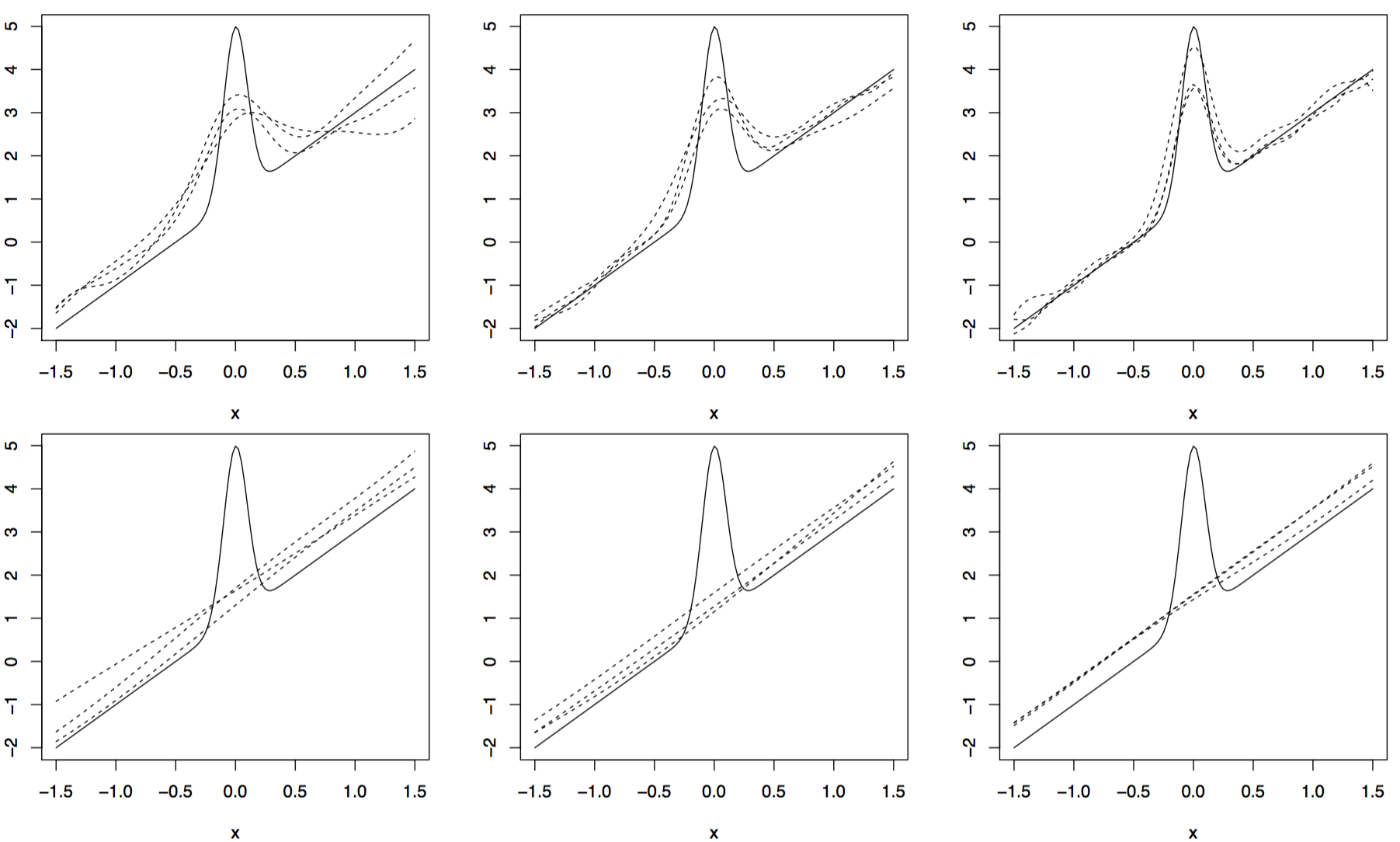}
 \caption{\small\baselineskip=12pt $\hat m_\LL(\cdot ; \rho_1,\rho_2)$ from Section~\ref{sec:fourth} (first row) and naive estimator $\hat m_{\LL,{\rm naive}}$ (second row) for three samples coming from model (vi.a)  with $n=100$ (left), $n=200$ (centre) and $n=500$ (right), and corresponding to the 1st, 2nd and 3rd quartiles of the ISEs. The continuous line depicts the true~$m$.}\label{F:model6}
\end{figure}


To illustrate these results graphically, we present a few figures that are representative of the conclusions of our simulations. For each estimator $\hat m$ presented in the figures, we show the three estimated curves corresponding to the first three quartiles of the 1000 ISEs defined above.
In Figure~\ref{F:model2}, using example (i.a), we illustrate the fact that, when all three methods can be applied, they often give similar results. Figure~\ref{F:model1} shows estimated curves for examples (ii.c) and (ii.d). We can see that, in case (ii.c), where $EX$ is close to zero, the estimator  $\hat m_\LL(\cdot ; \rho_1,\rho_2)$ from Section~\ref{sec:fourth} worked better than the one from Section~\ref{two-stage.estiamtion.sec}, but that the reverse is true in case (ii.d), where $EX$ and $EY$ are both far from zero. In that figure, we also depict the naive estimator $\hat m_{\LL,{\rm naive}}$, which performed very poorly. Finally, in Figure~\ref{F:model6}, we use example (vi.a) to demonstrate the improvement that our estimator $\hat m_\LL(\cdot ; \rho_1,\rho_2)$ from Section~\ref{sec:fourth} benefits from as the sample size $n$ increases. Here too, the naive estimator performed very poorly, even for $n$ large.

\subsection{Real data illustrations}

We applied our new method to the Boston house-price dataset described in Harrison and Rubinfeld~(1978), available at \verb+http://lib.stat.cmu.edu/datasets+, and which contains information about houses and their owners at 506 locations around Boston. 
As in \c{S}ent\"urk and M\"uller~(2005b), we are interested in the relationship between the median price (in USD 1000's) of houses, $\tilde Y$, and per capita crime rate by town, $\tilde X$, with the confounding effect of the proportion of population of lower educational status, $U$, removed. \c{S}ent\"urk and M\"uller's~(2005b), whose interest was in the correlation between $\tilde X$ and $\tilde Y$, concluded that this correlation alters dramatically after adjusting for the confounding effect of lower educational status. On the left panel of Figure~\ref{F:Realdata}, we depict the covariate-adjusted regression curve obtained using the local linear estimator $\hat m_\LL(\cdot;\rho_1,\rho_2)$ from Section~\ref{two-stage.estiamtion.sec}, the estimator $\hat m_\LL(\cdot;\rho_1,\rho_2)$ from Section~\ref{sec:fourth}, and the naive regression estimator $\hat m_{\LL,{\rm naive}}$ obtained by regressing $\tilde Y$ on $\tilde X$ after removing a few outliers. In this example, the estimator from Section~\ref{one-step} was identical to the one from Section~\ref{two-stage.estiamtion.sec}. 
We can see that $\hat m_{\LL,{\rm naive}}$ indicates a pronounced relationship between house price and crime rate (as crime rate increases, house price decreases), but once we adjust for the effect of lower educational status, the regression curve obtained by both versions of our estimator is almost flat, indicating a weak  relationship between the adjusted $X$ and $Y$.

\begin{figure}[htbp]
  \centering
  \includegraphics[width=3.6in]{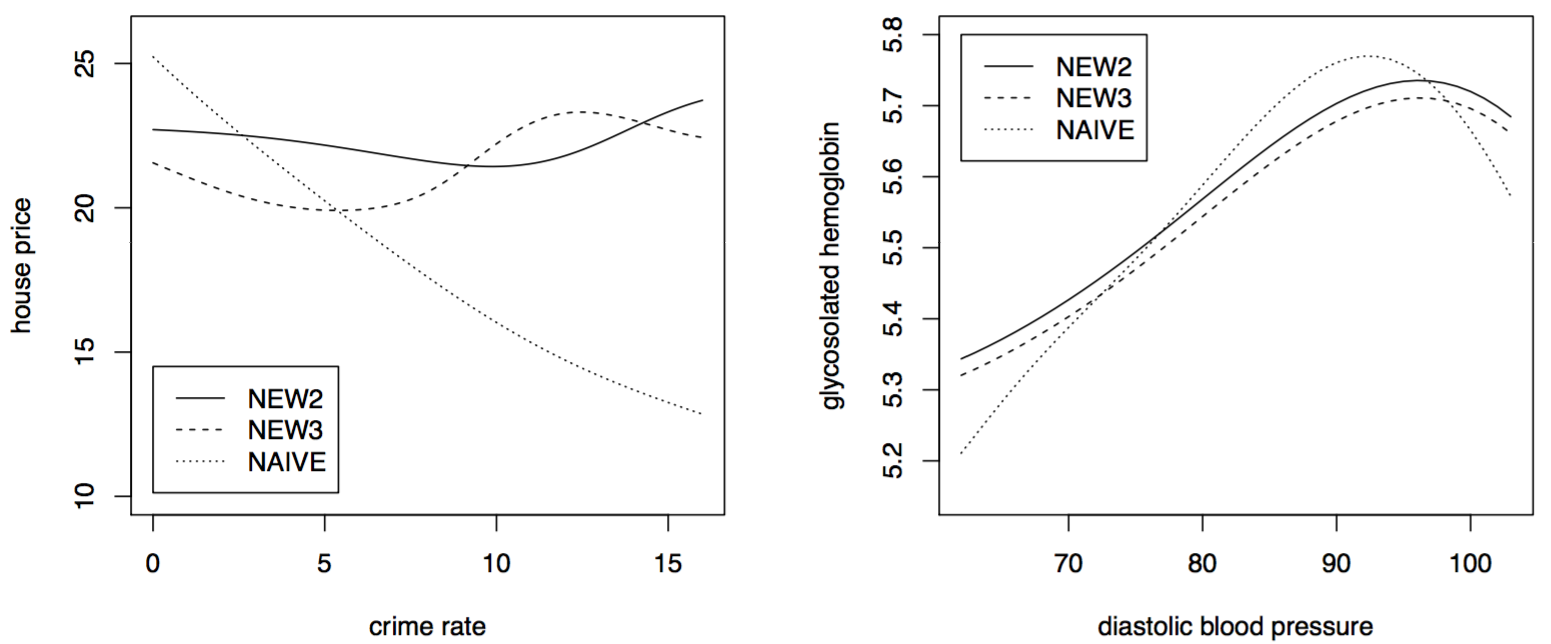}
 \caption{\small\baselineskip=12pt $\hat m_\LL(\cdot ; \rho_1,\rho_2)$ from Sections~\ref{two-stage.estiamtion.sec} (NEW2) and~\ref{sec:fourth} (NEW3), and naive estimator $\hat m_{\LL,{\rm naive}}$ (NAIVE) for the Boston data (left) and the diabetes data (right).}\label{F:Realdata}
\end{figure}


Next, we applied our procedure to the diabetes dataset used by Schorling et al.~(1997) and Willems et al.~(1997), available at \verb+http://biostat.mc.van+ \verb+derbilt.edu/DataSets+,  which represents a subset of 403 individuals taken from a larger cohort of 1046 subjects who participated in a study for African Americans about obesity, diabetes and related factors in central Virginia. As in \c{S}ent\"urk and Nguyen~(2006), our goal was to examine the relationship between glycosolated hemoglobin level $\tilde Y$, a biomarker for diabetes, and diastolic blood pressure $\tilde X$, adjusting for the effect of body mass index, $U$, which was found to be a confounder for both variables. As in \c{S}ent\"urk and Nguyen~(2006), we removed a few outliers before our analysis. As in the previous example, $\widehat{ E X}$ and $\widehat{ E Y}$ were far from zero, so that we used the estimator $\hat m_\LL(\cdot;\rho_1,\rho_2)$ from Section~\ref{two-stage.estiamtion.sec}, which we compared with the naive estimator $\hat m_{\LL,{\rm naive}}$. Here too, the estimator from Section~\ref{one-step} was identical to the one from Section~\ref{two-stage.estiamtion.sec}. We also computed the estimator from  $\hat m_\LL(\cdot;\rho_1,\rho_2)$ from Section~\ref{sec:fourth}. These estimators, depicted on the right panel of Figure~\ref{F:Realdata}, show that after adjusting for body mass index, the relationship between  glycosolated hemoglobin level and diastolic blood pressure is noticeably less pronounced. We should highlight that, in this example, the data were rather sparse for diastolic blood pressure greater than $100$, and the few patients for which $\tilde X$ was greater than $100$ had a rather low value of $\tilde Y$, whence the decreasing shape on the right hand side of the graph, which may just be an artifact of the sparseness of the data in that area.

Another interesting application of our method is to the baseline data collected from studies A and B of the Modification of Diet in Renal Disease Study (Levey et al., 1994). The nonlinear relationship between the baseline unadjusted glomerular filtration rate (GFR) and serum creatinine (SCr) is of particular interest. Taking body surface area (BSA) as the confounder, Cui et al. (2009) used a parametric nonlinear model of the form $m(x) = \beta_1 \exp(-\beta_2 - \beta_3 x^2) + \beta_4$ to study the relationship between GFR and SCr after correcting for the distorting effect of BSA. Because this dataset is not publicly available, we shall not compare the proposed nonparametric method with that of Cui et al. (2009) in this paper.

\section{Generalizations to the multivariate case}\label{sec:multivariate}
Our approach can be generalized to the $d$-variate case, $d\geq 1$, where we observe data distributed like a vector $(U,\tilde {\bf X}\T,\tilde Y)$, with $\tilde{\mathbf{X}}\in \br^d$ a distorted version of $\mathbf{X}\in \br^d$. 
Reflecting the fact that the components of $\tilde{\mathbf{X}}$ may not all be distorted, we write $d=d_1+d_2$, with $d_1\geq 0$ and $d_2\geq 1$, and let $\mathbf{X}=( \mathbf{X}_{1}\T,  \mathbf{X}_{2}\T   )\T$ and $\tilde{\mathbf{X}}=( \mathbf{X}_{1}\T,  \tilde{\mathbf{X}}_2\T   )\T$, where $\mathbf{X}_{1}= (X_{1},\ldots, X_{d_1} )\T$ and $\tilde{\mathbf{X}}_{2}=(\tilde X_{d_1+1} , \ldots, \tilde X_{d})\T$ is a distorted version of 
$\mathbf{X}_{2}=(X_{d_1+1} , \ldots, X_{d})\T$, and where we use the convention that $\mathbf{X}=\mathbf{X}_{2}$  if $d_1=0$. In this notation, the data  $\{(U_i , \tilde{Y}_i,  \mathbf{X}_{1i}\T , \tilde{\mathbf{X}}_{2i}\T  )\}_{i=1}^n$ we observe are generated by the model
\be
	\begin{cases}
			 Y = m(\mathbf{X})  + \varepsilon\, \sigma({\bf X}),   \\
	 	  \tilde{Y} = \psi(U) \, Y , \ \ \tilde X_{d_1+r}  = \varphi_r(U) \, X_{d_1+r} , \ \   r=1,\ldots, d_2 , 
	\end{cases} 
  \label{multivariate.model}
\ee
where $m({\bf x})=E(Y|{\bf X}={\bf x})$ is a  curve we wish to estimate, 
the random variables $\mathbf{X}$, $U$ and $\varepsilon$ are mutually independent, $E(\varepsilon)=0$ and $\var(\varepsilon) = 1$. As in \eqref{SM.cond1}, we assume that
$
	E \{ \psi(U) \} = 1$, $E \{ \varphi_r(U) \} = 1$, for $ r=1,\ldots, d_2.
$

The procedures from Section~\ref{one-step} to \ref{sec:fourth} can each be generalized to the multivariate setting, but for space constraint here we show only how to generalize the approach from Section~\ref{one-step}. The same ideas can be applied for the methods from Sections~\ref{two-stage.estiamtion.sec} and \ref{sec:fourth}.
To construct a nonparametric version of the estimator from Section~\ref{one-step}, we first  construct predictors $\hat{Y}_i$ and $\hat{X}_{i , d_1+1}, \ldots , \hat{X}_{id}$ as in equation \eqref{generated.response.predictor}, and let 
$$
	\hat{\mathbf{X}}_i = (  X_{i1}, \ldots, X_{id_1} ,  \hat{X}_{i , d_1+1}, \ldots , \hat{X}_{id})\T.
$$ 
Next, we use a standard multivariate local linear regression estimator applied to the data $(\hat{\mathbf{X}}_i\T , \hat Y_i)$. That is,  we define (see Fan and Gijbels,~1996) $\hat{m}_{\LL}(\mathbf{x}) = \hat{\alpha}_0$, where
$
	(\hat{\alpha}_0, \hat{\alpha}_1  )  = \argmin_{\alpha_0 \in \br, \alpha_1 \in  \br^{d}}    \sn \big\{ \hat{Y}_i - \alpha_0 -  \alpha_1  \T (\hat{\mathbf{X}}_{ i} - \mathbf{x} ) \big\}^2 \, \mathbf{K}_{\mathbf{h}}(\hat{\mathbf{X}}_i - \mathbf{x}) ,
$
with $\mathbf{K}_{\mathbf{h}}(\mathbf{x}) = \prod_{r=1}^d h_r^{-1} K(x_r/h_r)$  a $d$-dimensional product kernel, $K$ a univariate kernel, and $\mathbf{h} = (h_1, \ldots, h_d)\T$ a vector of bandwidths.

It is well known that fully nonparametric estimators suffer from the curse of dimensionality, which means that as $d$ increases, such estimators can only work reasonably well if the sample size is very large.  
To overcome this problem, a common approach is to restrict the regression model so that only univariate curves have to be fitted. A popular example is the additive model (Hastie and Tibshirani,~1990), which assumes that
$m(\mathbf{X})=m_0 + \sum_{j=1}^d m_j(X_j)$. In our context, the additive covariate-adjusted regression model can be written as
\begin{align}
\begin{cases} 	 Y = m_0 + \sum_{j=1}^d m_j(X_j)  + \varepsilon\, \sigma({\bf X}),  \\
  \tilde{Y} = \psi(U) \, Y , \ \  \tilde{X}_{d_1+r} = \varphi_r(U) \, X_{d_1+r }, \ \ r=1,\ldots, d_2 , 
  \end{cases}  \label{additive.model}
\end{align}
where $m_1,\ldots, m_d$ are unknown univariate functions satisfying $E \{ m_j(X_j) \}=0$ for $j=1,\ldots, d$ and $m_0$ is an unknown parameter. 

In the standard setting where the $(\mathbf{X}_i\T ,Y_i)$'s are directly observed, there are several ways to fit the additive model; see Horowitz~(2014) for an overview of estimation and inference for nonparametric additive models. The simplest approach is to adapt to our setting the iterative backfitting algorithm of Buja et al.~(1989), as follows. First, let $\hat{m}_0 = n^{-1} \sn \tilde{Y}_i$ and $\hat{m}_j \equiv 0$ for $j=1,\ldots, d$.
For $j=1,\ldots,d$, update $\hat m_j$ by taking it equal to a local linear regression estimator using the data $\{ (  \hat{X}_{ij}, \hat{Y}_i  -\hat{m}_0 - \sum_{k\neq j} \hat{m}_j(\hat{X}_{ik} )  ) \}_{i=1}^n $. 
Iterate  until the estimates $\hat{m}_j$ stabilize. (Here $\hat X_{ij}=X_{ij}$ if $j\leq d_1$.)

Alternatively, instead of taking $\hat m_j=0$ as initial estimators,  we could start with a linear approximation of the model in \eqref{additive.model}. See Appendix~C in the supplementary file for details. 
We could also apply similar transformations to other existing methods for fitting additive models, such as  the approach suggested by Horowitz and Mammen~(2004). The main theoretical challenge is a delicate analysis on how the presence of generated response and predictors affects the first order asymptotic properties of the final estimators. However, deriving such results requires much more work than can possibly done in this paper, and so we leave this problem for future research. The method proposed in this section can be applied to creatinine data, which was analyzed by \c{S}ent\"urk and M\"uller~(2006). In this study, serum creatinine level is taken as the response and the two predictors include cholesterol level and serum albumin level. The confounder variable $U$ is taken to be body mass index defined as weight/height$^2$. The readers can find more details about this dataset in \c{S}ent\"urk and M\"uller~(2006).

\section{Proof of Theorem~\ref{thm.1}}
\label{sec:proofs}

\noi 
We start by introducing basic notations. For a kernel function $K$, we write $\mu_{K,\ell}= \int u^{\ell} \,  K(u) \, du$ for non-negative integers $\ell$. For any set $S$, we denote its complement by $S^{{\rm c}}$ and its cardinality by $\# S$. Throughout, we let $\con$ denote a finite positive constant independent of $n$, which may take different values at each occurrence. We also use the following notation: $\mu_0 =E (X)$, $m_0=E (Y)$, $\mu_0^+ = E|X|$, $m_0^+ = E|Y|$ and
\begin{align}
	 \varphi_0   = \mu_0\, \varphi , \quad \psi_0  = m_0 \, \psi , \quad  \varphi^+_0   = \mu_0^+ \, \varphi , \quad \psi_0^+  = m_0^+ \, \psi .   \label{equiv.mod}
\end{align}

We proceed with the proof of Theorem~\ref{thm.1}. For $u\in I_U=[0,1]$, write
\be 
	w_0(u) \equiv 1, \ \hat w_X(u) =   \hat \mu_0^+ \,     \varphi(u) /\hat \varphi_{0,\LL}^+(u) \ , \ \hat w_Y(u)= \hat m_0^+ \,     \psi(u) / \hat \psi_{0,\LL}^+(u),	\label{def.r}
\ee
where $\hat{\mu}_0^+ = \widehat{E|X|} = n^{-1} \sn |\tilde{X}_i|$, $\hat{m}_0^+ = \widehat{E|Y|} = n^{-1} \sn |\tilde{Y}_i|$
and $\hat  \varphi_{0,\LL}^+$ and $ \hat \psi_{0,\LL}^+$ 
are local linear estimators of $\varphi_0^+$ and $\psi_0^+$  defined below \eqref{hatphi1}.

Noting the model at \eqref{basic.model}, and hence by \eqref{generated.response.predictor} and \eqref{def.r},
\be 
	\hat X_i  = X_i   \hat w_X(U_i), \quad \hat Y_i =  Y_i \, \hat w_Y(U_i) . \label{hat.X.tilde.Y}
\ee
Substituting the expressions in \eqref{hat.X.tilde.Y} into \eqref{LL.m} gives
\begin{align}
&	 \hat  m_{\NW}(x)  -m(x)  = 	\{ n \hat{f}_{\hat{X}}(x) \}^{-1} \sn K_h(x-\hat X_i)\{\hat w_Y(U_i)-w_0(U_i)\} Y_i \nn \\
   +& 	\{ n \hat{f}_{\hat{X}}(x) \}^{-1} \sn K_h(x-\hat X_i) \{m(X_i)-m(x) \}     
  + 	\{ n \hat{f}_{\hat{X}}(x) \}^{-1}  \sn K_h(x-\hat X_i)  \nn\\
  & \times \sigma(X_i)\varepsilon_i   
 \equiv   \hat \Pi_{01}(x)+   \hat \Pi_{02}(x) + \hat \Pi_{03}(x), \label{dec.1}
\end{align}
where
\be
	 \hat{f}_{\hat{X}}(x)  \equiv   n^{-1} \sn K_h(x-\hat{X}_i). \label{hatf.definition}
\ee

\noi
{\it Proof of (i).}\\
We start by establishing uniform bounds for $\hat w_X$ and $\hat w_Y$ which will be useful throughout the proof.
Recalling that the $U_i$'s are supported on $I_U=[0,1]$, for $Z=Y$ or $Z=X$ we use the notation 
$ \| \hat w_Z  - w_0  \|_\infty= \sup_{u\in [0,1]}|\hat w_Z(u)  - w_0(u)|$. 
To derive our bounds, note that under Conditions (B1)--(B6), for $\ell= 0 , 1 ,2$, we have (Masry, 1996; Hansen, 2008)
\begin{align} \label{unif.rate.1}
\begin{split}
	&\sup_{u\in [0,1]} \big| \hat \varphi_{0,LL}^{+ \, (\ell)} (u)- \varphi_0^{+\,(\ell)}(u) \big| =O_P \{ \delta_{ \ell n}(g_1)  \},  \\
	&\sup_{u\in [0,1]} \big|  \hat \psi_{0,LL}^{+\,(\ell)}(u)- \psi_0^{+\,(\ell)}(u) \big| = O_P \{ \delta_{ \ell n}(g_2)  \}, 
\end{split}  
\end{align}
where, for all $t>0$, 
\be
\delta_{\ell n}(t) \equiv  t^2+ (n t^{2\ell+1} )^{-1/2}  (\log n)^{1/2} .   \label{deltan.definition}
\ee
In particular, for $g_1 = g_{1n} \asymp n^{-\beta_1}$ and $g_2 = g_{2n} \asymp n^{-\beta_2}$, we have $\delta_{0n}(g_1) = O(   n^{- \lambda_1 } \sqrt{\log n} ) $ and $\delta_{0n}(g_2) = O(   n^{- \lambda_2 } \sqrt{\log n} ) $, where
$
 \lambda_\nu \equiv   \min(2\beta_\nu, 1/2-\beta_\nu/2) \in ( 0, 2/5]$, for $\nu =1, 2$.
 
Now, using \eqref{equiv.mod} and \eqref{def.r}, we can write
\be
\hat w_X(u)-w_0(u) 
= \frac{ \hat \mu_0^+ \,    \varphi(u) - \hat{\varphi}^+_{0, \LL}(u) }{\hat \varphi^+_{0,\LL}(u)} =  \frac{ (\hat \mu_0^+ - \mu_0^+  )  \varphi(u) }{\hat \varphi_{0, \LL}^+(u)}+\frac{   \varphi_0^+(u)  -\hat \varphi^+_{0, \LL}(u) }{\hat \varphi^+_{0,\LL}(u)  }\,,  \label{wX.difference}
\ee
and a similar equation can be written for $\hat w_Y$. 

Since, by Condition~(B1), $\gamma_1 \equiv  \min_{u\in [0,1]}\min \{ |\varphi_0^+(u)| , |\psi_0^+(u)|  \} >0$, a direct consequence of \eqref{unif.rate.1} and Taylor expansion is that
$
	\{  \hat \varphi^+_{0, \LL}(u)  \}^{-1} =  \{ \varphi_0^+(u) +  \hat \varphi^+_{0, \LL}(u) - \varphi^+_0(u)    \}^{-1} = \{ \varphi^+_0(u) \}^{-1} +  O_P \{ \delta_{0n}(g_1)  \}
$
uniformly over $u\in [0,1]$. Moreover, we also have
$
	\hat \mu_0^+ =  \mu_0^+ + O_P(n^{-1/2})$ and $\hat m_0^+ = m_0^+ +O_P(n^{-1/2}).
$
Substituting the previous two displays into \eqref{wX.difference} gives, for $Z_1 =X$ and $Z_2 =Y$,
\begin{equation}
   \| \hat w_{Z_\nu} - w_0  \|_\infty = O_P \{ \delta_{0n}(g_\nu)  \} =O_P \{  n^{- \lambda_\nu } (\log n)^{1/2}  \} . \label{pre.1}   
\end{equation}
Later in our proof, it will also be useful to use the fact that $\delta_{0n}(g_1) = o(h)$ because $\alpha_0<2\beta_1$.

Next we study the common denominator $\hat f_{\hat{X}}(x)$ of $\hat \Pi_{01}(x)$, $\hat \Pi_{02}(x)$ and $\hat \Pi_{03}(x)$. Let 
\be
	\hat{f}_X(x) = n^{-1} \sn K_h(x-X_i) \label{tildef.definition}
\ee
denote the standard kernel estimator of $f_X(x)$ that we would use if the $X_i$'s were available. For this estimator, it is well known (see e.g.~Theorem~6 in Hansen, 2008) that $\max_{x\in [a, b]}  | \hat{f}_X(x) -  f_X(x) | = O_P\{ \delta_{0n}(h) \}$. Shortly we shall prove that 
\be 
	 \max_{x\in [a,b] } \big| \hat{f}_{\hat{X}}(x) - \hat{f}_X(x) \big|= O_P \{ h^{-1} \delta_{0n}(g_1)  \} = o_P(1) ,  \label{dif.3*}
\ee
which further leads to $\max_{x\in [a,b]}   | \hat{f}_{\hat{X}}(x) - f_X(x)  | =  O_P \{ h^{-1} \delta_{0n}(g_1) + \delta_{0n}(h)  \} = o_P(1)$. In turn, using arguments similar to those we used above to treat the denominator of $ \hat w_{Z_\nu}  - w_0 $, and taking into account the fact that $  \min_{x\in [a,b]}f_X(x)\allowbreak >0$, we obtain
\begin{align}
	\{\hat{f}_{\hat{X}}(x)\}^{-1} & = \{f_X(x) + \hat{f}_{\hat{X}}(x) - f_X(x)\}^{-1} \nn \\
	& =\{f_X(x)\}^{-1} + O_P \{ h^{-1} \delta_{0n}(g_1) \} = \{f_X(x)\}^{-1} + o_P(1) \label{inverse.hatf.consistency}
\end{align}
uniformly over $x\in [a,b]$, and that $\max_{x\in[a,b]} \{ \hat{f}_{\hat{X}}(x) \}^{-1} = O_P(1)$.

Next we prove \eqref{dif.3*}. For this, note that for any $C>0$, we can write
\begin{align*}
& P\Big\{\max_{x\in [a,b] } \big| \hat{f}_{\hat{X}}(x) - \hat{f}_X(x) \big| > C  h^{-1} \delta_{0n}(g_1)  \Big\} \\
& \leq   P\Big\{\max_{x\in [a,b] } \big| \hat{f}_{\hat{X}}(x) - \hat{f}_X(x) \big| >C  h^{-1} \delta_{0n}(g_1), {\cal A}_n \Big\}
+ P ({\cal A}_n^{\rm c}  ),
\end{align*}
where ${\cal A}_n$  is an event that we shall define below, and which is such that $P({\cal A}_n)\to 1$ as $n\to\infty$.
Therefore, to prove \eqref{dif.3*}, it suffices to handle the first term on the right side of the inequality above.
Towards this end, first, comparing the definitions \eqref{hatf.definition} and \eqref{tildef.definition} we see that for each $x\in \br$, 
\begin{equation}
  \big| \hat{f}_{\hat{X}}(x) - \hat{f}_X(x) \big| 
   \leq  \| K'\|_\infty  \frac{\| \hat{w}_X -w_0 \|_\infty }{nh^2} \sn |X_i| I\big( |X_i-x|\leq h \mbox{ or } |\hat{X}_i -x |\leq h \big). \label{hatf.diff.bound.1}
\end{equation}
To further bound the right side of \eqref{hatf.diff.bound.1}, we shall show that $\hat{X}_i$ and $X_i$ are uniformly close (see \eqref{uniform.rate.error.in.predictors} below) as long as the estimation error of $\hat{w}_X$ is well-controlled. To see this, for $\lambda \geq 0$, define the event
\be
	\mathcal{E}_n(\lambda) = \big\{  \| \hat{w}_X - w_0 \|_\infty \leq n^{-\lambda}   \big\}.  \label{event.E}
\ee
By \eqref{pre.1}, we have $P\{ \mathcal{E}_n(\lambda) \} \rightarrow 1$ as $n\to \infty$ provided that $\lambda<\lambda_1$.
Moreover, define events
\be
	\mathcal{E}_{1n}(\lambda) = \Big\{ \max_{1\leq i\leq n}|X_i| \leq \lambda \log n \Big\} \ \ \mbox{ and } \ \  \mathcal{E}_{2n}(\lambda) = \Big\{ \max_{1\leq i\leq n}|\varepsilon_i | \leq \lambda \log n \Big\}. \label{events}
\ee
In the proof of Lemma~F.1 in the supplementary file, we shall show that for every given $c>0$, there exist a constant $C_1 >0$ such that $P\{ \mathcal{E}_{1 n}(C_1) \} \geq 1- \con\,  n^{-c}$.

Let $\alpha\in(\alpha_0 , \lambda_1 )$ be a constant, such that under Condition~(B6), $n^{-\alpha}=o(h)$ and $P\{ \mathcal{E}_n(\alpha)^{{\rm c}} \} \rightarrow 0$ as $n\to \infty$. On the event $\mathcal{E}_n(\alpha) \cap \mathcal{E}_{1n}(C_1)$, we have
\be
	\max_{1\leq i\leq n} |\hat{X}_i - X_i | \leq  \| \hat{w}_X -w_0 \|_\infty \max_{1\leq i\leq n}|X_i| \leq C_1 n^{-\alpha} \log n ,  \label{uniform.rate.error.in.predictors}
\ee
such that for every $x\in [a,b]$,
$
	|X_i-x| \leq     |\hat{X}_i -x| + C_1   n^{-\alpha} \log n.
$
Therefore, on the event $\mathcal{E}_n(\alpha) \cap \mathcal{E}_{1n}(C_1)$ with $n$ sufficiently large,
\be
	I\big(  |\hat{X}_i -x |\leq h \big) \leq I ( |X_i-x | \leq 2h ).  \label{indicator.inequality}
\ee
It follows from \eqref{hatf.diff.bound.1} and \eqref{indicator.inequality} that,  on $\mathcal{E}_n(\alpha) \cap \mathcal{E}_{1n}(C_1)$ with $n$ large enough,
\begin{align}
  &   \max_{x\in [a,b] }   \big|  \hat{f}_{\hat{X}}(x) - \hat{f}_X(x) \big|  \notag\\
 &   \leq \| K' \|_\infty  \| \hat{w}_X - w_0 \|_\infty  (n h^2)^{-1} \max_{x\in [a,b]} \sum_{i=1}^n |X_i| I ( |X_i-x|\leq 2h ) \nn \\
  &   \leq \| K' \|_\infty  \| \hat{w}_X - w_0 \|_\infty  (n h^2)^{-1} \max_{x\in [a,b]} (|x|+2h) \sum_{i=1}^n  I ( |X_i-x|\leq 2h ) \nn \\
 & \leq \con  \| \hat{w}_X -w_0 \|_\infty \,  h^{-2}   \max_{x\in [a,b] } \{ \hat{F}_X(x+2h  ) - \hat{F}_X(x- 2h)   \},      \label{dif.3} 
\end{align}
where $\hat{F}_X(x)=n^{-1}\sn I(X_i\leq x)$ denotes the empirical distribution function. To further bound the right-hand side of \eqref{dif.3}, we let $F_X$ be the distribution function of $X$ and then apply the Dvoretzky-Kiefer-Wolfwitz inequality (Massart, 1990) to obtain that $P( \s{n}\| \hat{F}_X  - F_X \|_\infty >y) \leq  2\exp(-2y^2)$ for all $y>0$, where $\| \hat{F}_X  - F_X \|_\infty\equiv  \sup_{x\in \br} |\hat{F}_X(x)-F_X(x)|$. For $\lambda >0$, define the event
\be
	\mathcal{E}_{0n}(\lambda ) = \Big\{  \sqrt{n}  \| \hat{F}_X  - F_X \|_\infty  \leq ( \lambda \log n)^{1/2}  \Big\}, \label{event.E0}
\ee
such that $P\{\mathcal{E}_{0n}(1/2)\}\geq 1- 2n^{-1}$. Under Condition~(B3), we deduce that on the $\mathcal{E}_{0n}(1/2)$ with $n$ sufficiently large,
\begin{align}
	 \max_{x\in [a,b] } & \{ \hat{F}_X(x+2h   ) - \hat{F}_X(x-2h)  \} 
	  \leq \max_{x\in [a,b] }  \{ F_X(x+2h  ) - F_X(x-2h )  \} \nn\\
	  &+   \{2(\log n)/n\}^{1/2}   
	 \leq 4 \| f_X \|_\infty \, h +   \{2(\log n)/n\}^{1/2}  \leq \con\, h.  \label{EDF.inequality}
\end{align}
Substituting this into \eqref{dif.3} and taking $\mathcal{A}_n \equiv  \mathcal{E}_n(\alpha) \cap \mathcal{E}_{0n}(1/2) \cap \mathcal{E}_{1n}(C_1)$ imply that for all sufficiently large $n$,
\begin{align}
	& P \Big\{  \max_{x\in [a,b] } \big| \hat{f}_{\hat{X}}(x) - \hat{f}_X(x) \big| > C  h^{-1} \delta_{0n}(g_1) , {\cal A}_n \Big\} \nn \\
	&  \leq P\big\{  \| \hat{w}_X -w_0 \|_\infty > \con \delta_{0n}(g_1)   , {\cal A}_n  \big\} \nn \\
	&  \leq  P\big\{  \| \hat{w}_X -w_0 \|_\infty > \con \, \delta_{0n}(g_1)   \big\},   \label{hatf.diff.bound.2}
\end{align}
and that $P(\mathcal{A}_n^{{\rm c}}) \rightarrow 0$ as $n\rightarrow \infty$. Together, \eqref{pre.1} and \eqref{hatf.diff.bound.2} prove \eqref{dif.3*}.

Next we study $\hat \Pi_{01}(x)$. For this, we first write $\hat{f}_{\hat{X}}(x)\,\hat{\Pi}_{01}(x)$ as 
\begin{align}
 n^{-1} &\sn K_h(x-X_i) ( \hat{w}_Y -w_0)(U_i)   \, Y_i + n^{-1} \sn \{ K_h(x-\hat{X}_i) - K_h(x-X_i) \} \notag\\
&  \times (\hat{w}_Y -w_0)(U_i) \, Y_i \equiv    \, J_1(x) + J_2(x).  \label{J1J2.definition}
\end{align}
Applying Lemma~F.4 with $g_2 \asymp n^{-\beta_2}$ to $J_1(x)$ implies 
\be 
	\max_{x\in [a,b]} | J_1(x) | = O_P(g_2^2)  = O_P( n^{-2 \beta_2} ). \label{J1.uniform.rate}
\ee 
For $J_2(x)$, note that 
$
 J_2 (x)   \leq \| K' \|_\infty \| \hat{w}_X - w_0 \|_\infty \| \hat{w}_Y - w_0 \|_\infty 
 \, (nh^2)^{-1} \sn \allowbreak|X_i|\{|m(X_i)|+ \sigma(X_i) |\varepsilon_i|\} I\big( |X_i -x |\leq h \, \mbox{ or } \, |\hat{X}_i -x | \leq  h \big).
$
The argument leading to \eqref{dif.3*} can be used to prove that
$	\max_{x\in [a,b]} (nh)^{-1} \allowbreak\sn |X_i \, m(X_i)| \allowbreak I\big( |X_i -x |\leq h \, \mbox{ or } \, |\hat{X}_i -x | \leq  h \big) = O_P(1) 
$
and the same bound holds if the $m(X_i)$'s are replaced by the $\sigma(X_i)$'s. Moreover, similarly to (F.9) in the proof of Lemma~F.1, it can be proved that
\be
 \max_{1\leq i\leq n}|\varepsilon_i| = O_P(\log n).  \label{max.exp.ubd}
\ee
This, together with \eqref{pre.1} and the two displays before \eqref{max.exp.ubd} yields
\begin{equation}
 \max_{x\in [a,b]} |J_2(x)| = O_P\{ h^{-1} \delta_{0n}(g_1) \delta_{0n}(g_2) \log n \} = o_P\{ g_2^2 + (ng_2)^{-1/2}  \}. \label{J2.uniform.rate}
\end{equation}
Here, the last step follows from Condition~(B6) and the assumption that $\alpha_0<2\beta_1$. Together, \eqref{inverse.hatf.consistency}, \eqref{J1J2.definition}, \eqref{J1.uniform.rate} and \eqref{J2.uniform.rate} imply
\be 
	\max_{x\in [a,b]} \big| \hat \Pi_{01}(x) \big| =O_P(g_2^2) + o_P \{ (ng_2)^{-1/2}  \} . \label{Pi01.uniform.rate}
\ee

For $\hat{\Pi}_{02}(x)$, we write $K_h(x-\hat{X}_i)$ in $\hat{f}_X(x) \, \hat{\Pi}_{02}(x)$ as $K_h(x-\hat{X}_i) - K_h(x-X_i) + K_h(x-X_i)$. A similar argument to what we used to study \eqref{hatf.diff.bound.1} gives
\begin{align}
 & \max_{x\in [a,b]} \Big| n^{-1} \sn \{ K_h(x-\hat{X}_i) - K_h(x-X_i )  \} \{m(X_i) - m(x) \} \Big| \leq \| m'  \|_\infty \| K'    \|_\infty\notag\\
 &  \times \frac{\|\hat{w}_X -w_0 \|_\infty}{nh^2} \max_{x\in[a,b]} \Big|  \sn |X_i(X_i-x)| \,  I\big( |X_i -x | \leq h \, \mbox{ or } \, |\hat{X}_i -x | \leq h \big)  \Big| \nn \\
 &  =  O_P\{  \delta_{n,0}( g_1 ) \}. \label{error.in.bias}
\end{align}
Together with \eqref{deltan.definition} and \eqref{inverse.hatf.consistency}, this implies
\be 
	  \max_{x\in [a,b]}  \big| \hat  \Pi_{02}(x) - \Pi_{02}(x) \big| = O_P \{ \delta_{n,0}(g_1)   \},   \label{approxi.2}
\ee
where $\Pi_{02}(x) \equiv  \{n \hat{f}_{\hat{X}}(x)\}^{-1} \sn  K_h(x-  X_i) \{m(X_i) - m(x)\}$.

Next, we write $\hat{f}_{\hat{X}}(x)\, \Pi_{02}(x)= n^{-1} \sn  K_h(x-  X_i) \{m(X_i) - m(x)\}$ as
\be
	  n^{-1} \sn \{g_{n,i}(x) - E  g_{n,i}(x) \} +  n^{-1} \sn E  g_{n,i}(x) \equiv  R_n(x) +  n^{-1} \sn E  g_{n,i}(x)   , \label{Pi02.dec}
\ee
where $g_{n,i}(x)= K_h(x-X_i)\{m(X_i)-m(x)\}$. To bound $\max_{x\in [a,b]}|R_n(x)|$, we create a grid using $N$ points of the form $x_j=a+j\epsilon $ with $\epsilon=(b-a)/N$ for some $N \geq 1$ to be determined below \eqref{event.Cn}. Since $g_{n,i}'(x)=h^{-1}K'_h(x-X_i)\{m(X_i)-m(x)\}- m'(x) K_h(x-X_i)$, by the mean value theorem we have, for every $x,  y \in \br$,  $| g_{n,i}(x)  - g_{n,i}(y) | \leq  (  \| K \|_\infty + \| K' \|_\infty    ) \| m' \|_\infty  \,h^{-1} |x-y|$. Therefore,
\be 
	\max_{x\in [a,b]}|R_n(x)| \leq \max_{1\leq j\leq N }|R_n(x_j)| + 2 (  \| K \|_\infty + \| K' \|_\infty   ) \| m' \|_\infty  \,\epsilon h^{-1}.  \label{Rnx.ubd.1}
\ee 
For each $x\in \br$ fixed, $g_{n,1}(x), \ldots ,g_{n,n}(x)$ are independent random variables satisfying $| g_{n,i}(x) | \leq \| K \|_\infty \| m' \|_\infty $ and 
$
	 E \{ g_{n,i}(x) \}^2 
	 = h^{-1} \int K^2(t)\{m(x-ht)-m(x)\}^2 f_X(x-ht) \, dt   \leq \| m' \|_\infty^2 \| f_X \|_\infty  \, h\int t^2 K^2(t)  \, dt.
$
Hence, by Bernstein's inequality and Boole's inequality, for every $y\geq 0$,
\begin{align}
	    P\Big\{ \max_{1\leq j\leq N }&|R_n(x_j)| \geq   y  \Big\}  
	   \leq \sum_{j=1}^N P \Big[ \Big| n^{-1}  \sn \{ g_{n,i}(x_j) - E g_{n,i}(x_j) \} \Big| \geq y \Big]  \nn \\
	   \leq & 2N \exp\Big\{ - \frac{n  y^2}{ 2 ( c_K ^2\| m' \|_\infty^2 \| f_X \|_\infty  \, h    + \| K \|_\infty \| m' \|_\infty y/3) } \Big\}, \label{Bennnett.inequality}
\end{align}
where $c_K \equiv  \{ \int t^2 K^2(t)  \, dt\}^{1/2}$. For every $\lambda>0$, define the event
\be
\mathcal{C}_n(\lambda ) = \Big\{ \max_{1\leq j\leq N }|R_n(x_j)| \leq  c_K\| m' \|_\infty \| f_X \|_\infty^{1/2}  \sqrt{ \frac{h \lambda }{n}} +  \| K \|_\infty \| m' \|_\infty \frac{\lambda}{n} \Big\}, \label{event.Cn}
\ee
such that in view of \eqref{Bennnett.inequality}, $P\{ \mathcal{C}_n(\lambda)^{{\rm c }} \} \leq 2N\exp(-\tau \lambda)$ for some absolute constant $\tau>0$. By taking $N=n$ and $\lambda=2\tau^{-1}\log n$, it follows from \eqref{Rnx.ubd.1} and \eqref{event.Cn} that
\be
	\max_{x\in [a,b]} |R_n(x)| = O_P \{  h^{1/2}(n/\log n)^{-1/2} + n^{-1} \log n +  (nh)^{-1}  \}. \label{Rnx.ubd.2}
\ee

For the second term on the right-hand side of \eqref{Pi02.dec}, standard arguments show that, under Conditions~(B3) and (B5),
\begin{equation}
     E g_{n,i}(x) =  \{  m''(x)f_X(x) /2 + m'(x)f_X'(x) \} \mu_{K,2} \,h^2 + o(h^2)  \label{bias.uniform.rate.2}
\end{equation}
uniformly in $x\in [a,b]$. Consequently, combining \eqref{inverse.hatf.consistency}, \eqref{approxi.2}, \eqref{Rnx.ubd.2} and \eqref{bias.uniform.rate.2}, we get
\be
	 \max_{x\in [a,b]} \big| \hat  \Pi_{02}(x)  \big| = O_P(h^2). \label{Pi02.uniform.rate}
\ee

For the last term $\hat{\Pi}_{03}(x)$ in \eqref{dec.1}, we need to control the stochastic error
\be
 {\Delta}_{n,\infty}\equiv 	\max_{x\in [a,b]}  \Big| n^{-1} \sn  K_h(x-\hat{X}_i )\sigma(X_i) \varepsilon_i \Big|
 \label{stochastic.error.term}
\ee
for $\hat{X}_i = X_i \hat{w}_X(U_i)$ as in \eqref{hat.X.tilde.Y}. To this end, we shall use a lattice argument by making a finite approximation of the compact interval $[a,b]$ using a sequence $\{x_j\}_{j=1}^N$ of equidistant points $x_j=a+j\epsilon $ for $\epsilon=(b-a)/N$, and then discretize $\Delta_{n,\infty}$ to define
$
	\Delta_{n,N} \equiv  \max_{1\leq j\leq N} \big| n^{-1} \sn K_h(x_j-\hat{X}_i ) \sigma(X_i) \varepsilon_i \big| .
$
Here, $N$ is a positive integer that will be determined after \eqref{deviation.ineq}.

Instead of dealing with $\Delta_{n,\infty}$ directly, we shall prove that $\Delta_{n,N}$ provides a fine approximation to $\Delta_{n,\infty}$, at least with high probability, and then restrict attention to $\Delta_{n,N}$. By definition of $\Delta_{n,N}$, we have $|\Delta_{n,\infty} - \Delta_{n,N} | \leq \| \sigma  \|_\infty  \| K' \|_\infty \, \epsilon h^{-2} \max_{1\leq i\leq n}|\varepsilon_i|$. Together with \eqref{max.exp.ubd}, this leads to
\be
	|\Delta_{n,\infty} - \Delta_{n,N} | = O_P ( N^{-1} h^{-2} \log n ). \label{finite.approxi.rate}
\ee

For $\Delta_{n,N}$, shortly we shall prove by taking $N=n$ that
\be
\Delta_{n,N} = 	O_P \{ (nh/\log n)^{-1/2}  \},  \label{discrete.ubd}
\ee
which together with \eqref{finite.approxi.rate} leads to
\be
	 \Delta_{n,\infty}   = O_P \{ (nh/\log n)^{-1/2}    + (nh^2)^{-1} \log n  \} = O_P \{ (nh /\log n)^{-1/2} \}, \label{stochastic.error.rate}
\ee
where the last step relies on the identity $(nh^2)^{-1}\log n= (nh/\log n)^{-1/2}\allowbreak(nh^3/\log n)^{-1/2}$ and Condition~(B6). Combing \eqref{inverse.hatf.consistency} and \eqref{stochastic.error.rate} yields
\be
	\max_{x\in[a,b]} \big| \hat{\Pi}_{03}(x) \big| \leq    \max_{x\in [a,b]} \{\hat{f}_{\hat{X}}(x)\}^{-1} \Delta_{n,\infty}  = O_P \{  (nh /\log n)^{-1/2} \}. \label{Pi03.uniform.rate}
\ee

Together, \eqref{Pi01.uniform.rate}, \eqref{Pi02.uniform.rate} and \eqref{Pi03.uniform.rate} complete the proof of \eqref{thm.1}.

Next we prove \eqref{discrete.ubd}. For $\lambda>0$, let $V_{1n}(x)   =  \{  \sn K^2_h(  x- \hat{X}_i )\sigma^2(X_i)  \}^{1/2}$, 
$	 V_{2n}(x)  =\max_{1\leq i\leq n} K_h ( x- \hat{X}_i )\sigma(X_i)$
 and  define the event
\be
	\mathcal{D}_n(N,\lambda) = \Big\{  |\Delta_{n,N}| \leq   \max_{1\leq j\leq N}V_{1n}(x_j)  \sqrt{\lambda}/n+ \max_{1\leq k\leq N}V_{2n}(x_j) \lambda/ n \Big\}\,. \label{event.Dn}
\ee

To deal with $ V_{1n}(x)$, as in the proof of \eqref{hatf.diff.bound.2}, put $ {\cal A}_n = \mathcal{E}_n(\alpha) \cap \mathcal{E}_{1n}(C_1)  \cap \mathcal{E}_{0n}(1/2) $ with $\alpha \in (\alpha_0 ,  \lambda_1)$ such that $P(\mathcal{A}_n^{{\rm c}}) \rightarrow 0$ as $n\rightarrow \infty$, where $\mathcal{E}_n(\alpha)$, $\mathcal{E}_{1n}(C_1)$ and $\mathcal{E}_{0n}(1/2)$ are as in \eqref{event.E}, \eqref{events} and \eqref{event.E0}, respectively. On the event ${\cal A}_n $ with $n$ sufficiently large, it follows from \eqref{indicator.inequality} and \eqref{EDF.inequality} that
\begin{align}
	\max_{1\leq j\leq N}V_{1n}(x_j)& \leq  \max_{x\in [a,b]} V_{1n}(x)   
	 \leq  \frac{ \| \sigma \|_\infty \| K \|_\infty}{h} \max_{x\in [a,b]} \sqrt{  \sn I ( |X_i-x| \leq 2h  )  } \nn\\
	&\leq \con \| \sigma \|_\infty \| K \|_\infty (n/h)^{1/2} .  \label{V1.ubd}
\end{align}
It is easy to see that $\max_{x\in [a,b]}V_{2n}(x)\leq \| \sigma \|_\infty \| K \|_\infty  \, h^{-1}$. This, combined with \eqref{event.Dn} and \eqref{V1.ubd} yields, on the event $\mathcal{D}_n(N,\lambda) \cap \mathcal{A}_n$ with $n$  large enough,
\be
	 \Delta_{n,N}  \leq \con \| \sigma \|_\infty \| K \|_\infty \big\{  \sqrt{\lambda/(nh)} + \lambda/(nh) \big\}.  \label{DnN.restricted.bound}
\ee

Next we show that for properly chosen $N$ and $\lambda$, $P \{ \mathcal{D}_n(N,\lambda)^{{\rm c}} \} \rightarrow 0$ as $n\rightarrow \infty$. Observe that $\hat{w}_X$ defined in \eqref{def.r} is a measurable function of $\{(X_i, U_i)\}_{i=1}^n$ and thus is independent of $\{\varepsilon_i\}_{i=1}^n$. Conditional on $\{(X_i, U_i)\}_{i=1}^n$, taking 
$
	\mathbf{a}=( a_1, \ldots, a_n)\T = \big( K_h(x-\hat{X}_i)\sigma(X_i), \ldots, K_h(x-\hat{X}_n)\sigma(X_n) \big)\T
$
in Lemma~F.2 and using Boole's inequality, we obtain that for every $\lambda \geq 0$,
$
	P\big[   \Delta_{n,N}   >   \max_{1\leq j\leq N}V_{1n}(x_j) \sqrt{\lambda}/n + \max_{1\leq k\leq N}V_{2n}(x_j) \, \lambda/n  \,  \big| \{(X_i, U_i)\}_{i=1}^n  \big] \leq  2N \exp(- c \lambda)
$
where $c >0$ is a constant independent of $n$ and $N$. Taking expectations on both sides of the inequality gives that for every $\lambda \geq 0$, $P\{ \mathcal{D}_n(N,\lambda)^{{\rm c}}  \} \leq  2N \exp(- c \lambda)$. Taking $N=n$ and $\lambda=2c^{-1}\log n$ we get 
\be
	P \{ \mathcal{D}_n(n,\lambda)^{{\rm c}} \} \leq  2n^{-1}.	\label{deviation.ineq}
\ee	
Combining \eqref{DnN.restricted.bound} with $N=n, \lambda=2c^{-1}\log n$, \eqref{deviation.ineq} and the fact that $P(\mathcal{A}_n^{{\rm c}})\rightarrow 0$ proves \eqref{discrete.ubd} as claimed.

\noi
{\it Proof of (ii).}\\
To prove the asymptotic normality, we need to use a more refined argument. In what follows, $x\in [a,b]$ is fixed and we deal with the sum in \eqref{dec.1} over each $\hat{\Pi}_{0j}(x)$ separately. 

First, for $\hat{\Pi}_{01}(x)$, recall in \eqref{J1J2.definition} that $\hat{f}_{\hat{X}}(x) \, \hat{\Pi}_{01}(x)= J_1(x)+ J_2(x)$. By \eqref{J2.uniform.rate} and Condition (B6),
$
	   |J_2(x)|   = o_P \{ h^{-1}\delta_{0n}(g_1) \delta_{0n}(g_2) \log n \}=o_P( g_1^2 + g_2^2 ).
$
For $J_1(x)$, Lemma~F.4 with $g_{2n} \asymp n^{-\beta_2}$ implies
$
	 J_1(x) =   -\tfrac{1}{2} m(x)f_X(x)\allowbreak E\{\psi''(U)/\psi(U) \}  \mu_{L,2} \, g_2^2+ o_P(g_2^2). \nn
$
The last two displays and \eqref{inverse.hatf.consistency} imply
\begin{equation}
  \hat{\Pi}_{01}(x)  = - m(x) E\{\psi''(U)/\psi(U) \}  \mu_{L,2} \, g_2^2/ 2  + o_P( g_1^2 + g_2^2 ) . \label{Pi01.pointwise.rate}
\end{equation}

For $\hat{\Pi}_{02}(x)$, by a first-order Taylor's expansion we obtain
\begin{align}
 & \hat{f}_{\hat{X}}(x) \, \hat{\Pi}_{02}(x)	=  n^{-1} \sn K_h(x - \hat{X}_i ) \{ m(X_i) -m(x) \}\nn \\
	& = n^{-1} \sn K_h(x - X_i ) \{ m(X_i) -m(x) \} + (nh^2)^{-1} \sn K'\Big( \frac{x-X_i}{h} \Big) \nn \\
	& \qquad \times  ( w_0-\hat{w}_X)(U_i) \, X_i \{ m(X_i) -m(x) \} + (2nh^3)^{-1} \sn K''(\xi_n)\nn \\
	& \qquad \times   (w_0 -\hat{w}_X)^2(U_i) \, X_i^2 \{ m(X_i) -m(x) \} I\big( |\hat{X}_i -x|\leq h \big) \nn  \\
	& \equiv   I_1(x) + I_2(x) + I_3(x), \label{dec.2}
\end{align}
where $\xi_n$ is a random variable that lies between $(x-X_i)/h$ and $(x-\hat{X}_i)/h$.

A standard argument shows that $ I_1(x) = O_P\{ h^2 + (nh)^{-1/2} \}$. Together with \eqref{inverse.hatf.consistency}, this yields
\begin{align}
	  \{\hat{f}_{\hat{X}}(x)\}^{-1} I_1(x) 
	   &  = \{ f_X(x) \}^{-1} I_1(x) +  O_P [ h^{-1}\delta_{n,0}(g_1 ) \{h^2 + (nh)^{-1/2}\}  ] \nn \\
	& = \{ f_X(x) \}^{-1} I_1(x) +  o_P \{ h^2 + (nh)^{-1/2}  \} .
\end{align}
For $I_2(x)$, it follows from (F.12) in Lemma~F.3 and \eqref{inverse.hatf.consistency} that
$
\{\hat{f}_{\hat{X}}(x)\}^{-1} I_2(x)\allowbreak =  x m'(x) E\{\varphi''(U) / \varphi(U) \} \mu_{L,2} \,  g_1^2/2 + o_P ( g_1^2 ) .  
$
For $I_3(x)$, a similar argument to that leading to \eqref{dif.3*}  yields $ \max_{x\in [a,b]} | I_3(x) | \allowbreak3= O_P \{ h^{-1} \delta_{n,0}^2 (g_1) \}$ and hence,
$
	  \{ \hat{f}_{\hat{X}}(x) \}^{-1}   I_3(x)  = o_P \{ (nh )^{-1/2}  \}  .  
$
Combining with with \eqref{dec.2} we get
\begin{align}
	\hat{\Pi}_{02}(x) =&  \{f_X(x)\}^{-1} I_1(x) + x m'(x) E\{\varphi''(U) / \varphi(U) \} \mu_{L,2} \,  g_1^2/2 \notag\\
	&+ o_P \{ g_1^2 + h^2 + (nh )^{-1/2}  \}  \label{Pi02.pointwise.rate}
\end{align}
for $I_1(x)$ as in \eqref{dec.2}.
Finally, for the stochastic error term $\hat{\Pi}_{03}(x)$, we shall use an argument similar to that employed in Mammen et al.~(2012) based on empirical process theory. Write $\beta_1=( 1+\xi_0 ) /5$ for some $\xi_0\geq 0$. First, we argue that the estimator $\hat w_X$ falls within a ``nice'' function space, the complexity of which can be measured via covering numbers. Let $\mathcal{M}_{0n}$ be the set of functions $[0, 1]\mapsto \br$ whose derivatives up to order two exist and are uniformly bounded in order by $(ng_1^5/\log n)^{-1/2} \asymp n^{ \xi_0 /2} (\log n)^{1/2} $. Since $\beta_1\geq 1/5$, we have $\lambda_1=\min(2\beta_1, 1/2-\beta_1/2)=1/2-\beta_1/2$. For some $ \alpha \in ( \alpha_0 , 1/2-\beta_1/2)$ to be specified in the paragraph after \eqref{variance}, we define the following set of functions:
\be  
	\mathcal{N}_{0n} =   \big\{ w \in  \mathcal{M}_{0n} : \| w - w_0 \|_\infty \leq n^{-\alpha } \big\}.  \label{function.space}
\ee 
By \eqref{unif.rate.1}, using the same argument that we used to derive \eqref{pre.1}, we have $P(\hat w_X \in \mathcal{N}_{0n}) \rightarrow 1$ as $n\rightarrow \infty$.

Note that $\hat{f}_{\hat{X}}(x) \, \hat{\Pi}_{03}(x)$ in \eqref{dec.1} can be written as 
$
n^{-1} \sn  \{ K_h(x-\hat{X}_i)- K_h(x-X_i) \}\sigma(X_i) \varepsilon_i  + n^{-1} \sn K_h(x-X_i)  \sigma(X_i) \varepsilon_i.
$
For the first term, by Lemma~F.1 we have, for any $\kappa_1 \in (0, 1/2+3\alpha/4 - 3\alpha_0/2- \xi_0/8)$,
$
	 n^{-1} \sn \{ K_h(x-\hat{X}_i)- K_h(x-X_i) \} \sigma(X_i) \varepsilon_i   = O_P( n^{-\kappa_1} ).
$
On the other hand, it is straightforward to show that
$
	   n^{-1} \sn K_h(x- X_i) \sigma(X_i)   \varepsilon_i   = O_P\{  (nh)^{-1/2} \} =  O_P ( n^{-1/2+\alpha_0/2}   ) .
$
Combining this and \eqref{inverse.hatf.consistency}, we get
\begin{align}
	 \hat \Pi_{03}(x)  =&\{n f_X(x)\}^{-1}\sn K_h(x-X_i) \sigma(X_i) \varepsilon_i  \notag\\
	 &+ O_P \{  n^{-\kappa_1 } + n^{-1+\beta_1/2 + 3\alpha_0 /2}  (\log n)^{1/2} \} \nn \\
	 =&   \{n f_X(x)\}^{-1}\sn K_h(x-X_i) \sigma(X_i) \varepsilon_i    +  O_P ( n^{-\kappa_1}  ) .  \label{Pi03.pointwise.rate}
\end{align}

Assembling \eqref{Pi01.pointwise.rate}, \eqref{Pi02.pointwise.rate} and \eqref{Pi03.pointwise.rate} we obtain that, for any $\alpha\in (\alpha_0, 1/2-\beta_1/2 )$ and $\kappa_1\in (0, 1/2+3\alpha/4-3\alpha_0/2 - \xi_0/8 )$,
\begin{align}
  	\hat{m}_{\NW}(x) - m(x)  
   =&\tilde{B}(x)+ \{f_X(x)\}^{-1} I_1(x)  + \sqrt{V(x)} \, N(x) \notag\\
   &+ o_P (   n^{-\kappa_1} +  g_1^2 + g_2^2  + h^2 ),  \label{oracle.comparison.1}
\end{align}
where $\tilde{B}(x)$ and $I_1(x)$ are as in part (ii) of Theorem~\ref{thm.1}  and \eqref{dec.2}, respectively, and
$
	N(x) \equiv  \{ V(x) \}^{-1/2}\, \{n f_X(x)\}^{-1}\sn K_h(x-X_i) \sigma(X_i) \varepsilon_i 
$
for $V(x)$ is as in part (ii) of Theorem~\ref{thm.1}. Further, for $I_1(x)=n^{-1}\sn K_h(x-X_i)\{ m(X_i) - m(x) \}$, proceeding as in \eqref{Pi02.dec} we derive that
\begin{equation}
	 \{f_X(x)\}^{-1} I_1(x)    =  B_0(x) +  o_P \{ h^2 + (nh)^{-1/2}  \}    \label{bias}
\end{equation}
for $B_0(x)$ as in part (ii) of Theorem~\ref{thm.1}. For the third addend on the right-hand side of \eqref{oracle.comparison.1}, Lyapounov's central limit theorem combined with Slutsky's theorem yield
\begin{equation}
	N(x)   \xrightarrow {\mathscr{D}} N(0,1), \   \textrm{ as $n\rightarrow \infty$}.\label{variance}
\end{equation}
In particular, for $h=h_n \asymp n^{-\alpha_0}$ with $\alpha_0 \in (0, 1/2- \beta_1)$, by taking $\alpha$ and $\kappa_1$ in such a way that
$
	\tfrac{4}{3} \alpha_0 < \alpha < \tfrac{1}{2} - \tfrac{1}{2} \beta_1$ and $\tfrac{1}{2}-\tfrac{1}{2}\alpha_0 <  \kappa_1 < \tfrac{1}{2} + \tfrac{3}{4} \alpha - \tfrac{3}{2} \alpha_0 - \tfrac{1}{8} \xi_0 ,
$
we have $n^{-\kappa_1}=o\{(nh)^{-1/2}\}$. This, together with \eqref{oracle.comparison.1}--\eqref{variance} proves \eqref{AN.mNW}.			\qed

\medskip
\noindent
{{\large \bf Acknowledgement.}} We thank three referees and an Associate Editor for their helpful comments which led to an improved version of the manuscript. This research was supported by the Australian Research Council.

\begin{supplement}[id=suppA]
 \stitle{Supplement to ``Nonparametric covariate-adjusted regression''}
 \slink[doi]{}
  \sdatatype{.pdf}
 \sdescription{This supplemental material contains more details for the implementation of the proposed estimators, additional simulation results as well as additional proofs omitted in the main text.}
\end{supplement}

\end{document}


\begin{frontmatter}
\title{Supplement to ``Nonparametric covariate-adjusted regression''}
\runtitle{Covariate-adjusted regression}

\begin{aug}
  \author{\fnms{Aurore}  \snm{Delaigle}\corref{}\thanksref{t2}\ead[label=e1]{A.Delaigle@ms.unimelb.edu.au}},
  \author{\fnms{Peter} \snm{Hall}\thanksref{t2}\ead[label=e2]{halpstat@ms.unimelb.edu.au}}
\and
  \author{\fnms{Wen-Xin}  \snm{Zhou}\corref{}\thanksref{t2}\ead[label=e3]{wenxinz@princeton.edu}}

  \thankstext{t2}{Research supported by grants and fellowships from the Australian Research Council.}

  \runauthor{A. Delaigle, P. Hall and W.-X. Zhou}

  \affiliation{University of Melbourne and Princeton University}

 \address{School of Mathematics and Statistics \\ 
 	University of Melbourne \\ 
 	Parkville, Victoria 3010 \\
    Australia \\ 
\printead{e1}\\
\phantom{E-mail:\ }\printead*{e2}\\
\phantom{E-mail:\ }\printead*{e3}}

 \address{Department of Operations Research \\
			  and Financial Engineering \\
			   Princeton University \\
			   Princeton, New Jersey 08544 \\
			   USA \\  \printead{e3}}

\end{aug}

\begin{abstract}
This supplemental material contains more details for the implementation of the proposed estimators, additional simulation results as well as additional proofs omitted in the main text.
\end{abstract}

\end{frontmatter}

\appendix \label{appendix.sec}
\setcounter{equation}{0}

\section{Details for Section~5.2} \label{app:detailsparam}
Here we provide the details for the way in which we compute $\rho_1^*$ and $\rho_2^*$. 
As commonly done in the literature, abusing notation, we use ${\rm AMISE}_w(\gamma,g)$ to denote the weighted asymptotic ``mean integrated squared error'' of the standard local linear estimator of a regression curve $\gamma(v)=E(Z|V=v)$ computed with a bandwidth $g$, which in fact denotes the integral of the sum of the asymptotic squared bias and variance terms coming from the asymptotic distribution of the estimator, weighted by  $w(x)=f_V(x)\omega(x)$, where $\omega$ is a weight function. For $k,\nu=0,1,\ldots$, define
$\nu_{K,k}=\int x^k K^2(x)\,dx$ and  $\theta_\nu\equiv \int \{\gamma^{(\nu)}(x)\}^2 f_V(x)\omega(x)\,dx$.
Also, let $\sigma^2(v)\equiv\var(Z\,|\,V=v)$. Then, it is well known (Fan and Gijbels, 1996) that
\begin{align}
{\rm AMISE}_w(\gamma,g)=\mu_{K,2}^2 \frac{g^4}{4}\,  \theta_2   +  \frac{\nu_{K,0}}{ng } \int  \sigma^2(x) \omega(x) \,dx\,.\label{AMISEw}
\end{align}

For the estimators in Section~3.3, we take $\rho_1^*$ and $\rho_2^*$ equal to the square root of estimators of, respectively, ${\rm AMISE}_w(\varphi_0,g_{1,\PI})$ and ${\rm AMISE}_w(\psi_0,g_{2,\PI})$, taking there, respectively, $\omega(x)=1\{|\hat \varphi_{0,\LL}(x)|\leq q_\varphi\}$ and $\omega(x)=1\{|\hat \psi_{0,\LL}(x)|\leq q_\psi\}$ with $q_\varphi$ and $q_\psi$ denoting the empirical 0.2-quantile of the $|\hat \varphi_{0,\LL}(U_i)|$'s and the $|\hat \psi_{0,\LL}(U_i)|$'s, respectively.  
For the estimators in Section~3.4, we take $\rho_1^*$ and $\rho_2^*$ equal to estimators of, respectively, ${\rm AMISE}_w(\varphi^*,g_{1,\PI})$ and ${\rm AMISE}_w(\psi^*,g_{2,\PI})$, taking there, respectively, $\omega(x)=1\{|\hat \varphi_{\LL}^*(x)|\leq q_\varphi\}$ and $\omega(x)=1\{|\hat \psi_{\LL}^*(x)|\leq q_\psi\}$ with $q_\varphi$ and $q_\psi$ denoting the empirical 0.2-quantile of the $|\hat \varphi_{\LL}^*(U_i)|$'s and the $|\hat \psi_{\LL}^*(U_i)|$'s, respectively.

To estimate ${\rm AMISE}_w(\gamma,g)$, we use the same ideas as those for computing plug-in bandwidths.
First, we replace  $\int  \sigma^2(x) \omega(x) \,dx$ by $\hat\sigma^2  \int   \omega(x) \,dx$, where  $\hat\sigma^2$ is a constant approximation to  the function $\sigma^2(x)$, obtained  using the difference-based approach; see Buckley et al.~(1988) and Hall et al.~(1990). That is, 
$\hat\sigma^2=\sum_{i=1}^{n-1}(Z_{[i+1]}-Z_{[i]})^2/\{2(n-1)\},$
where $Z_{[i]}$ denotes the concomitant of $V_{(i)}$, the $i$th order statistic of the $V_j$'s.
Then, we estimate $\theta_2$ by
$\hat\theta_2\equiv n^{-1} \sum_{i=1}^n \{\hat \gamma''(V_i)\}^2 \omega(V_i),$
where
$\hat \gamma''$ is a local polynomial estimator of $\gamma''$ of order three computed with a bandwidth $h_2$ selected  as in Ruppert et al.~(1995).

\section{Additional simulation results}
\label{app:simul}

\begin{table}[h]
\begin{center}
\vspace*{-.2cm}
\scriptsize
\caption{\baselineskip=14pt Simulation results for models (i.a) to (iii.a) and (i.b) to (iii.b). The numbers show 100 $\times$ Median integrated squared error [first quartile, third quartile]  calculated from $1000$ simulated samples, using our estimators   $\hat m_\LL$ from Section~2.2 (NEW1), $\hat m_\LL(\cdot ; \rho_1,\rho_2)$ from Section~2.3 (NEW2) and  $\hat m_\LL(\cdot ; \rho_1,\rho_2)$ from Section~2.4 (NEW3), the  naive estimator $\hat m_{\LL,{\rm naive}}$ (NAIVE) and the oracle estimator $\tilde  m_{\LL}$ (ORACLE).}
\label{table:Simul1}
\vspace*{.15cm}
\hspace*{-.1cm}
\begin{tabular}{|c|c|l|l|l|l|l|l|l|l|l|}
\hline
Model&$n$&ORACLE & NEW1 & NEW2 & NEW3& NAIVE  \\
\hline
(i.a)
&$100$&10 [7,14]&16 [11,23]&61 [32,105]&19 [12,37]&72 [56,97]\\
&$200$&5 [4,7]&8 [6,11]&25 [13,46]&9 [6,15]&64 [52,77]\\
&$500$&2 [2,3]&3 [2.75,5]&6 [4,11]&4 [3,6]&56 [49,64]\\
&$1000$&1 [1,2]&2 [1,2]&3 [2,4]&2 [1,3]&53 [48,59]\\
&&&&&&\\[-.3cm]
(ii.a)
&$100$&0.33 [0.24,0.46]&0.84 [0.57,1.26]&1.54 [0.94,2.46]&1.15 [0.64,4.75]&7 [5,9]\\
&$200$&0.17 [0.13,0.23]&0.40 [0.27,0.57]&0.72 [0.45,1.10]&0.49 [0.30,1.20]&7 [5,8]\\
&$500$&0.08 [0.06,0.10]&0.15 [0.11,0.22]&0.26 [0.17,0.40]&0.19 [0.12,0.42]&6 [6,7]\\
&$1000$&0.04 [0.03,0.05]&0.08 [0.06,0.11]&0.12 [0.08,0.19]&0.10 [0.06,0.26]&6 [5,7]\\
&&&&&&\\[-.3cm]
(iii.a)
&$100$&61 [48,76]&101 [83,122]&143 [114,174]&113 [85,164]&197 [181,222]\\
&$200$&38 [31,47]&62 [50,75]&94 [71,118]&68 [54,92]&189 [175,206]\\
&$500$&20 [17,23]&31 [26,37]&48 [37,61]&32 [26,42]&180 [170,192]\\
&$1000$&11 [10,13]&16 [14,20]&26 [20,34]&18 [14,24]&176 [167,185]\\
&&&&&&\\[-.3cm]
(i.b)
&$100$&10 [7,14]&20 [14,28]&23 [15,32]&25 [16,46]&143 [109,187]\\
&$200$&5 [4,7]&10 [7,14]&11 [8,16]&12 [8,20]&130 [107,159]\\
&$500$&2 [2,3]&4 [3,6]&5 [3,6]&5 [4,10]&123 [108,140]\\
&$1000$&1 [1,2]&2 [2,3]&2 [2,3]&3 [2,5]&121 [109,133]\\
&&&&&&\\[-.3cm]
(ii.b)
&$100$&0.33 [0.25,0.47]&0.99 [0.66,1.51]&1.24 [0.78,1.85]&1.61 [0.84,5.63]&9 [8,11]\\
&$200$&0.18 [0.13,0.23]&0.48 [0.32,0.70]&0.57 [0.40,0.86]&0.68 [0.38,1.66]&9 [8,10]\\
&$500$&0.08 [0.06,0.10]&0.18 [0.13,0.28]&0.20 [0.14,0.31]&0.25 [0.14,0.58]&9 [8,9]\\
&$1000$&0.04 [0.03,0.05]&0.09 [0.06,0.14]&0.11 [0.07,0.16]&0.14 [0.08,0.43]&8 [8,9]\\
&&&&&&\\[-.3cm]
(iii.b)
&$100$&59 [47,75]&139 [114,167]&154 [123,194]&158 [120,282]&317 [292,345]\\
&$200$&38 [31,46]&95 [76,119]&105 [83,129]&108 [83,167]&306 [289,323]\\
&$500$&19 [16,23]&51 [41,66]&56 [43,71]&59 [42,90]&300 [289,311]\\
&$1000$&11 [10,13]&30 [23,39]&31 [25,41]&38 [26,69]&297 [290,304]\\

\hline  
\end{tabular}
\end{center}
\end{table}

\begin{table}[t]
\begin{center}
\vspace*{-.2cm}
\scriptsize
\caption{\baselineskip=14pt Simulation results for models (i.c) to (iii.c) and (i.d) to (iii.d). The numbers show  100 $\times$ Median integrated squared error [first quartile, third quartile]  calculated from $1000$ simulated samples, using our estimators  $\hat m_\LL(\cdot ; \rho_1,\rho_2)$ from Section~2.3 (NEW2) and $\hat m_\LL(\cdot ; \rho_1,\rho_2)$ from Section~2.4 (NEW3), the  naive estimator $\hat m_{\LL,{\rm naive}}$ (NAIVE) and the oracle estimator $\tilde  m_{\LL}$ (ORACLE).}
\label{table:Simul2}
\vspace*{.15cm}
\hspace*{-.1cm}
\begin{tabular}{|c|c|l|l|l|l|}
\hline
Model&$n$&ORACLE &  NEW2 & NEW3& NAIVE  \\
\hline
(i.c)
&$100$&10 [7,14]&70 [39,116]&33 [21,53]&139 [125,155]\\
&$200$&5 [4,7]&32 [18,59]&16 [11,25]&136 [125,146]\\
&$500$&2 [2,3]&12 [7,20]&6 [4,9]&140 [132,150]\\
&$1000$&1 [1,2]&6 [3,10]&3 [2,4]&146 [139,153]\\
&&&&&\\[-.3cm]
(ii.c)
&$100$&0.34 [0.25,0.47]&2.89 [1.91,4.25]&2.02 [1.32,3.30]&25 [21,30]\\
&$200$&0.18 [0.14,0.24]&1.62 [1.07,2.43]&0.96 [0.60,1.57]&26 [22,30]\\
&$500$&0.08 [0.06,0.10]&0.62 [0.39,0.97]&0.39 [0.22,0.68]&27 [25,29]\\
&$1000$&0.04 [0.03,0.06]&0.30 [0.19,0.50]&0.17 [0.10,0.33]&28 [26,30]\\
&&&&&\\[-.3cm]
(iii.c)
&$100$&60 [48,76]&179 [153,216]&153 [123,195]&446 [329,598]\\
&$200$&38 [30,46]&139 [115,164]&99 [82,122]&493 [390,612]\\
&$500$&19 [16,23]&85 [69,103]&52 [43,67]&579 [497,671]\\
&$1000$&11 [10,13]&51 [41,63]&29 [23,38]&636 [573,701]\\
&&&&&\\[-.3cm]
(i.d)
&$100$&10 [7,14]&55 [36,89]&60 [35,108]&896 [763,1033]\\
&$200$&5 [4,7]&27 [18,42]&29 [17,55]&898 [809,1004]\\
&$500$&2 [2,3]&11 [7,18]&12 [7,23]&924 [864,976]\\
&$1000$&1 [1,2]&5 [4,9]&6 [3,13]&934 [887,983]\\
&&&&&\\[-.3cm]
(ii.d)
&$100$&0.33 [0.24,0.47]&2.75 [1.78,4.06]&2.66 [1.72,4.09]&23 [20,26]\\
&$200$&0.18 [0.14,0.24]&1.39 [0.90,2.07]&1.30 [0.81,2.05]&24 [21,26]\\
&$500$&0.08 [0.06,0.10]&0.50 [0.32,0.83]&0.53 [0.31,0.98]&25 [24,27]\\
&$1000$&0.04 [0.03,0.05]&0.25 [0.14,0.41]&0.25 [0.14,0.48]&26 [25,28]\\
&&&&&\\[-.3cm]
(iii.d)
&$100$&60 [48,78]&235 [188,314]&206 [170,258]&393 [358,437]\\
&$200$&38 [31,46]&172 [141,221]&156 [128,192]&387 [360,420]\\
&$500$&19 [16,23]&101 [82,127]&101 [79,132]&408 [383,439]\\
&$1000$&11 [10,13]&62 [50,81]&64 [51,93]&431 [409,453]\\

\hline  
\end{tabular}
\end{center}
\end{table} 

In addition to the models introduced in Section~5.3, we considered the following heteroscedastic versions of models (iv.a) to (vi.a):
(iv.c) $m(\cdot)=m_1(\cdot+1)$, $X_i\sim N(0,1.5^2)$, $\sigma(x)=0.15(1+x^2)^{1/2}$;
(v.c) $m(\cdot)=m_2(\cdot+1)$, $X_i\sim N(0,1.5^2)$, $\sigma(x)=0.025(1+x^2)^{1/2}$;
(vi.c) $m(\cdot)=m_3(\cdot+0.5)$, $X_i\sim N(0,0.75^2)$, $\sigma(x)= 0.275(1+x^2)^{1/2}$.

In Tables \ref{table:Simul1} to \ref{table:Simul4}, for each method considered in our numerical study in Section~5.3, we report the first, second and third quartiles of the 1000 ISEs resulting from estimators computed on 1000 samples generated according to the model at~(2.1), in each of the settings (i.a) to (vi.c) detailed in Section~5.3.

\begin{table}[t]
\begin{center}
\vspace*{-.2cm}
\scriptsize
\caption{\baselineskip=14pt Simulation results for models (iv.a) to (vi.a) and (iv.b) to (vi.b). The numbers show 100 $\times$ Median integrated squared error [first quartile, third quartile]  calculated from $1000$ simulated samples, using our estimator from Section~3.4 (NEW3), the  naive estimator $\hat m_{\LL,{\rm naive}}$ (NAIVE) and the oracle estimator $\tilde  m_{\LL}$ (ORACLE).}
\label{table:Simul3}
\vspace*{.15cm}
\hspace*{-.1cm}
\begin{tabular}{|c|c|l|l|l|l|l|l|l|l|l|}
\hline
&&\multicolumn{3}{|c|}{ (a)}&\multicolumn{3}{c|}{ (b)}\\
Model&$n$&ORACLE & NEW3 & NAIVE&ORACLE & NEW3 & NAIVE \\
\hline
(iv)
&$100$&10 [7,14] 
&39 [25,62]&320 [259,395]
&8 [5,14]&25 [16,38]&116 [102,134]\\
&$200$&5 [4,7] 
&21 [13,32]&332 [285,388]
&4 [3,6]&12 [8,18]&115 [104,128]\\
&$500$&2 [2,3] 
&8 [5,13]&354 [321,390]
&2 [1,3]&5 [3,7]&118 [109,127]\\
&$1000$&1 [1,2] 
&4 [2,7]&369 [344,399]
&1 [1,2]&2 [2,4]&123 [116,131]\\
&&&&&&&\\[-.3cm]
(v)
&$100$&
0.34 [0.25,0.49]&1.86 [1.23,2.93]&10 [8,12] &0.29 [0.19,0.44]&1.69 [1.01,3.18]&16 [11,23]\\
&$200$&
0.17 [0.13,0.23]&0.95 [0.59,1.53]&10 [8,11] &0.16 [0.11,0.23]&0.78 [0.46,1.35]&16 [12,22]\\
&$500$&
0.08 [0.06,0.10]&0.38 [0.23,0.69]&10 [9,11] &0.07 [0.05,0.09]&0.29 [0.18,0.56]&15 [13,19]\\
&$1000$&
0.04 [0.03,0.05]&0.18 [0.11,0.32]&11 [10,12] &0.04 [0.03,0.05]&0.15 [0.09,0.29]&16 [14,19]\\
&&&&&&&\\[-.3cm]
(vi)
&$100$
&60 [48,76]&153 [115,216]&272 [255,301]&95 [74,128]&188 [155,229]&381 [282,576]\\
&$200$
&37 [30,46]&82 [65,113]&266 [253,280]&62 [48,79]&130 [108,159]&396 [313,519]\\
&$500$
&19 [17,23]&36 [29,49]&258 [249,266]&34 [27,42]&71 [57,90]&416 [343,512]\\
&$1000$
&11 [10,13]&19 [15,24]&255 [246,261]&20 [16,25]&40 [32,49]&459 [403,531]\\

\hline  
\end{tabular}
\end{center}
\end{table}

\begin{table}[t]
\begin{center}
\vspace*{-.2cm}
\scriptsize
\caption{\baselineskip=14pt Simulation results for models (iv.c) to (vi.c). The numbers show  100 $\times$ Median integrated squared error [first quartile,third quartile]  calculated from $1000$ simulated samples, using our estimator $\hat m_\LL(\cdot ; \rho_1,\rho_2)$ from Section~2.4 (NEW3), the  naive estimator $\hat m_{\LL,{\rm naive}}$ (NAIVE) and the oracle estimator $\tilde  m_{\LL}$ (ORACLE).}
\label{table:Simul4}
\vspace*{.15cm}
\hspace*{-.1cm}
\begin{tabular}{|c|c|l|l|l|}
\hline
Model&$n$&ORACLE &   NEW3& NAIVE  \\
\hline
(iv.c)
&$100$&12 [8,17]&43 [27,70]&323 [265,396]\\
&$200$&6 [4,9]&22 [14,35]&341 [292,397]\\
&$500$&3 [2,4]&8 [5,14]&358 [322,397]\\
&$1000$&2 [1,2]&4 [3,7]&372 [346,399]\\
&&&&\\[-.3cm]
(v.c)
&$100$&0.39 [0.28,0.56]&1.97 [1.29,3.28]&9.61 [7.80,11.8]\\
&$200$&0.21 [0.15,0.29]&1.00 [0.63,1.61]&9.71 [8.35,11.3]\\
&$500$&0.09 [0.07,0.12]&0.39 [0.24,0.68]&10.2 [9.29,11.3]\\
&$1000$&0.05 [0.04,0.07]&0.19 [0.12,0.34]&10.8 [9.94,11.7]\\
&&&&\\[-.3cm]
(vi.c)
&$100$&39 [29,51]&143 [101,202]&272 [254,294]\\
&$200$&21 [16,26]&74 [54,106]&265 [253,277]\\
&$500$&9 [8,11]&26 [20,36]&258 [249,265]\\
&$1000$&5 [4,6]&12 [10,17]&254 [246,260]\\

\hline  
\end{tabular}
\end{center}
\end{table}

First we discuss the results from Table~\ref{table:Simul1}. In cases (i.a) to (iii.a), $EX$ and $EY$ are relatively close to zero, and unsurprisingly, the estimator from Section~3.3 did not work as well as the one from Section~3.4, which, although it worked well, was outperformed by the estimator from Section~3.2. In cases (i.b) to (iii.b), all quantities are far from zero, so that the estimator from Section~3.4 was outperformed by the one from Section~3.3, which was itself outperformed by the estimator from Section~3.2. 

The results from Table~\ref{table:Simul2} show that, in cases (i.c) to (iii.c), $EX$ and $EY$ are close to zero so that the estimator from Section~3.3  encountered problems and was significantly outperformed by the estimator from Section~3.4, whereas in cases (i.d) to (iii.d), where $EX$ and $EY$ are far from zero, the estimator from Section~3.3 typically worked best.

Table~\ref{table:Simul3} shows simulation results in cases where the only estimator we could apply was the one from Section~3.4. The results there indicate that, while this estimator was of course outperformed by the oracle estimator that can only be computed when the error-free data are available, it worked well and its performance improved as sample size increased. In Table~\ref{table:Simul4} we show the results obtained in the heteroscedastic version of models (iv.a)  to (vi.a). We can see that the method from Section~3.4 worked well in those cases too.

\section{Details for additive models in Section~6}\label{app:addmod}
In the additive model considered in  Section~6, instead of taking $\hat m_j=0$ as initial estimators, following Section~7.2 of Fan and Gijbels~(1996) we could start with a linear approximation of the model in (6.4).
There we would use the approximation $ Y = m^0_0 + \sum_{j=1}^d m^0_j   X_j  + \varepsilon$, and compute estimators $\hat{m}_0^0, \hat{m}^0_1,\ldots,\hat{m}^0_d$ of $m_0^0, m^0_1,\ldots,m^0_d$  using the methods of \c{S}ent\"urk and M\"uller~(2005a, 2006).
Then, recalling the condition $E \{ m_j(X_j) \}=0$ for $j=1,\ldots, d$, which ensures identifiability of the functions $m_j$'s, we would rewrite our approximation as
$ Y = m^0_0 + \sum_{j=1}^d m^0_j   X_j  + \varepsilon \,\sigma({\bf X})= ( m^0_0 +  \sum_{j=1}^d m^0_j \bar X_j )  + \sum_{j=1}^d m^0_j  \cdot( X_j-\bar X_j) + \varepsilon\,\sigma({\bf X})$, where, for $j=1,\ldots,d$, $\bar X_j= n^{-1}\sum_{i=1}^n\,\tilde X_{ij}$.
In this notation, recalling also the notation $m_0,m_1,\ldots, m_d$ in (6.4), we would take our initial estimators equal to $\hat{m}_0 = \hat{m}_0^0 +  \sum_{j=1}^d \hat{m}^0_j \bar X_j$ and, for $j=1,\ldots,d$, $\hat m_j(x)=\hat m_j^0 \cdot (x-\bar X_j)$. 
The subsequent updates would be obtained in the same way as those for the algorithm described in the previous paragraph.

\section{Proof of Theorem~4.2}\label{sec:proofTh2}

We shall prove only part (ii), since part (i) can be derived by adapting the results on the rate of convergence in the sup-norm for the local constant estimator $\hat{m}_{\NW}$ as in Theorem~4.1. To prove the asymptotic normality for $\hat{m}_{\LL}$, we introduce the following notation. As in the proof of (4.1), let $\alpha$ be a constant between $\alpha_0$ and $\lambda_1 =\min(2\beta_1, 1/2-\beta_1/2)=1/2-\beta_1/2$, and write $\beta_1=(1+\xi_0)/5$ for some $\xi_0\geq 0$. Moreover, let $	\mathbf{v}_i(x ; w)  = \left(v_{1i}(x ; w) , v_{2i}(x ; w) \right)\T =  \left(1, \{X_i w(U_i)-x \}/h \right)\T $, $\mathbf{S}_n(x;w)   = n^{-1} \sn   K_h\{X_i w(U_i) - x\} \mathbf{v}_i(x ; w) \mathbf{v}_i(x ; w)\T  $ and
\begin{align}
	 \mathbf{T}_n(x; w_1, w_2) & = \left( T_{1n}(x; w_1, w_2) , T_{2n}(x; w_1, w_2)\right)\T  \nn \\
	&= n^{-1} \sn K_h\{X_i  w_1(U_i) -x \} \mathbf{v}_i(x ;w_1)  w_2(U_i) Y_i  \in \br^2 ,  \label{def.Tb}
\end{align}
where $w, w_1, w_2$ are functions $[0, 1] \mapsto \br$. In particular, we put $\mathbf{v}_{i}(x) = \mathbf{v}_i(x ; w_0) = (1, h^{-1}(X_i-x))\T $ and $\mathbf{S}_n(x) = \mathbf{S}_n(x ; w_0)$ for $w_0 \equiv 1$ as in (7.2).

In the above notation, the local linear estimator $\hat m_{\LL}$ given in (3.5) can be written as 
\begin{align}
\hat m_{\LL}(x) & = \mathbf{e}_1\T \mathbf{S}^{-1}_n(x ; \hat w_X ) \, \mathbf T_n(x ; \hat w_X, \hat w_Y)  \nn \\
&= \mathbf{e}_1\T \mathbf S_n^{-1}(x ; \hat w_X )  \, \mathbf T_n(x ; \hat w_X, \hat w_Y -w_0 )+\mathbf{e}_1\T \mathbf S^{-1}_n(x ;\hat w_X )  \, \mathbf T_n(x; \hat w_X, w_0 ) \nn \\
	& \equiv  \hat m_{\LL,1}(x) + \hat m_{\LL,2}(x),   \label{mLL.dec} 
\end{align}
where $\hat w_X$ and $\hat w_Y$ are as in (7.2).

First we study $\hat m_{\LL,2}$. Noting the model at (2.1), a standard decomposition gives 
\begin{align*}
	\mathbf T_n(x ; \hat w_X, w_0 ) & =     m(x) \, n^{-1} \sn K_h\{X_i   \hat w_X(U_i) -x \} \mathbf{v}_i(x ; \hat w_X) \\
	& \  \ + n^{-1} \sn K_h\{X_i  \hat w_X(U_i) -x \} \mathbf{v}_i(x ;\hat w_X)  \{ m(X_i) -m(x) - m'(x)(X_i -x)\}    \\ 
	& \ \ +   m'(x) \, n^{-1}\sn K_h\{X_i   \hat w_X(U_i) -x \} \mathbf{v}_i(x;\hat w_X)   (w_0 -\hat w_X)(U_i) \, X_i \\
	& \ \ +   m'(x) \, n^{-1}\sn K_h\{X_i \hat w_X(U_i) -x \} \mathbf{v}_i(x;\hat w_X)   \{ X_i  \hat w_X(U_i) -x \} \\
	& \ \ + n^{-1} \sn K_h\{X_i   \hat w_X(U_i) -x \} \mathbf{v}_i(x;\hat w_X)  \sigma(X_i) \varepsilon_i \\
	& \equiv    \mathbf{Q}_{n,1}(x ;\hat w_X) +  \mathbf{Q}_{n,2}(x; \hat w_X)+  \mathbf{Q}_{n,3}(x; \hat w_X) +   \mathbf{Q}_{n,4}(x; \hat w_X)+  \mathbf{Q}_{n,5}(x; \hat w_X),
\end{align*}
and note that $ \mathbf{e}_1\T \mathbf{S}^{-1} _n(x ; \hat w_X )\, \mathbf{Q}_{n,1}(x; \hat w_X) = m(x)$, $\mathbf{e}_1\T \mathbf{S}_n(x; \hat w_X )^{-1} \, \mathbf{Q}_{n,4}(x;\hat w_X) = 0$. Hence, we only need to deal with $\mathbf{Q}_{n,2}(x; \hat w_X)$, $\mathbf{Q}_{n,3}(x;\hat w_X)$ and $\mathbf{Q}_{n,5}(x; \hat w_X)$.

\noindent 
{\it Step 1. Controlling the stochastic error term $\mathbf{e}_1\T \mathbf S_n(x ; \hat w_X)^{-1} \, \mathbf Q_{n,5}(x;\hat{w}_X)$}. Since $\alpha<\lambda_1$, for $\mathcal{E}_n(\alpha)$ as in (7.14) we have $P\{ \mathcal{E}_n(\alpha) \} \to 1$ as $n\to \infty$. This, together with \eqref{dif.1} implies that for any $\kappa_1\in (0, 1/2+3\alpha/4 - 3\alpha_0/2 - \xi_0/8)$,
\be 
   \mathbf Q_{n,5}(x ;\hat{w}_X) = \mathbf Q_{n,5}(x; w_0) + O_P  (n^{-\kappa_1 } ),  \label{pf.2.0}
\ee
where $\mathbf Q_{n,5}(x;w_0) \equiv n^{-1} \sn K_h(X_i-x) \mathbf{v}_i(x) \sigma(X_i) \varepsilon_i$ with $\mathbf v_i(x)=(1, h^{-1}(X_i-x))\T$.

For $\mathbf S_n(x ; \hat w_X)$, using arguments similar to those used to prove (7.11) and noting that $\delta_{0n}(g_1)=O_P\{ (ng_1/\log n)^{-1/2} \}$ for $g_1=g_{1n}\asymp n^{- \beta_1 }$ with $\beta_1\geq 1/5$, it can be shown that
\begin{equation}
 	\left\| \mathbf  S_n(x; \hat{w}_X)   - \mathbf{S}_n(x) \right\|_\infty    = O_P \{ h^{-1}(ng_1)^{-1/2} (\log n)^{1/2} \}. \label{pf.2.1}
\end{equation}
Here and below, $\| \mathbf{a} \|_\infty= \max_k |a_k|$ and $\| \mathbf A \|_\infty =\max_{k,\ell} |A_{k \ell}|$ denote, respectively, the elementwise $\ell_\infty$-norm of a vector $\mathbf{a}=(a_k)$ and of a matrix $\mathbf A=(A_{k\ell})$.

For $\mathbf  Q_{n,5}(x;w_0)$, since $\varepsilon_i$ and $X_i$ are independent and $E(\varepsilon_i) =0$, a direct consequence of Condition (B5) is that
\be
	 \left\| \mathbf  Q_{n,5}(x;w_0) \right\|_\infty = O_P \{ (nh )^{-1/2} \} = O_P ( n^{-1/2 +\alpha_0/2}  ) . \label{Tn5.uniform.rate}
\ee

Next, for $\mathbf S_n(x)$ in \eqref{pf.2.1}, define $\mathbf S^*_n(x) = E \{\mathbf S_n(x) \}$ and note that $\| \mathbf S_n(x)  - \mathbf S^*_n(x) \|_\infty =O_P\{  (nh)^{-1/2}  \}$. It is also straightforward to show that $\mathbf S_n^*(x) = f_X(x) \,\mathbf S + O(h)$, where $\mathbf S = \mbox{diag}\,(\mu_{K,0}, \mu_{K,2})$ is strictly positive-definite. Since $x\in [a,b] \subset I_X$, then $f_X(x)>0$ and we can apply to $S_n(x)$ and $f_X(x) \, \mathbf{S}$ the approximation $(A+\varepsilon B)^{-1}=A^{-1} - \varepsilon A^{-1} B A^{-1} + O(\varepsilon^2)$ as $\varepsilon\rightarrow 0$ which holds for any invertible matrix $A$. Proceeding that way we get
\be 
 \left\| \mathbf S^{-1}_n(x)    - \{ f_X(x) \}^{-1} \mathbf S^{-1}  \right\|_\infty =O_P\{ (nh )^{-1/2} +h \} = O_P(h) .   \label{matrix.inverse.consistency.2}
\ee
Recall that $\alpha< \lambda_1$, the right-hand side of \eqref{pf.2.1} is of order $o_P(n^{-\alpha+\alpha_0})$. Hence, a similar argument to that used to prove \eqref{matrix.inverse.consistency.2} leads to
\be
	  \left\| \mathbf S^{-1}_n(x ; \hat w_X)  - \mathbf S^{-1}_n( x )  \right\|_\infty  = o_P ( n^{-\alpha+\alpha_0}  ).  \label{matrix.inverse.consistency.1}
\ee

Together, \eqref{pf.2.0}, \eqref{Tn5.uniform.rate} and \eqref{matrix.inverse.consistency.1} imply
\be  
	  \mathbf{e}_1\T  \mathbf S^{-1}_n(x ;\hat w_X )  \, \mathbf Q_{n,5}(x ; \hat w_X)  =  \mathbf{e}_1\T \mathbf  S^{-1}_n(x )  \, \mathbf Q_{n,5}(x ; w_0) + O_P (n^{-\ka_1}  ).  \label{Tn5.comparison}
\ee

\noindent 
{\it Step 2. Controlling the bias term $\mathbf{e}_1\T \mathbf S^{-1}_n(x;\hat{w}_X)  \,\mathbf  Q_{n,2}(x;\hat{w}_X)$}. For $\mathbf Q_{n,2}(x ;\hat w_X)$, letting $\bar{K}(t)=tK(t)$ and using (7.9) and the argument used to prove (7.27), it can be shown that both
$   n^{-1} \sn  | K_h(\hat{X}_i -x) - K_h(X_i-x) | \cdot  | m(X_i)  - m(x)  - m'(x)(X_i-x) |$ and $  n^{-1} \sn | \bar{K}_h(\hat{X}_i -x) - \bar{K}_h(X_i-x) |  \cdot  | m(X_i)  - m(x) - m'(x)(X_i-x) |$ are of order $O_P \{  h \delta_{0n}(g_1)  \} = o_P ( n^{-\alpha_0-\alpha} )$. This implies that
\begin{equation}
	  \big\| \mathbf  Q_{n,2}(x ;\hat w_X) - \mathbf Q_{n,2}(x ; w_0 ) \big\|_\infty = o_P (n^{-\alpha_0- \alpha}  ),  \label{Tn2.difference}
\end{equation}
where $\mathbf Q_{n,2}(x;w_0)  \equiv  n^{-1}\sn K_h(X_i-x) \mathbf{v}_i(x) \{ m(X_i) - m(x) -m'(x)(X_i-x) \}$. Under Condition (B3), $| m(X_i) - m(x) -m'(x)(X_i-x)  | I( |X_i-x| \leq h ) \leq  \frac{1}{2} \| m'' \|_\infty \, h^2$ holds almost surely. Therefore, the argument leading to (7.20) can be used to prove that $ \| \mathbf Q_{n,2}(x ; w_0 )  \|_\infty  =  O_P(h^2) = O_P (n^{-2\alpha_0} )$. Together with \eqref{matrix.inverse.consistency.1} and \eqref{Tn2.difference}, this implies
\be
	  \mathbf{e}_1\T \mathbf  S^{-1}_n(x; \hat{w}_X)  \, \mathbf Q_{n,2}(x ; \hat w_X)  =   \mathbf{e}_1\T \mathbf S^{-1}_n(x ) \, \mathbf Q_{n,2}(x ; w_0)  + o_P  (n^{-\alpha_0 -\alpha  }  ). \label{Tn2.comparison}
\ee

\noindent 
{\it Step 3. Bounding the term $\mathbf{e}_1\T \mathbf S^{-1}_n(x;\hat{w}_X) \, \mathbf Q_{n,3}(x;\hat{w}_X)$}. Under Condition (B6), property (7.9) with $Z=X$, and arguments similar to those employed to bound $J_2(x)$ in (7.22), can be used to show that
$
 \left\| \mathbf Q_{n,3}(x ; \hat w_X) - m'(x) n^{-1}\sn K_n(X_i-x)\mathbf v_i(x) (w_0-\hat w_X)(U_i) \, X_i \right\|_\infty 
 = O_P \{ (ng_1 h/\log n)^{-1} \} = o_P(h^2) .
$
Further, using \eqref{matrix.inverse.consistency.1} we obtain
\begin{equation}
   \mathbf{e}_1\T \mathbf S^{-1}_n(x; \hat w_X )  \, \mathbf Q_{n,3}(x ; \hat w_X)    =  m'(x) \Delta_X(x)+ o_P(h^2),  \label{Tn3.diff}  
\end{equation}
where $ \Delta_X(x) \equiv  \mathbf{e}_1\T \mathbf S^{-1} _n(x ) \, n^{-1} \sn K_h(X_i -x)   \mathbf v_{i}(x)   ( w_0 -\hat w_X)(U_i)\, X_i$. This, together with \eqref{matrix.inverse.consistency.2} and the bound
$\| n^{-1}\sn K_h(X_i -x)   \mathbf v_{i}(x)   ( w_0 -\hat w_X)(U_i)\, X_i  \|_\infty =O_P \{ (ng_1 )^{-1/2}  \}
$ leads to 
\begin{align}
	 \Delta_X(x) &  = \mathbf{e}_1\T \mathbf  S^{-1} \{n f_X(x)\}^{-1}\sn K_h(X_i -x) \mathbf v_{i}(x)  ( w_0 -\hat w_X)(U_i)\, X_i  + o_P( h^2 )  \nn \\  
	& =  \{ n f_X(x) \}^{-1} \sn K_h(X_i-x) ( w_0 -\hat w_X)(U_i)\, X_i  + o_P(h^2). \label{DeX.equiv}
\end{align}
Assembling \eqref{Tn3.diff} and \eqref{DeX.equiv} we obtain that
\begin{equation}
  \mathbf{e}_1\T \mathbf S^{-1}_n(x; \hat w_X ) \, \mathbf Q_{n,3}(x ; \hat w_X)  
       =  \frac{ m'(x) } { n f_X(x)  } \sn K_h(X_i-x)  ( w_0 -\hat w_X)(U_i)\,  X_i + o_P(h^2). \label{Tn3.comparison}  
\end{equation}

\noindent 
{\it Step 4. Approximation to $\hat{m}_{\LL,2}$ in \eqref{mLL.dec}}. Now, let $\tilde m_{\LL}(x) =  \mathbf{e}_1\T \mathbf S^{-1}_n(x ) n^{-1}\sn K_h(X_i-x)\mathbf v_i(x) Y_i$ be the local linear regression estimator of $m(x)=E(Y_i|X_i=x)$ that we would use if the $(Y_i,X_i)$'s were available. A standard argument, see for example, (A1) and (A2) in Delaigle et al.~(2009), shows that $\tilde m_{\LL}(x)$ can be written as
$$
	\tilde m_{\LL}(x) =  m(x)  +  \mathbf{e}_1\T \mathbf S^{-1}_n(x ) \left\{\mathbf Q_{n,2}(x;w_0) + \mathbf Q_{n,5}(x;w_0) \right\},
$$
where $\mathbf Q_{n,2}(x;w_0) $ and $\mathbf Q_{n,5}(x;w_0) $ are as in \eqref{Tn2.difference} and \eqref{pf.2.0}, respectively. In this notation, it follows from \eqref{Tn5.comparison}, \eqref{Tn2.comparison} and \eqref{Tn3.comparison} that, for $\hat{m}_{\LL,2}$ as in \eqref{mLL.dec},
\begin{equation} 
	 \hat m_{\LL,2}(x)  -  \tilde m_{\LL}(x) 
	 = \frac{ m'(x)}{ n  f_X(x)}   \sn K_h(X_i-x) (w_0 -\hat w_X)(U_i)\, X_i  + O_P ( n^{- \kappa_1 } ) + o_P(h^2).  \label{mLL2.approximation}
\end{equation}

\noindent 
{\it Step 5. Approximation to $\hat{m}_{\LL,1}$ in \eqref{mLL.dec}}. Next we turn to 
$
	\hat m_{\LL,1}(x)= \mathbf{e}_1\T \mathbf S^{-1}_n(x;  \hat w_X ) \,\allowbreak \mathbf T_n(x ; \hat w_X, \hat w_Y -w_0 ) ,
$ 
where $\mathbf T_n=(T_{1n}, T_{2n})\T$ is as in \eqref{def.Tb}. By (7.9) and (7.24), arguments similar to those used to bound $J_2(x)$ in (7.22) yield, on this occasion,
\begin{align*} 
&  \left| T_{1n}(x; \hat{w}_X, \hat w_Y- w_0 ) - T_{1n}(x; w_0, \hat w_Y- w_0) \right| \\
		& \leq   \| K' \|_\infty  \| \hat w_Y -w_0 \|_\infty \| \hat{w}_X -w_0 \|_\infty     (nh^2)^{-1} \sn |X_i Y_i| I\big( |X_i-x|\leq h \, \mbox{ or } \, |\hat{X}_i -x |\leq h \big) \\
		& = O_P \big[  h^{-1} (ng_1)^{-1/2}  \{ g_2^2 + (ng_2)^{-1/2}\}  \log n  \big] =  o_P \{  g_2^2 + (nh)^{-1/2}  \},
\end{align*}
where the last equality follows from Condition (B6). An analogous argument leads to the same bound for $  | T_{2n}(x; \hat{w}_X , \hat w_Y-  w_0 ) - T_{2n}(x; w_0, \hat w_Y- w_0 )  |$. Combining the above calculations with \eqref{matrix.inverse.consistency.1} we obtain
$\hat m_{\LL,1}(x)  = \mathbf{e}_1\T \mathbf S^{-1}_n(x, \hat w_X ) \, \mathbf  T_n(x ; w_0, \hat w_Y -w_0 )  +o_P(g^2) =\De_Y(x)+ o_P \{  g_2^2 + (nh)^{-1/2}  \}$, where $\De_Y(x) \equiv \mathbf{e}_1\T \mathbf S^{-1}_n(x )  \, n^{-1} \sn K_h(X_i-x)  \mathbf v_i(x) (\hat w_Y- w_0)(U_i)\,Y_i$. The argument leading to \eqref{DeX.equiv} now gives
$
	\De_Y(x)  =  \{n f_X(x)\}^{-1} \sn K_h(X_i - x)(\hat w_Y- w_0)(U_i)\, Y_i + o_P(h^2)  .
$
We deduce that
\be
	\hat m_{\LL,1}(x) = \{n f_X(x)\}^{-1} \sn K_h(X_i - x)(\hat w_Y- w_0)(U_i)\, Y_i  + o_P \{ g_2^2 + h^2+(nh)^{-1/2}  \} . \label{mLL1.approximation}  
\ee

\noindent 
{\it Step 6. Completion of the proof of (4.2)}. For the standard local linear estimator $\tilde{m}_{\LL}$ of $m$ that we would use if the $(X_i, Y_i)$'s were available, it is known that
\begin{equation}
	\tilde{m}_{\LL}(x)- m(x) = \sqrt{ V(x) }\,  N(x) + B_1(x) + o_P \{ h^2 + (nh)^{-1/2}  \} , \label{CLT.local.linear}
\end{equation}
where the distribution of $N(x)$ converges to the standard normal distribution, and the functions $V$ and $B_1$ are given by $V(x)   = \{ nh f_X(x)\}^{-1}   \sigma^2(x) \int K^2(u) \, du$ and $B_1(x)  =   h^2m''(x) \mu_{K,2}/2$, respectively. See, for example, Delaigle et al.~(2009), setting there the measurement errors to zero.

In the present context, for $h \asymp n^{-\alpha_0}$ with $\alpha_0 \in ( 0,  1/2-\beta_1)$, we may take $\alpha$ and $\kappa_1$ in the same way as in the paragraph after (7.53) so that $n^{-\kappa_1}=o\{(nh)^{-1/2}\}$. Finally, combining \eqref{mLL2.approximation}-- \eqref{CLT.local.linear} and Lemmas~\ref{lm.0} and \ref{lm.1} proves (4.2). \qed

\section{Proof of Theorem~4.3}\label{sec:proofTh3}
\noi
{\it Proof of (i).}\\
We adopt the notation in the proof of Theorem~4.1.  Let $M_0=\max_{u\in I_U}|\varphi_0(u)|$ and note that from (2.2) and (7.1), we have $M_0 \geq |\mu_0|$. For $x\in \br$ and $\lambda>0$, define
\begin{equation}
	\hat{f}_{\hat{X}}(x ; \lambda)   =  n^{-1} \sum_{i\in \hat{\mathcal{C}}_n(\lambda) } K_h( x-  \hat X_i )  =n^{-1} \sn K_h(x-\hat{X}_i) I\{ U_i \in \hat{\mathcal{L}}_n(\lambda)  \}.\label{hat.fx}
\end{equation}
Recalling (7.3), we can write
\begin{align}
	 \hat m(x ;  \rho) -m(x)   & = \{n \hat{f}_{\hat{X}}(x ; \rho)\}^{-1} \sum_{i\in \hat{\mathcal{C}}_n(\rho) }K_h(x-\hat{X}_i)\{ \hat{w}_Y(U_i)- w_0(U_i)\} Y_i \nn \\
	& \quad  +  \{n \hat{f}_{\hat{X}}(x ; \rho)\}^{-1}  \sum_{i\in \hat{\mathcal{C}}_n(\rho) }K_h(x-\hat{X}_i)\{m(X_i)-m(x)\}  \nn \\
	& \quad +   \{n \hat{f}_{\hat{X}}(x ; \rho)\}^{-1}  \sum_{i\in\hat{\mathcal{C}}_n(\rho) } K_h(x-\hat{X}_i)  \sigma(X_i) \varepsilon_i \nn \\
	& \equiv  \hat{\Pi}_{1}(x ) + \hat{\Pi}_{2}(x ) +\hat{\Pi}_{3}(x ) , \label{hatm-m.decomposition}
\end{align}
where $w_0\equiv 1$ is as in (7.2). Our proof consists of analysing $\hat{f}_{\hat{X}}(x ;  \rho)$, $\hat{\Pi}_{1}(x )$, $\hat{\Pi}_{2}(x )$ and $\hat{\Pi}_{3}(x )$ separately.

\noindent
{\it Step 1.}
We first deal with $\hat{f}_{\hat{X}}(x ;  \rho)$. For $x\in \br$ and $\lambda >0$, define
$
	\hat{f}_X(x;\lambda )  = n^{-1} \sn K_h(x-X_i) I\{ U_i \in \hat{\mathcal{L}}_n(\lambda)  \}.
$
Analogously to (7.11), we claim that
\be
	\max_{x\in [a,b]} \big| \hat{f}_{\hat{X}}(x; \rho) - \hat{f}_X(x;\rho) \big| = O_P \{  h^{-1}\delta_{0n}(g_1)  \} = O_P  \{  h^{-1}\delta_{0n}(g_1) \} .  \label{thm3.4}
\ee
To prove \eqref{thm3.4}, recalling the notation at (7.2), we note that, for every $u\in \hat{\mathcal{L}}_n( \rho) $,
\begin{equation}
	   \left| \hat{w}_X(u) - w_0(u) \right|   = \frac{| \hat{\mu}_0 \, \varphi(u) - \hat{\varphi}_{0,\LL}(u) | }{|\hat{\varphi}_{0,\LL}(u)|}   \leq  \rho^{-1} \left\{ |\hat{\mu}_0 - \mu_0|\,M_0 + |\hat{\varphi}_{0,\LL}(u) - \varphi_0(u) | \right\} , \label{thm3.5}
\end{equation}
where we used (7.8) with the superscript ``+'' removed. This implies that
\begin{align}
	& \max_{x\in [a, b]}\left|  \hat{f}_{\hat{X}}(x; \rho ) - \hat{f}_X(x; \rho ) \right|  \nn \\
& \leq    \| K' \|_\infty  \, \rho^{-1}  \left( |\hat{\mu}_0 - \mu_0|\,M_0 + \| \hat{\varphi}_{0,\LL} - \varphi_0 \|_\infty \right)   \nn \\
&  \quad    \times (nh^2)^{-1}    \max_{x\in [a, b] }  \sn |X_i|  I\{ | X_i-x| \leq h  \, \mbox{ or } \,  | X_i \hat{w}_X(U_i) - x | \leq h \}  I\{ U_i \in \hat{\mathcal{L}}_n(\lambda)  \}.   \label{thm3.6}
\end{align}
By \eqref{thm3.5} and \eqref{thm3.6}, the arguments we used before to deal with (7.13) now lead to \eqref{thm3.4}.

Next we study $\hat{f}_X(x; \rho)$. As in the proof of part (i) of Theorem~4.1, let $\alpha\in(1/5 , 2/5)$ be a constant, and define the event
\be
	\mathcal{G}_n(\alpha) = \left\{    \| \hat{\varphi}_{0,\LL} -  \varphi_0 \|_\infty \leq n^{-\alpha}   \right\} . \label{event.G}
\ee
By (7.6) with $\ell=0$ and $g_1 \asymp (n/\log n)^{-1/5}$, we have $P\{\mathcal{G}_n(\alpha)\} \rightarrow 1 $ as $n\to \infty$. 

Let $t_0 \in (\rho, M_0)$ be such that $\mathcal{L}(t_0)$ has positive Lebesgue measure and note that, for every $u\in \mathcal{L}(t_0)$, 
$
	|\hat{\varphi}_{0,\LL}(u)|\geq |\varphi_0(u)|-\| \hat{\varphi}_{0,\LL}-\varphi_0 \|_\infty \geq t_0 -   \| \hat{\varphi}_{0,\LL}-\varphi_0 \|_\infty.
$
Hence, on the event $\mathcal{G}_n(\alpha)$ with $n$ sufficiently large so that $t_0 -n^{-\alpha} \geq \rho$, the set $\mathcal{L}(t_0)$ is contained in $\hat{\mathcal{L}}_n(\rho)  $. Therefore, for every $x\in [a,b]$, $ \hat{f}_X(x; \rho) \geq n^{-1} \sn K_h(x-X_i ) I  \{  U_i \in \mathcal{L}(t_0 ) \}$. By Theorem~6 in Hansen~(2008), we have
\be
	n^{-1} \sn K_h(x-X_i ) I \{  U_i \in \mathcal{L}(t_0 ) \} =  p_0  f_X(x)  + O_P\{ \delta_{0n}(h) \}   \nn 
\ee
uniformly over $x\in [a,b]$, where under Condition (B2),
$$
  p_0  \equiv  P \{ |\varphi(U)|\geq t_0  \} = \int_0^1 I \{ u\in \mathcal{L}(t_0) \} f_U(u) \, du \geq \inf_{u\in [0,1]} f_U(u) \cdot \lambda( \mathcal{L}(t_0)) >0.
$$
Here, $\lambda(A)$ stands for the Lebesgue measure of $A$ if $A\subset \br$ is Lebesgue measurable. The last two displays, together with \eqref{thm3.4} yield
\begin{equation}
  \hat{f}_{\hat{X}}(x;\rho) \geq p_0 f_X(x)   + O_P \{ h^{-1} \delta_{0n}(g_1)  \} =p_0 f_X(x) +o_P(1)  \label{hatfxrho.lbd}
\end{equation}
uniformly in $x\in [a,b]$, from which we conclude that with probability tending to 1, $\hat{f}_{\hat{X}}(x;\rho)$ is bounded away from zero on $[a,b]$.

\noindent
{\it Step 2.}
For the numerator of $\hat{\Pi}_{1}(x)$, by \eqref{thm3.5},  arguments similar to those leading to (7.22), (7.23) and (7.25) can be used to prove that
\begin{equation}
	  \max_{x\in [a,b]} \Big| n^{-1} \sum_{i\in \hat{\mathcal{C}}_n(\rho) }  K_h(x-\hat{X}_i) ( \hat{w}_Y-w_0)(U_i) Y_i  \Big|  = O_P \{ \delta_{0n}(g_2)  \}. \label{numerator.Pi1.rate}
\end{equation}
Together with \eqref{hatfxrho.lbd}, this gives
\begin{equation}
	\max_{x\in [a,b]} \big| \hat{\Pi}_1(x) \big| = O_P \{ \delta_{0n}(g_2) \}. \label{Pi1.rate}
\end{equation}

\noindent
{\it Step 3.}
For $\hat{\Pi}_{3}(x)$, it suffices to control the stochastic error
$$
\max_{x\in [a,b]} \bigg| n^{-1} \sn \sigma(X_i) K_h(x-\hat{X}_i )  I \{ U_i \in \hat{\mathcal{L}}_n(\rho)  \} \varepsilon_i \bigg|.
$$
Since $\hat \varphi_{0,\LL}(U_i)$ is a measurable function of $\{(X_i,U_i)\}_{i=1}^n$,  $I  \{ U_i \in \hat{\mathcal{L}}_n(\rho)  \} = I\{ | \hat{\varphi}_{0, \LL}(U_i) | \geq \rho  \}$ is independent of $\varepsilon_i$. With an argument similar to that used in (7.36)--(7.39), it can be shown that
\be
\max_{x\in [a,b]} \big| \hat{\Pi}_3(x) \big| = O_P \{ (nh/\log n)^{-1/2}  \}. \label{Pi3.rate}
\ee

\noindent
{\it Step 4.}
Turning to $\hat{\Pi}_2(x)$, observe that $\hat{f}_{\hat{X}}(x;\rho)\, \hat{\Pi}_2(x)$ can written as
\begin{align}
	& n^{-1} \sn  \{ K_h(x-\hat{X}_i) - K_h(x-X_i) \}\{m(X_i)-m(x)\}I \{ U_i \in \hat{\mathcal{L}}_n(\rho)  \}  \nn \\
	&  \quad  + n^{-1}\sn K_h(x-X_i) \{ m(X_i) -m(x) \} I \{ U_i \in \hat{\mathcal{L}}_n(\rho)  \}  \equiv   \hat{\Pi}_{21}(x) + \hat{\Pi}_{22}(x) .  \label{Pi2.dec}
\end{align}

For $\hat{\Pi}_{21}(x)$, in view of \eqref{thm3.6}, it can be proved that
\be
	\max_{x\in[a,b]} \big| \hat{\Pi}_{21}(x) \big| =   O_P \{\delta_{0n}(g_1)  \}. \label{Pi21.rate}
\ee

Next we focus on $\hat{\Pi}_{22}(x)$. Set
\be
	R_{2n}(x) =  n^{-1}\sn K_h(x-X_i) \{ m(X_i) -m(x) \} I\{ U_i \in  \mathcal{L}(\rho)  \},  \label{Pi22.def}
\ee
and note that for each $1\leq i\leq n$,
\begin{align}
\big| I \{ U_i \in \hat{\mathcal{L}}_n(\rho)  \}  - I \{ U_i \in  \mathcal{L}(\rho)  \} \big| 
 & =  \big| I  \{|\hat{\varphi}_{0,\LL}(U_i)| \geq \rho  \} - I \{  |\varphi_0(U_i)| \geq \rho \}  \big| \nn \\
 & \leq    I\{ \rho - \| \hat{\varphi}_{0,\LL} - \varphi_0 \|_\infty \leq    |\varphi_0(U_i)| \leq \rho + \| \hat{\varphi}_{0,\LL} - \varphi_0 \|_\infty \}. \nn
\end{align}
Together with \eqref{event.G}, this implies that on the event $\mathcal{G}_n$,
\begin{equation}
   \max_{x\in [a,b]} \big|   \hat{\Pi}_{22}(x) - R_{2n}(x)  \big|      \leq   \| K \|_\infty \| m' \|_\infty \max_{x\in [a,b]}  n^{-1} \sn I ( |X_i-x|\leq h, U_i \in A_n ), \label{Pi22.diff}
\end{equation}
where $A_n \equiv  \left\{ u\in I_U: \rho- n^{-\alpha} \leq |\varphi_0(u)| \leq \rho + n^{-\alpha} \right\}$. Since $\rho \in {\Theta}$,  $\varphi_0$ takes the value $\rho$ only at finitely many points $u_1,\ldots, u_q$ for some $q\geq 1$ which is independent of $n$. For each $u_k$, by the inverse function theorem, there exits a neighbourhood $B_k$ of $u_k$ such that $\varphi_0$ is invertible on $B_k$ and the inverse, denoted by $\varphi_{k,{\rm inv}}$, is continuously differentiable with $\varphi_{k, {\rm inv}}'(t)=1/\varphi'_0\{\varphi_{k,{\rm inv}}(t)\}$ for $t\in R_k=\{\varphi_0(u):  u\in B_k\}$. By Condition (C2), $ \min_{1\leq k\leq q}|\varphi'_0(u_k)|>0$ and hence $b_0\equiv  \max_{1\leq k\leq q}\max_{t\in R_k} |\varphi_{k,{\rm inv}}'(t)|<\infty$. Since $\rho=\varphi_0(u_k)$ is an interior point of $R_k$, for all sufficiently large $n$, $[\rho-n^{-\alpha}, \rho+n^{-\alpha}] \subseteq \bigcap_{k=1}^q R_k$ and $A_n^+ \equiv   \left\{ u\in I_U: \rho- n^{-\alpha} \leq \varphi_0(u) \leq \rho + n^{-\alpha} \right\} = \bigcup_{k=1}^q \varphi_{k,{\rm inv}}\left( [\rho-n^{-\alpha}, \rho+n^{-\alpha}] \right)$. For every $u\in A_n^+$, there exists some $1\leq k\leq q$ and $t\in [\rho-n^{-\alpha}, \rho+n^{-\alpha}]$ such that $u=\varphi_{k,{\rm inv}}(t)$ and 
$$
	|u - u_k | = |\varphi_{k,{\rm inv}}(t) - \varphi_{k,{\rm inv}}(\rho)| \leq b_0 |t-\rho | \leq b_0  n^{-\alpha}.
$$
This implies that for all sufficiently large $n$,
\be
	A^+_n \subseteq B_n  \equiv \bigcup_{k=1}^q [ u_k -  b_0 n^{-\alpha  }, u_k +  b_0 n^{-\alpha } ] \cap I_U. \label{set.An}
\ee
The set $A_n^- =\{ u\in I_U: -\rho- n^{-\alpha} \leq  \varphi_0(u)  \leq  -\rho + n^{-\alpha} \} $ can be dealt with in a similar way, so that \eqref{set.An} also holds with $A^+_n$ replaced by $A_n$.

To further bound the right-hand side of \eqref{Pi22.diff}, we define
$$
	D_n = \sup_{-\infty < x_1\leq x_2 < \infty \atop u_1, u_2\in I_U: u_1\leq u_2} \bigg| n^{-1}\sn I ( X_i \in [x_1,x_2], U_i \in [u_1,u_2] ) - P(x_1\leq X \leq x_2) P(u_1\leq U \leq u_2) \bigg| . 
$$
Recall that $X$ and $U$ are independent. Applying Inequality~1, Inequality 2 with $s=4$ and Example~1 with $d=2$ in Chapter 26 of Shorack and Wellner~(1986) to $D_n$, we obtain that, for every $t>0$ and $n\geq \max( 2/t^2, 3)$, 
$$
	P( D_n \geq t ) \leq 4 \,\cdot 3 (2n)^{4} \exp(-nt^2/8)/(2\cdot 4!) = 4 n^4 \exp(-nt^2/8).
$$
Taking $t=C n^{-1/2} \sqrt{\log n}$ with $C>4\sqrt{2}$, it follows that $D_n=O_P( n^{-1/2} \sqrt{\log n} )$, which, together with \eqref{set.An} yields
\begin{equation}
	   n^{-1}\sn I ( |X_i-x|\leq h , U_i \in B_n  )   =  P ( |X-x|\leq h  ) P ( U_i \in B_n   ) + O_P\{   n^{-1/2} (\log n)^{1/2} \} \nn
\end{equation}
uniformly over $x\in [a,b]$. Under Condition~(B3),
$\sup_{x\in \br}P ( |X-x|\leq h  ) \leq \con \, h$ and by \eqref{set.An} and Condition~(B2), $P ( U_i \in B_n ) \leq \sum_{k=1}^q P ( |U_i-u_k| \leq b_0 n^{-\alpha}  ) \leq \con \, n^{-\alpha}$. This, together with the last display implies
\begin{equation}
	  n^{-1}\sn I( |X_i-x|\leq h , U_i \in B_n  )  =   O_P\{ h n^{-\alpha }+ n^{ -1/2} (\log n)^{1/2} \}  \label{joint.consistency}
\end{equation}
uniformly over $x\in [a,b]$. Combining \eqref{Pi22.diff}, \eqref{set.An} and \eqref{joint.consistency}, we deduce that
\be
	\max_{x\in [a,b]} \big|      \hat{\Pi}_{22}(x) - R_{2n}(x)  \big| =  O_P\{ h n^{-\alpha} +n^{ -1/2} (\log n)^{1/2} \}. \label{Pi22.approxi}
\ee
By studying the mean and the variance of $R_{2n}(x)$ separately, it can be proved, using arguments similar to those used in (7.29)--(7.34) that $\max_{x\in [a,b]} | R_{2n}(x)   |  = O_P(h^2)$. Since $\alpha>\alpha_0$, by Condition (B6) we have $n^{-\alpha}=o(h)$. Then it follows from \eqref{hatfxrho.lbd}, \eqref{Pi2.dec}, \eqref{Pi21.rate} and \eqref{Pi22.approxi} that
\be
 \max_{x\in [a,b]} \big|  \hat{\Pi}_{2}(x)   \big|  = O_P \{ \delta_{0n}(g_1) + h^2 \}  . \label{Pi2.rate}
\ee

\noindent
{\it Step 5.}
The proof of the theorem is completed by combining \eqref{Pi1.rate}, \eqref{Pi3.rate} and \eqref{Pi2.rate}.

\medskip
\noi
{\it Proof of (ii).}\\
Finally we sketch the proof for the asymptotic normality of $\hat{m}(x; \rho)$, which is based on a straightforward adaptation of the arguments we used in the proof of Theorem~4.1.

Define $\hat{I}_{i} = I\{U_i  \in \hat{\mathcal{L}}_n(\rho) \}$ and $I_i = I\{U_i  \in {\mathcal{L}}(\rho)  \}$ for $i=1,\ldots, n$. In this notation, we have $\hat{f}_{\hat{X}}(x; \rho) = n^{-1} \sn K_h(x-\hat{X}_i) \hat{I}_i$. For any $\alpha\in (\alpha_0, 2/5)$ fixed, it can be proved that
\be 
	  | \hat{f}_{\hat{X}}(x; \rho) -  \tilde{f}_X(x; \rho) | = O_P\{ h^{-1} n^{-\alpha } + h^{-1} \delta_{0n}(g_1) \} = o_P(1) , \label{hatf.consistency}
\ee
where $\tilde f_X(x; \rho)= n^{-1} \sn K_h(x-X_i) I_i$. Similarly, it can be proved that
\begin{align}
 	\Big| n^{-1} \sn K_h(x-\hat X_i)(\hat{w}_Y - w_0)(U_i) (\hat{I}_i - I_i) Y_i \Big| = O_P \{ h^{-1} n^{-\alpha}\delta_{0n}(g_2)  \}, \nn \\
 	\Big| n^{-1} \sn K_h(x-\hat X_i) \{ m(X_i)  - m(x)\} (\hat{I}_i - I_i)   \Big| = O_P( n^{-\alpha} ). \nn
\end{align}
For the noise term $n^{-1} \sn K_h(x-\hat X_i) \sigma(X_i) \hat{I}_i \, \varepsilon_i$, consider the difference 
\begin{align}
    D_1 \equiv  n^{-1} \sn  K_h(x-\hat{X}_i)\sigma(X_i) (\hat{I}_i - I_i  ) \varepsilon_i.  \nn
\end{align}
Applying a conditional version of Lemma~\ref{Bernstein.inequality} implies that, for every $t\geq 0$, 
\begin{align}
	P \left\{  | D_1 | \geq   n^{-1} ( \hat V_1 \sqrt{ t} +   \hat V_2 t  )  \bigg| \{ (X_i, U_i) \}_{i=1}^n \right\} \leq 2\exp(-c t), \nn
\end{align}
where $\hat V_1^2 = \sn \sigma^2(X_i) K^2_h(x-\hat{X}_i) (\hat{I}_i - I_i )^2  \leq \| \sigma \|_\infty^2  \| K \|_\infty^2 \, h^{-2}  \sn  (\hat{I}_i - I_i )^2$ and $\hat V_2 = \max_{1\leq i\leq n} |\sigma(X_i) K_h(x-\hat{X}_i) (\hat{I}_i - I_i )| \leq \| \sigma \|_\infty  \| K \|_\infty \, h^{-1}$. It thus follows from \eqref{Pi22.diff} and \eqref{set.An} that on the event $\mathcal{G}_n(\alpha)$, for the above $\alpha \in (\alpha_0 , 2/5)$,
\begin{align}
 \hat V_1^2 \leq  \| \sigma \|_\infty^2  \| K \|_\infty^2 \, h^{-2}  \sn I(U_i \in B_n) \equiv V_1^2 = O_P\left\{  h^{-2} n^{1-\alpha} + (n \log n)^{1/2} \right\} . \nn
\end{align}
Hence, it follows immediately that $	|D_1| =  o_P\{(nh)^{-1/2}\}$. Moreover, as in the proof of Theorem~4.1, it can be shown that
\begin{align}
 \Big|  n^{-1} \sn \{ K_h(x-\hat{X}_i) - K_h(x- X_i) \}  \sigma(X_i)  I_i \, \varepsilon_i \Big| = o_P\{(nh)^{-1/2}\}. \nn
\end{align}

Putting the above calculations together gives $|\hat m(x;\rho) - m(x) -  \hat{B}(x;\rho)| = o_P\{  h^2+ (nh)^{-1/2} \}$, where $\hat{B}(x;\rho)$ is such that
\begin{align}
	& \hat{B}(x;\rho) \tilde f_X(x;\rho) \nn \\
	& =  n^{-1}\sn K_h(x-\hat{X}_i)  (\hat{w}_Y - w_0)(U_i)  I_i \, Y_i \nn \\
	& \quad +  n^{-1}\sn K_h(x- \hat{X}_i)  \{m(X_i) - m(x)\} I_i  +  n^{-1}\sn K_h(x- {X}_i) \sigma(X_i) I_i \, \varepsilon_i . \label{tildeB.dec}
\end{align}
Using a slightly modified form of Lemma~\ref{lm.1}, we have
\begin{align}
  & n^{-1}\sn K_h(x-\hat{X}_i)  (\hat{w}_Y - w_0)(U_i)  I_i \, Y_i  \nn \\
  & = - m(x) f_X(x) E[ \psi''(U) I\{U\in \mathcal{L}(\rho)\}/\psi(U) ] \mu_{L,2} \, g_2^2/2 + o_P(g_2^{2}) + O_P(n^{-1/2}). \label{tildeB.ubd1}
\end{align}
For the second term on the right-hand side of \eqref{tildeB.dec}, an expansion similar to that in (7.50) holds, i.e.,
\begin{align}
  & n^{-1}\sn K_h(x- \hat{X}_i)  \{m(X_i) - m(x)\} I_i \nn \\
  & = n^{-1}\sn K_h(x- X_i)  \{m(X_i) - m(x)\} I_i  \nn \\
  & \quad  + x m'(x) E[ \varphi''(U) I\{U\in \mathcal{L}(\rho)\} /\varphi(U) ] \mu_{L,2} \, g_1^2/2 +o_P\{ g_1^2 + h^2+ (nh)^{-1/2} \}.\label{tildeB.ubd2}
\end{align}

Finally, combining \eqref{tildeB.dec}--\eqref{tildeB.ubd2} with \eqref{hatf.consistency} proves (4.3) via a standard argument. \qed

\section{Other technical results}

\subsection{Auxiliary lemmas}  

\noindent
Here we provide proofs for the remaining technical results of the paper.
 
\begin{lemma} \label{MRS.lemma.1}
Assume that Conditions (B1)--(B6) are fulfilled, and let $\beta_1=(1+\xi_0)/5$ for some $\xi_0\geq 0$. Then for any $\alpha\in (\alpha_0, 1/2-\beta_1/2 )$ and $\kappa_1\in ( 0 , 1/2 + 3\alpha/4  - 3\alpha_0/2- \xi_0/8)$,
\begin{align}
 & \sup_{x\in [a,b] \atop w_1, w_2 \in \mathcal{N}_{0n}} \bigg\| n^{-1} \sn  K_h\{ X_i  w_1(U_i) -x \} \mathbf{v}_i(x ; w_1) \sigma(X_i)\varepsilon_i      \nn \\
 & \qquad \qquad  \qquad \quad   - n^{-1} \sn K_h\{ X_i w_2(U_i)- x \} \mathbf{v}_i(x; w_2)\sigma(X_i) \varepsilon_i  \bigg\|_\infty =O_P (n^{- \kappa_1}  ) , \label{dif.1} 
\end{align}
where $\mathcal{N}_{0n}$ is as in (7.49) and for every function $w: \br \mapsto \br$,
$$
	\mathbf{v}_i(x; w) = \left( 1,  \{ X_i w(U_i) - x \}/h \right)\T \in \br^2.
$$

\end{lemma}
 
\begin{proof}[Proof of Lemma \ref{MRS.lemma.1}] 

We prove \eqref{dif.1} in two steps.\\[.3cm]
\noindent
{\it Step 1.} In the first step, we use Lemma~1 of Mammen et al.~(2012). There, they observe data $S_1,\ldots,S_n$, and are interested in a function $r_0$ which is supported on a compact interval $I_R$. Moreover, an estimator $\hat r$ of $r_0$ is available, and they assume that there exists a set $\mathcal{M}_n$ which is such that $\lim_{n\to\infty}P(\hat r\in \mathcal{M}_n)=1$ and that for every $\epsilon \in ( 0 , n^{-\delta} ]$, the set 
\begin{equation}
\bar{\mathcal{M}}_n=\mathcal{M}_n \cap \left\{  r: \|r-r_n\|_\infty\leq n^{-\delta} \right\}	\label{eqA2}
\end{equation}
can be covered by at most $C \exp(\epsilon^{-\gamma }n^\xi)$ balls with radius $\epsilon$, where $\delta, \gamma, \xi >0 $ and $r_n$ is such that $\|r_n-r_0\|_\infty=o(n^{-\delta})$.  Under this assumption, which they refer to as the complexity assumption, the moment condition that $ E\{\exp(c|\varepsilon|)| S \} \leq C $ almost surely for some constants $C,c >0$, and when the bandwidth $h$ satisfies 
\begin{equation}
h = h_n \asymp n^{-\eta}				\label{eqA3}
\end{equation}
for some $\eta \in (0, \delta )$, their Lemma~1 states that, for any $\kappa_1 \in \left( 0, \frac{1}{2} -\frac{3}{2}\eta   +(1-\frac{1}{2}\gamma ) \delta - \frac{1}{2} \xi  \right)$,
\begin{align}
 & \sup_{x\in I_R \atop r_1, r_2 \in \bar{\mathcal{M}}_n} \bigg\| n^{-1} \sn  K_h\{ r_1(S_i) -x \} \mathbf{w}_i(x ; w_1)  \varepsilon_i      \nn \\
 & \qquad \qquad  \qquad \quad   - n^{-1} \sn K_h\{ r_2(S_i)- x \} \mathbf{w}_i(x; w_2) \varepsilon_i  \bigg\|_\infty =O_P  (n^{- \kappa_1} ) ,  \label{eqA4}
\end{align}
where $	\mathbf{w}_i(x; r) = \left( 1,  \{ r(S_i) - x \}/h \right)\T$. 

To apply their lemma to our problem, we need to verify that the above assumptions hold in our context, where  $S_i=(X_i, U_i)^T$, for $i=1,\ldots,n$, are i.i.d. and for an appropriate form of the function $r$. We shall show below that a suitable choice of $r$ is
\be 
 r(s)=r(x,u)=  \begin{cases}
   x w(u)  \quad & \mbox{ if }  a_1\leq x\leq b_1 ,   \\
   b_1w(u)   \quad &      \mbox{ if }  x > b_1 ,   \\
   a_1 w(u) \quad &  \mbox{ if }  x< a_1,  
   \end{cases}       \label{eqA5}
\ee
where $w$ is a function from $[0,1]$ to $\br$ and $a_1=a-1, b_1 =b+1$, with $a$ and $b$ as in the statement of Theorem~4.1.

Next we verify that  the assumptions used by Mammen et al.~(2012) are satisfied in our context. First, condition \eqref{eqA3} holds in our case with $\eta = \alpha_0$. Second, to verify the complexity assumption, we need to carefully construct a set $\bar{\mathcal{M}}_n$ appropriate for our setting. Indeed, if we simply define the set $\bar{\mathcal{M}}_n$ as in Mammen et al.'s~(2012) discussion of their result for kernel smoothers with $S_i\in\br^2$, then we will obtain a suboptimal bound, because in our case, although the function $r$ is bivariate, it has the particular structure \eqref{eqA5}, where only the univariate function $w$ is unknown. Therefore, we need to define the set $\bar{\mathcal{M}}_n$ carefully enough to obtain sharp rates of convergence that correspond to the univariate case in Mammen et al.~(2012). We take
\be
   \bar{\mathcal{M}}_n \equiv    \Big\{ r :  \br \times [0,1] \mapsto \br \mbox{ is of the form \eqref{eqA5}} \Big|  w \in \mathcal{N}_{0n} \Big\}  , \label{eqA6}
\ee
where $\mathcal{N}_{0n}$ is defined in (7.49).

Next, we verify the complexity assumption of Mammen et al.~(2012) provided three paragraphs above this one. For $w_0$ as in (7.2), let $r_0$ be the function from $\br \times [0,1]$ to $ \br$ given by \eqref{eqA5} with $w$ replaced by $w_0$. By \eqref{eqA5}, for every $r, \tilde{r} \in \mathcal{M}_n$, there exist $w, \tilde{w} \in \mathcal{N}_{0n}$ such that
\begin{align}
 \| r - \tilde{r} \|_\infty = \max_{(x,u)\in \br \times [0,1]} |r(x,u) - \tilde{r}(x,u) | \leq  A_1 \cdot \| w - \tilde{w} \|_\infty , \label{eqA7}
\end{align}
where $A_1 \equiv \max(1, |a_1|, |b_1|) \geq 1$.

Recalling the definition of $\mathcal{N}_{0n}$ in (7.49), it follows from Theorem~2.7.1 in van der Vaart and Wellner~(1996) that, for every $ \epsilon \in ( 0, n^{-\alpha}]$ with $\alpha\in (\alpha_0,  1/2 - \beta_1 /2 )$, the set $\mathcal{N}_{0n}$ can be covered by at most $\exp\{C\, \epsilon^{-1/2} n^{\xi_0/4} (\log n)^{1/4}\}$ balls with $\| \cdot \|_\infty$-radius $\epsilon$, where $C >0$ is an absolute constant. This and \eqref{eqA7} jointly imply that, for every $\epsilon \in ( 0,  n^{-\alpha}]$, the set $\bar{\mathcal{M}}_n$ of functions $\br \times [0,1] \mapsto \br$ can be covered by at most $\exp\{C_1 \, \epsilon^{-1/2} n^{\xi_0/4} (\log n)^{1/4}\}$ balls with $\| \cdot \|_\infty$-radius $\epsilon $, where $C_1 = C A_1^{1/2}$. Therefore, our space  satisfies the complexity assumption with $\delta=\alpha$, $\gamma =\frac{1}{2}$, and $\xi= \frac{1}{4}\xi_0$. Further, by Condition (B6) and the boundedness assumption on $\sigma$ in (B5), all the conditions of Lemma~1 in Mammen et al.~(2012) are satisfied, so that \eqref{eqA4} holds in our case too. Specifically,  for any $\kappa_1 \in \left( 0, \frac{1}{2} - \frac{3}{2}\alpha_0 + \frac{3}{4} \alpha - \frac{1}{8}\xi_0 \right)$, we have
\begin{align}
 &\max_{\tilde{K}\in \{ K, \bar{K} \} } \sup_{x\in [a,b] \atop  r_1, r_2 \in \bar{\mathcal{M}}_n }  \bigg\|  n^{-1} \sn  \tilde{K}_h\{ r_1(X_i,U_i) -x \}  \sigma(X_i)\varepsilon_i      \nn \\
 & \qquad \qquad \qquad \qquad  \qquad \qquad   - n^{-1} \sn \tilde{K}_h\{ r_2(X_i,U_i) - x \}  \sigma(X_i) \varepsilon_i  \bigg\|_\infty =O_P (n^{- \kappa_1} ), \label{eqA8} 
\end{align}
for $\bar{K}(t)\equiv  tK(t)$.    \\[.1cm]

\noindent
{\it Step 2.}
Denote by $\hat{\Delta}_{0n}$ and $\hat{\Delta}_{1n}$ the left-hand side of \eqref{dif.1} and \eqref{eqA8}, respectively. In what follows, we shall prove that  the sequence of events $\mathcal{A}_n=\mathcal{E}_{1n}(\lambda)$, with $\mathcal{E}_{1n}(\lambda)$ as in (7.15) and $\lambda=2c_1^{-1}$ for $c_1 >0$ as in Condition (B3), satisfies (i) $P(\mathcal{A}_n) \to 1$ as $n\rightarrow \infty$, and (ii) $\hat{\Delta}_{0n} \leq \hat{\Delta}_{1n}$ on $\mathcal{A}_n$ for all sufficiently large $n$. Conclusion \eqref{dif.1} follows immediately.

To prove (i), note that under Condition (B3), we have
$$
	P ( \mathcal{A}_n^{{\rm c}}  )  = P\left(  \max_{1\leq i\leq n}|X_i|> \lambda \log n \right)  \leq \sn P ( |X_i|>\lambda \log n  ) \leq  K_1 \, n^{1-c_1\lambda},
$$ 
where $K_1=E \{ \exp(c_1|X|) \} $. For $\lambda=2c_1^{-1}$, it follows that $P (\mathcal{A}_n^{{\rm c}} ) \to 0$ as $n\to\infty$. 

To prove (ii), for $x\in [a,b]$, $w_1, w_2 \in \mathcal{N}_{0n}$ and $\tilde{K}= K$ or $\bar{K}$, consider the difference
\begin{align}
 & n^{-1} \sn  \left[ \tilde{K}_h\{ X_i w_1(U_i) -x \} - \tilde{K}_h\{ X_i w_2(U_i)- x \} \right] \sigma(X_i) \varepsilon_i  \nn  \\ 
  & = n^{-1} \sn  \left[ \tilde{K}_h\{ X_i  w_1(U_i) -x \} - \tilde{K}_h\{ X_i w_2(U_i)- x \} \right]  \sigma(X_i) \varepsilon_i \nn \\
  & \qquad \qquad \qquad \qquad \qquad \times  I \{ |X_iw_1(U_i)-x| \leq h \, \mbox{ or } \,  |X_iw_2(U_i)-x| \leq h  \}\,. \label{eqA9}
\end{align}
For every $w\in \mathcal{N}_{0n}$, we have
\begin{align}
 & 	|X_i - x | = |X_i w_0(U_i) -x | \nn \\
 & \leq |X_iw_0(U_i ) - X_i w(U_i) | + |X_i w(U_i) - x| \leq |X_i | n^{-\alpha} +  |X_i w(U_i) - x|\,,  \nn
\end{align}
which implies that on the event $\mathcal{A}_n \cap \{|X_i w(U_i)-x| \leq h \} $ with $n$ sufficiently large, $|X_i-x| \leq \lambda n^{-\alpha} \log n + h \leq 1$ and hence
\be
	  \mathcal{A}_n\cap \left\{ |X_i w(U_i)-x| \leq h \right\}  \subseteq  \left\{  X_i \in [a_1, b_1] \right\} \label{eqA10}
\ee
uniformly over $w \in \mathcal{N}_{0n}$. Together, \eqref{eqA9} and \eqref{eqA10} prove (ii) and thus complete the proof of Lemma~\ref{MRS.lemma.1}.
\end{proof}

The following lemma is Proposition~5.16 in Vershynin~(2012) with a slight modification which gives a Bernstein-type inequality for sums of sub-exponential random variables.

\begin{lemma} \label{Bernstein.inequality}
Let $\varepsilon_1, \ldots, \varepsilon_n$ be i.i.d. centered sub-exponential random variables satisfying $E\{ \exp(c_1|\varepsilon_i|) \}<C_1$ for some constants $c_1, C_1>0$. Then for every $\mathbf{a}=(a_1,\ldots, a_n)\T \in \br^n$ and every $y\geq 0$,
\begin{align*}
P\left\{ \left| \sn a_i \varepsilon_i \right| \geq \max\left(  \| \mathbf{a} \|_2 \sqrt{y} ,   \| \mathbf{a} \|_\infty y \right) \right\} \leq 2\exp(-c y),
\end{align*}
where $ c>0$ is a constant only depending on $c_1$ and $C_1$.
\end{lemma}

\begin{proof}[Proof of Lemma~\ref{Bernstein.inequality}] The conclusion follows directly by taking $t=K \max\{\|  \mathbf{a} \|_2 \sqrt{y} , \|  \mathbf{a} \|_\infty y\}$ in Proposition~5.16 of Vershynin~(2012) so that $\min\{ t^2/ (K \|  \mathbf{a} \|_2)^2  ,t/(K \|  \mathbf{a} \|_\infty ) \} \geq  y$.
\end{proof}

\begin{lemma} \label{lm.0}
{\rm
Let $w_0$ and $\hat{w}_X$ be as in (7.2). If Conditions (2.2) and (B1)--(B6) are satisfied, then for every $x\in [a, b] \subseteq I_X\equiv  \{z: f_X(z)>0\}$,
\begin{align}
	&  n^{-1} \sn K_h(X_i-x)(w_0 - \hat{w}_X)(U_i)\, X_i \nn \\
	& \qquad \qquad \qquad \qquad =  \tfrac{1}{2}x f_X(x)  E \{ \varphi''(U)/\varphi(U) \}   \mu_{L,2}   \,  g_1^2  + o_P ( g_1^2 ) + O_P ( n^{-1/2}  ) \label{stochastic.order.1}
\end{align}
and
\begin{align}
 & (nh)^{-1} \sn K'_h(x-X_i) (w_0 - \hat{w}_X )(U_i)\, X_i \{ m(X_i) -m(x) \} \nn \\
&  \qquad \qquad \qquad  = \tfrac{1}{2}x m'(x) f_X(x)  E \{ \varphi''(U)/\varphi(U) \}   \mu_{L,2}   \,  g_1^2 +  o_P ( g_1^2  )  + O_P( n^{-1/2}  ).   \label{stochastic.order.1'}
\end{align}
}
\end{lemma}

\begin{proof}[Proof of Lemma~\ref{lm.0}] For the sake of clarity we provide the detailed proof only for \eqref{stochastic.order.1}, however, in {\it Step 3} below we shall mention changes that have to be made to prove \eqref{stochastic.order.1'}.

\medskip
\noindent
{\it Step 1. Preliminaries}. First, recall that $\varphi_0^+= \mu_0^+ \, \varphi$, where $\mu^+_0=E|X|$. A direct consequence of (7.2), (7.6) and the fact $\hat \mu_0^+ =\mu_0^+ + O_P(n^{-1/2})$ is that
$$
\hat w_X(u)
=  \tilde w_X(u)+O_P(n^{-1/2})
$$
uniformly over $u\in [0,1]$, where $\tilde{w}_X(u)\equiv  \{ \hat \varphi^+_{0,\LL}(u) \}^{-1} \varphi^+_0(u) $ satisfies
\begin{align}
	w_0(u)- \tilde w_X(u)   = \frac{\hat \varphi^+_{0,\LL}(u)- \varphi_0^+(u)}{\varphi_0^+(u)} +O_P \{ \delta^2_{0n}(g_1) \}   \label{lm.0.1}
\end{align}
for $\delta_{0n}$ as in (7.7).
Recall that $\hat{\varphi}_{0,\LL}^+(u)$ defined below (3.3) is the local linear estimator of $\varphi_0^+(u) =E( |\tilde{X}| \, |U=u) $. We rewrite $\hat{\varphi}^+_{0,\LL}$~as
\begin{align*}
	 \hat \varphi^+_{0,\LL}(u)  =  \mathbf{e}_1\T \mathbf S^{-1}_n(u) \,  n^{-1} \sn  L_{g_1}(U_i-u)  \mathbf{v}_i(u)  | \tilde X_i | ,
\end{align*}
where $\mathbf{v}_i(u) = (1, g_1^{-1}(U_i-u))\T \in \br^2$ and $\mathbf S_n(u)  = n^{-1} \sn  L_{g_1}(U_i-u) \mathbf{v}_i(u) \mathbf{v}_i(u)\T \in \br^{2\times 2}$. Further, define $\eta_i=|\tilde{X}_i| -E( | \tilde{X}_i| \, | U_i)=\varphi(U_i)( |X_i| -\mu^+_0)$, such that, using (2.1) and (2.4), $E(\eta_i| U_i)=0$ and
$$
	| \tilde{X}_i| = \varphi_0^+(U_i) + \eta_i.
$$

Similarly to (A.1) and (A.2) in Delaigle et al.~(2009), $\hat{\varphi}_{0,\LL}^+(u)-\varphi_0^+(u)$ can be written as $ \hat{\Gamma}_1(u)+ \hat{\Gamma}_2(u)$, where
\begin{align}
	& \hat{\Gamma}_1(u) \equiv   \mathbf{e}_1\T \mathbf S^{-1}_n(u) \,  n^{-1} \sn  L_{g_1}(U_i-u)  \mathbf{v}_i(u) \left\{ \varphi_0^+(U_i) - \varphi^+_0(u) - (\varphi_0^+)'(u)(U_i-u) \right\} \label{LL.bias}
\end{align}
and
\begin{align}
	&  \hat{\Gamma}_2(u) \equiv  \mathbf{e}_1\T \mathbf S^{-1}_n(u) \,  n^{-1} \sn  L_{g_1}(U_i-u)  \mathbf{v}_i(u) \eta_i . \label{LL.noise}
\end{align}
For $\mathbf{S}_n(u)$ in \eqref{LL.bias} and \eqref{LL.noise}, let
\begin{align}
\mathbf{S}_{n}^*(u)  & = E \{ \mathbf{S}_n(u) \} = 
\left(
\begin{array}{ll}
  \int   L(v) f_U(u+ g_1 v) \, dv      &   \int  v L(v) f_U(u+ g_1 v) \, dv \\
  \int  v L(v) f_U(u+ g_1 v) \, dv &  \int  v^2 L(v) f_U(u+ g_1 v) \, dv
\end{array}
\right).  \label{S*n.definition}
\end{align}
Standard arguments show that $\mathbf{S}_{n}^*(u) = \mathbf{S}^*(u) + O(g_1)$ uniformly in $u\in [0,1]$, where $\mathbf{S}^*(u)  = f_U(u) \, \mbox{diag}\,(\mu_{L,0},\mu_{L,2})$ if $u\in (g_1,1-g_1)$ and for $0\leq \rho \leq 1$, $\mathbf{S}^*(u) = f_U(0^+) \, \mathbf{M}^+_{L}(\rho)$ if $u=\rho g_1$ and $\mathbf{S}^*(u) = f_U(1^-) \, \mathbf{M}^-_L(\rho)$ if $u=1-\rho g_1$, where $ f_U(0^+) \equiv  \lim_{u>0,u\to 0} f_U(u)$, $f_U(1^-)\equiv \lim_{u<1, u\to 1}f_U(u)$,
\beq 
 \mathbf{M}^+_L(\rho) \equiv    \left(
\begin{array}{ll}
     \mu^+_{L,0}(\rho )      &  \mu^+_{L,1}(\rho ) \\
  \mu^+_{L,1}(\rho ) &  \mu^+_{L,2}(\rho ) 
\end{array}
\right)   \ \ \mbox{ and }  \   \
 \mathbf{M}^-_L(\rho) \equiv   
  \left(
\begin{array}{ll}
     \mu^-_{L,0}(\rho )      &   \mu^-_{L,1}(\rho ) \\
  \mu^-_{L,1}(\rho ) &  \mu^-_{L,2}(\rho ) 
\end{array} \right)
\eeq
with $\mu^+_{L,\ell}(\rho)=  \int_{-\rho}^1 v^{\ell}L(v)\, dv$ and $ \mu^-_{L,\ell}(\rho) =\int_{-1}^\rho v^{\ell}L(v)\, dv$. By the symmetry of $L$, $\mu^+_{L,\ell}(\rho)=\mu^-_{L,\ell}(\rho)$ if $\ell$ is even, and $\mu^+_{L,\ell}(\rho)=-\mu^+_{L,\ell}(\rho)$ is $\ell$ is odd. In particular, $\mu^+_{L,\ell}(0)=\mu^-_{L,\ell}(0)=1/2$ and
\be
	\mathbf{M}^+_L(1)=\mathbf{M}^-_L(1)=\mbox{diag}\,(\mu_{L,0},\mu_{L,2}) \equiv \mathbf{M}_L.   \label{M.definition}
\ee

Under Condition (B5), using similar techniques as in the paragraph after \eqref{Tn5.uniform.rate} it can be proved that, as $n\to\infty$, 
\be 
\sup_{u\in [0,1]} \left\| \mathbf S^{-1}_n(u) - \{ \mathbf{S}_{n}^*(u) \}^{-1}  \right\| = O_P \{  (ng_1 )^{-1/2} (\log n)^{1/2} \}. \label{Sn.consistency}
\ee
Here, the additional logarithmic term comes from the lattice argument that can be used to deal with the maximization over $u\in [0,1]$. See, for example, (7.30)--(7.33).

Next we study $\hat{\Gamma}_1(u)$. For $i=1, \ldots ,n$, put 
\begin{align}
	\mathbf{G}_i(u) & = (G_{1,i}(u), G_{2,i}(u))\T \nn \\
	&  =  L_{g_1}(U_i-u)  \left\{  \varphi_0^+(U_i) - \varphi^+_0(u) -(\varphi_0^+)'(u)(U_i-u) \right\} \mathbf{v}_i(u).  \label{G.definition}
\end{align}
Under Condition (B2), the inequality $|\varphi^+_0(U_i)-\varphi^+_0(u) - (\varphi_0^+)'(u)(U_i-u) | I(|U_i-u|\leq g_1) \leq \frac{1}{2} \mu_0^+ \| \varphi'' \|_\infty \, g_1^2 I(|U_i-u|\leq g_1)$ holds almost surely. Then the argument leading to (7.20) can be used to prove that
\beq
  \max_{u\in [0,1]} \left\| n^{-1} \sn \mathbf{G}_i(u) \right\|_\infty = O_P( g_1^2 ).
\eeq
Together with \eqref{Sn.consistency}, this implies
\begin{align}
  \hat{\Gamma}_1(u)     =  \mathbf{e}_1\T \{ \mathbf S^*_n(u) \}^{-1} \,  n^{-1} \sn \mathbf{G}_i(u) + O_P \{ g_1^2  (ng/\log n)^{-1/2}  \}  = {\Gamma}_1(u) + o_P( g_1^2 )  \label{Gamma1.equiv}
\end{align}
uniformly over $u\in [0,1]$, where ${\Gamma}_1(u) \equiv  \mathbf{e}_1\T  \{ \mathbf S^*_n(u) \}^{-1} \,  n^{-1} \sn \mathbf{G}_i(u)$.

For $\hat{\Gamma}_2(u)$, first applying Theorem~2 in Hansen~(2008) to each of the two components of $n^{-1} \sn  L_{g_1}(U_i-u)  \mathbf{v}_i(u) \eta_i $ yields
\be
	\max_{u\in [0,1]} \left\|  n^{-1} \sn  L_{g_1}(U_i-u)  \mathbf{v}_i(u) \eta_i  \right\|_\infty  = O_P \{ (ng_1/\log n)^{-1/2}  \}. \label{LL.noise.uniform.rate}
\ee
Hence, under Condition (B6), combining \eqref{Sn.consistency} and \eqref{LL.noise.uniform.rate} leads to
\begin{align}
\hat{\Gamma}_2(u) &     =  \mathbf{e}_1\T\{  \mathbf S^*_n(u) \}^{-1} \,  n^{-1} \sn  L_{g_1}(U_i-u)  \mathbf{v}_i(u) \eta_i + O_P \{  ( ng_1 /\log n)^{-1} \}  \nn \\
  & = {\Gamma}_2(u) + o_P(g_1^2),  \label{Gamma2.equiv}
\end{align}
where ${\Gamma}_2(u) \equiv \mathbf{e}_1\T  \{ \mathbf S^*_n(u) \}^{-1} \,  n^{-1} \sn  L_{g_1}(U_i-u)  \mathbf{v}_i(u) \varphi(U_i)( | X_i| - \mu^+_0)$.

Putting \eqref{LL.bias}, \eqref{LL.noise}, \eqref{Gamma1.equiv} and \eqref{Gamma2.equiv} together, we obtain that
\begin{align}
	& \hat{\varphi}^+_{0,\LL}(u) -\varphi^+_0(u)  \nn \\
	& =  \Gamma_1(u) + \Gamma_2(u) + o_P(g_1^2) \nn \\
	& =  \mathbf{e}_1\T  \{ \mathbf S^*_n(u) \}^{-1}  \,  n^{-1} \sn E \{ \mathbf{G}_i(u) \} + \mathbf{e}_1\T \{ \mathbf S^*_n(u) \}^{-1} \,  n^{-1} \sn \left[ \mathbf{G}_i(u)- E \{ \mathbf{G}_i(u) \} \right]  \nn \\
	& \quad +   \mathbf{e}_1\T \{ \mathbf S^*_n(u) \}^{-1} \,  n^{-1} \sn  L_{g_1}(U_i-u)  \mathbf{v}_i(u) \varphi(U_i)( | X_i| - \mu^+_0)+ o_P ( g_1^2 ) \nn \\
	& \equiv  \Gamma_{11}(u) + \Gamma_{12}(u) + \Gamma_2(u) + o_P( g_1^2 )		\label{ll.dec.1}
\end{align}
uniformly over $u\in [0,1]$. In what follows, we study the first two terms on the right-hand side of \eqref{ll.dec.1}.

For $\Gamma_{11}(u)$, a standard argument based on Taylor expansion gives 
\be 
 \Gamma_{11}(u) = \mu_0^+ B_n(u) g_1^2 + o ( g_1^2 ) \label{LL.asymp.bias}
\ee
uniformly over $u\in [0,1]$, where $B_n(u)  \equiv \frac{1}{2} \varphi''(u) \mu_{L,2}$ if $u\in ({g_1},1-{g_1})$ and for $0\leq \rho\leq 1$,
\begin{align}
	B_n(u) \equiv  \frac{1}{2} \, \frac{ \{ \mu^+_{L,2}(\rho) \}^2-\mu^+_{L,1}(\rho)\mu^+_{L,3}(\rho)}{\mu^+_{L,2}(\rho)\mu^+_{L,0}(\rho)- \{\mu^+_{L,1}(\rho)\}^2} \begin{cases}
	 \varphi''(0^+) \quad & \mbox{ if } u = \rho {g_1}, \\
	 \varphi''(1^-) \quad & \mbox{ if } u = 1-\rho {g_1}.
	\end{cases}  \label{Bn.definition}
\end{align}

For $\Gamma_{12}(u)$, the argument that we employed before to prove (7.33) can be used to deal separately with the two components of $n^{-1} \sn  [ \mathbf{G}_i(u)- E \{ \mathbf{G}_i(u) \}  ] $, which are centered versions of
$$ 
  n^{-1} \sn L_{g_1}(U_i-u) \{\varphi^+_0(U_i)- \varphi^+_0(u) - (\varphi_0^+)'(u)(U_i-u)\} 
$$ 
and 
$$
n^{-1} \sn \bar{L}_{g_1}(U_i-u) \{\varphi^+_0(U_i)- \varphi^+_0(u) - (\varphi^+_0)'(u)(U_i-u)\},
$$
where $\bar{L}(t)=tL(t)$. This leads to
\be
\max_{u\in [0,1]} |\Gamma_{12}(u)|=    O_P\big\{ g_1^{3/2} n^{-1/2}(\log n)^{1/2} + g_1 n^{-1} \log n + n^{-1} \big\} = o_P ( g_1^2 ) . \label{LL.remainder}
\ee

Combining \eqref{ll.dec.1}, \eqref{LL.asymp.bias} and \eqref{LL.remainder}, we get
\begin{align}
	\hat{\varphi}^+_{0,\LL}(u) - \varphi^+_0(u) = \mu_0^+ B_n(u) g_1^2 + \Gamma_2(u) + o_P ( g_1^2)  \label{LL.stochastic.expansion}
\end{align}
uniformly over $u\in [0,1]$, where $\Gamma_2(u)$ is as in \eqref{Gamma2.equiv}.

\medskip
\noindent 
{\it Step 2. Proof of \eqref{stochastic.order.1}}. Let $x\in [a,b]$ be fixed. Substituting \eqref{LL.stochastic.expansion} into \eqref{lm.0.1} yields
\begin{align}
	 n^{-1} &\sn K_h(X_i-x)( w_0  -\hat w_X)(U_i)\, X_i  \nn  \\ 
		& =  g_1^2 n^{-1} \sn  K_h(X_i-x) X_i \frac{ B_{n}(U_i)}{\varphi(U_i)} \nn	\\
	& \quad +   ( \mu_0^+ n )^{-1} \sn K_h(X_i-x) X_i \,  \frac{\Gamma_2(U_i)}{\varphi(U_i)}  +    o_P ( g_1^2  )  + O_P(n^{-1/2}).   \label{predictor.bias.dec}
\end{align}
In what follows, we shall show that the first term on the right-hand side of \eqref{predictor.bias.dec} is the dominating term and the second one is negligible. 

First, using techniques similar to those we used before to deal with $n^{-1}\sn K_h(x-X_i)\{m(X_i)-m(x)\}$ as in (7.29), it can be proved that
\begin{align}
 g_1^2 n^{-1} \sn  K_h(X_i-x) X_i \frac{ B_{n}(U_i)}{\varphi(U_i)} =   E \{ K_h(X -x)X \}  E \{ B_{ n}(U)/\varphi(U) \} \, g_1^2 + o_P( g_1^2 ) , \nn
\end{align}
where, by standard calculations,
\begin{align}
	E \{  K_h(X-x)X  \} = x f_X(x) \mu_{K,0} +  O(h^2) . \label{v.exp.1}
\end{align}
Moreover, by the definition of $B_n$ in and above \eqref{Bn.definition},
\begin{align}
	   E \{  B_{n}(U)/\varphi(U) \} & = \int_{0}^1 \frac{B_n(u)}{\varphi(u)}  f_U(u) \, du \nn \\
	& =  \tfrac{1}{2}\mu_{L,2}\int_{ g_1 }^{1- g_1} \frac{\varphi''(u)}{\varphi(u)}  f_U(u) \, du +  \bigg( \int_0^{g_1} + \int_{1-g_1 }^1 \bigg) \frac{B_n(u)}{\varphi(u)}  f_U(u) \, du \nn \\
	& = \tfrac{1}{2} E \{ \varphi''(U)/\varphi(U)  \}\mu_{L,2} + O(g_1). \nn
\end{align}
The last three displays together imply that
\begin{align}
   g_1^2 n^{-1} \sn  K_h(X_i-x) X_i \frac{ B_{n}(U_i)}{\varphi(U_i)}   =  \tfrac{1}{2} x f_X(x) E \{ \varphi''(U)/\varphi(U)  \}  \mu_{L,2} \, g_1^2 + o_P ( g_1^2  ) . \label{predictor.bias.1}
\end{align}

Next, we focus on the second addend on the right-hand side of~\eqref{predictor.bias.dec}. For every $s\in \br$, put $\mathbf{w}(s)=(1,s)\T \in \br^2$. Substituting the expression ${\Gamma}_2(u)  =\mathbf{e}_1\T  \{ \mathbf S^*_n(u) \}^{-1} \,  n^{-1} \sum_{j=1}^n  L_{g_1}(U_j-u) \mathbf{w}\{ g_1^{-1}(U_j-u) \} \varphi(U_j)( | X_j| - \mu^+_0)$ into the second term on the right-hand side of~\eqref{predictor.bias.dec} yields
\begin{align}
	 &  (\mu_0^+ n^2 )^{-1} \sum_{i, j=1}^n  K_h(X_i-x)L_{g_1}(U_j-U_i ) X_i( |X_j| -\mu^+_0) \frac{ \varphi(U_j)}{\varphi(U_i)}   \mathbf{e}_1\T \mathbf{S}_n^*(U_i)^{-1} \mathbf{w}\{ g_1^{-1}(U_j-U_i)  \}       \nn \\
	 & =   ( \mu_0^+ n^2 )^{-1} \sum_{i\neq j}  K_h(X_i-x)L_{g_1}(U_j-U_i ) X_i( |X_j| - \mu^+_0) \frac{ \varphi(U_j)}{\varphi(U_i)}  \, \mathbf{e}_1\T \mathbf{S}_n^*(U_i)^{-1} \mathbf{w}\{ g_1^{-1}(U_j-U_i)  \}         \nn  \\ 
	 & \quad  +  L(0) ( \mu_0^+ n^2 g_1 )^{-1} \sn  K_h(X_i-x) X_i( |X_i| - \mu^+_0)   \mathbf{e}_1\T \mathbf{S}_n^*(U_i)^{-1} \mathbf{e}_1    \nn \\
	  &  = (\mu_0^+ n^2)^{-1} \sum_{i\neq j}  T_{ij}(x)   + O_P \{ (ng_1)^{-1}  \} ,      \label{predictor.bias.2}
\end{align}
where 
\begin{align}
	  T_{ij}(x) & \equiv K_h(X_i-x)L_{g_1}(U_j-U_i ) X_i( |X_j| - \mu^+_0) \nn \\
	& \quad \times   \{\varphi(U_i)\}^{-1}  \varphi(U_j)  \, \mathbf{e}_1\T \{ \mathbf{S}_n^*(U_i) \}^{-1} \mathbf{w}\{ g_1^{-1}(U_j-U_i)  \} .  \label{def:Tij}
\end{align}
Shortly we shall prove that
\be
	 n^{-2} \sum_{i\neq j} T_{ij}(x)  = O_P  \{ n^{-1/2} + (g_1 h)^{-1/2} n^{-1} \}, \label{sumTij.var.bound}
\ee
where, under Condition (B6), the right-hand side is of order $O_P(n^{-1/2}) + o_P(g_1^2)$. 

Together, \eqref{predictor.bias.dec}, \eqref{predictor.bias.1}, \eqref{predictor.bias.2} and \eqref{sumTij.var.bound} complete the proof of \eqref{stochastic.order.1}.

To prove \eqref{sumTij.var.bound}, it suffices to bound the variance of $\sum_{i\neq j} T_{ij}(x)$ since $ET_{ij}(x)=0$ for all $i\neq j$. First, observe that for $i\neq j$, $E(T_{ij}|X_i,U_i)=0$ and  
\be
	\xi_{j,x} \equiv  E\{ T_{ij}(x)|X_j,U_j\}  = E\{K_h(X_i-x) X_i\}  ( |X_j| - \mu^+_0)  Q_n(U_j), \label{def:xi}
\ee
where for $u \in[0,1]$,
\begin{align}
	Q_n(u ) & \equiv  \varphi(u)  \int_0^1   \mathbf{e}_1\T \{ \mathbf{S}_n^*(s) \}^{-1} \mathbf{w}\{ g_1^{-1}(u-s)\} L_{g_1}( u- s )  \frac{f_U(s)}{\varphi(s)} \, ds  \nn \\
	 &= \varphi(u) \int_{(u-1)/g_1}^{u/g_1}   \mathbf{e}_1\T \{ \mathbf{S}_n^*(u-g_1 s) \}^{-1} \mathbf{w}(s) L(s)  \frac{f_U(u-g_1 s)}{\varphi(u-g_1 s)} \, ds.  \label{def:Qn}
\end{align}
Next, using the inequality $(a+b)^2\leq 2a^2+2b^2$ which holds for all $a,b \in \br$, we have
\begin{align}
  E \bigg\{ \sum_{i\neq j} T_{ij}(x) \bigg\}^2& = E\bigg(  (n-1)\sum_{j=1}^n \xi_{j,x} + \sum_{i\neq j} \big[T_{ij}(x) -  E\{T_{ij}(x)|X_j,U_j\}  \big] \bigg)^2 \nn  \\
	& \leq 2n^2 E \bigg( \sum_{j=1}^n \xi_{j,x} \bigg)^2 + 2E\bigg( \sum_{i\neq j} \big[ T_{ij}(x) -  E\{ T_{ij}(x)|X_j,U_j\} \big] \bigg)^2. \label{sumTij.var.dec}
\end{align}
Since the $\xi_{j,x}$'s in \eqref{def:xi} are independent random variables with mean zero, we have 
\begin{equation}
	 E \bigg( \sum_{j=1}^n \xi_{j,x} \bigg)^2  = \sum_{j=1}^n E \xi_{j,x}^2 =  \var(|X| ) \big[ E\{ K_h(X-x)X \} \big]^2 E \{ Q_n(U) \}^2 \, n .  \label{sumTij.var.1}
\end{equation}

To bound the right-hand side of \eqref{sumTij.var.1}, note that uniformly over $u\in [2g_1, 1-2g_1]$ and $s\in [-1,1]$,
for $n$ sufficiently large,
\begin{align}
&	f_U(u-s g_1) = f_U(u) + O(g_1), \quad   \varphi(u-sg_1 )^{-1}  =  \varphi(u)^{-1} + O(g_1) \label{MiddInt1}\\
&  \mathbf{e}_1\T \mathbf{S}_n^*(u - sg_1)^{-1} \mathbf{w}( s ) L(s)  = f_U(u)^{-1} L^*(s) + O(g_1),   \label{MiddInt2}
\end{align}
where for the last term we used \eqref{S*n.definition}.

Recalling the definition of $Q_n$ at \eqref{def:Qn} and the fact that $L$ is supported on $[-1,1]$, \eqref{MiddInt1} and \eqref{MiddInt2} imply that for $n$ large enough and uniformly over $u \in [2g_1 , 1-2g_1] \subset [0,1]$,
$
  Q_n(u) =  \int  L^*(s)\, ds +O(g_1),
$
where $L^*(s) = \mathbf{e}_1\T \mathbf{M}_L^{-1} \mathbf{w}(s)L(s)$, with $\mathbf{M}_L$ as in \eqref{M.definition}, satisfies $\int  L^*(s)\, ds=1$.  We deduce that $E \{Q_n(U) \}^2 \leq \con $ Moreover, by \eqref{v.exp.1},
\be
	E \bigg( \sum_{j=1}^n \xi_{j,x} \bigg)^2   \leq \con\,  \{ xf_X(x) \}^2 \, n. \label{sumTij.var.2}
\ee 

Next we turn to the second term on the right-hand side of \eqref{sumTij.var.dec}. Put $\tilde{T}_{ij}(x)=T_{ij}(x)-E\{ T_{ij}(x)|X_j,U_j\}$ and note that 
$	E\{ \tilde{T}_{ij}(x) | X_i, U_i \} = E\{ \tilde{T}_{ij}(x) | X_j, U_j \} =0.$
In this notation, we have
\begin{align}
	 E \bigg\{ \sum_{i\neq j} \tilde{T}_{ij}(x)  \bigg\}^2 &= \sum_{i_1\neq j_1} \sum_{i_2\neq j_2} E\{\tilde{T}_{i_1j_1}(x)\tilde{T}_{i_2j_2}(x) \} \nn \\
	& = \sum_{i\neq j}  E\{\tilde{T}_{i j }(x) \}^2 + \sum_{i\neq j}  E\{\tilde{T}_{i j }(x)\tilde{T}_{ji }(x) \} + \sum_{i \neq j_1 \neq  j_2 }  E\{\tilde{T}_{i j_1 }(x) \tilde{T}_{i j_2 }(x)\}  \nn \\
	& \quad +\sum_{i\neq i_2 \neq j_1} E\{\tilde{T}_{i j_1}(x) \tilde{T}_{i_2 i}(x) \} +  \sum_{i_1 \neq i_2 \neq j}  E\{\tilde{T}_{i_1 j  }(x) \tilde{T}_{i_2 j  }(x)\}  \nn \\
	& \quad +  \sum_{i_1 \neq j_2 \neq j}  E\{\tilde{T}_{i_1 j  }(x) \tilde{T}_{j j_2  }(x)\}  + \sum_{i_1\neq j_1\neq i_2\neq j_2} E\{\tilde{T}_{i_1j_1}(x)\tilde{T}_{i_2j_2}(x) \}  \nn \\
	& \equiv   J_1 + J_2 + J_3 + J_4 + J_5 + J_6 + J_7. \nn
\end{align}
It is easy to see that $J_7=0$. For every $i \neq j_1 \neq j_2$, we also have
\begin{align*}
  & E\{\tilde{T}_{i j_1 }(x) \tilde{T}_{i j_2 }(x)\}   \\
  & = E\big[ E\{\tilde{T}_{i j_1 }(x) \tilde{T}_{i j_2 }(x)|X_i, U_i\} \big] =  E \big[ E\{\tilde{T}_{i j_1 }(x) | X_i, U_i \} E \{ \tilde{T}_{i j_2 }(x)|X_i, U_i \}   \big] = 0 ,
\end{align*}
which implies $J_3=0$. Similarly, it can be proved that $J_4=J_5=J_6=0$. For $J_2$, we have
$$
	E\{\tilde{T}_{i j }(x)\tilde{T}_{ji }(x) \} \leq \tfrac{1}{2} E\{\tilde{T}_{i j }(x)\}^2 + \tfrac{1}{2}E\{\tilde{T}_{ji }(x)\}^2,
$$
and hence $J_2\leq J_1$. Putting the above calculations together, we get
\be
	E\bigg\{ \sum_{i\neq j} \tilde{T}_{ij}(x)  \bigg\}^2 \leq 2 \sum_{i\neq j}E\{\tilde{T}_{i j }(x) \}^2  = 2n(n-1)E \{ \tilde{T}_{12}(x)\}^2, \label{sumTij.var.3}
\ee
where $\tilde{T}_{ij}(x)=T_{ij}(x)-E\{ T_{ij}(x)|X_j,U_j\}$ for $T_{ij}(x)$ as in \eqref{def:Tij} is such that
\begin{align}
 &  E \{ \tilde{T}_{12}(x)\}^2   \nn \\
 & \leq  E \{ T_{12}(x)\}^2 \nn \\
&  = E  \left( L_{g_1}^2(U_2-U_1) \{\varphi(U_2)/\varphi(U_1)\}^2  \left[ \mathbf{e}_1\T \mathbf{S}_n^*(U_1)^{-1} \mathbf{w}\{ g_1^{-1}( U_2 - U_1 )  \} \right]^2 \right)  \nn \\
& \quad \times E\{ X^2  K^2_h(X-x) \}  E (|X| -\mu^+_0)^2  \nn \\
& = g_1^{-1} \int_{-1}^1 \int_{0}^1 \frac{\varphi^2(u)}{\varphi^2(u-s g_1 )} \{ \mathbf{e}_1\T \mathbf{S}_n^*(u - s g_1 )^{-1} \mathbf{w}( s ) L(s)  \}^2  f_U(u-s g_1 ) f_U(u) \,  d u \, d s \nn\\
& \quad  \times \var(|X|)  \,  h^{-1}\int_{-1}^1  (x+th)^2 K^2(t) f_X(x+th) \, dt \,.
\label{sumTij.var.4}
\end{align}
To bound the right-hand side of \eqref{sumTij.var.4}, we write the double integral there as $\int_{-1}^1\int_0^1= \int_{-1}^1 \big( \int_0^{2g_1} + \int_{2g_1}^{1-2g_1} + \int_{1-2g_1}^1 \big)$, so that the last three displays jointly imply $E \{ \tilde{T}_{12}(x) \}^2 \leq \con \, x^2 \allowbreak  f_X(x)   (g_1 h)^{-1}  $. (Here, for the integral $\int_{2g_1}^{1-2g_1}$ we used  \eqref{MiddInt1} and \eqref{MiddInt2}, whereas we bounded $\int_0^{2g_1}$ and $\int_{1-2g_1}^1$ by $g_1$ times a constant.) Thus in view of \eqref{sumTij.var.3}, 
\be
E \bigg\{ \sum_{i\neq j} \tilde{T}_{ij}(x) \bigg\}^2   \leq \con\,  x^2 f_X(x)   (g_1 h)^{-1}n^2 \label{sumTij.var.5}
\ee
for all sufficiently large $n$.

Finally, \eqref{sumTij.var.dec}, \eqref{sumTij.var.2} and \eqref{sumTij.var.5} together prove \eqref{sumTij.var.bound}.

\medskip
\noindent 
{\it Step 3. Proof of \eqref{stochastic.order.1'}}. Following the lines of the proof of \eqref{stochastic.order.1} in {\it Step 2}, it can be proved that
\begin{align*}
& (nh)^{-1} \sn K'_h(x-X_i)  (w_0-\hat{w}_X)(U_i)\,X_i \{ m(X_i) -m(x) \}\\
& \qquad  \qquad = g_1^2 (nh)^{-1}  \sn K'_h(x-X_i) X_i \{m(X_i)-m(x)\} \frac{B_n(U_i)}{\varphi(U_i)} + O_P(n^{-1/2}) + o_P  ( g_1^2  ),
\end{align*}
where
\begin{align*}
 & g_1^2 (nh)^{-1} \sn K'_h(x-X_i) X_i \{m(X_i)-m(x)\} \frac{B_n(U_i)}{\varphi(U_i)} \\
	 & = h^{-1}g_1^2 E[K'_h(x-X)X\{m(X)-m(x)\}] E \{ B_n(U)/\varphi(U)\}  +O_P \{ g_1^2(nh)^{-1/2}  \} ,
\end{align*}
with $E \{B_n(U)/\varphi(U)\}= \frac{1}{2}E\{\varphi''(U)/\varphi(U) \} \mu_{L,2}+O(g_1)$ and
\begin{align*}
	& E[K'_h(x-X)X\{m(X)-m(x)\}]  \\
	& = \int K'(t)(x-ht)\{m(x-ht)-m(x)\}f_X(x-ht)\, dt \\
	& = - \bigg\{ x m'(x) f_X(x) \int t K'(t)\, dt \bigg\} \, h + O(h^2) = x m'(x) f_X(x)\mu_{K,0} \, h +O(h^2).
\end{align*}
The last two displays jointly complete the proof of \eqref{stochastic.order.1'}. 
\end{proof}

\begin{lemma} \label{lm.1}
{\rm
Let $w_0$ and $\hat{w}_Y$ be as in (7.2). If Conditions (2.2) and (B1)--(B6) are satisfied, then for every $x\in [a, b] \subseteq I_X$,
\begin{align}
	& n^{-1}\sn K_h(X_i-x)( \hat{w}_Y -w_0)(U_i)\, Y_i  \nn \\
	& \qquad \qquad \qquad = - \tfrac{1}{2}m(x)f_X(x) E  \{\psi''(U)/\psi(U) \} \mu_{L,2}  \, g_2^2  + o_P(g_2^2) + O_P(n^{-1/2}). 
	 \label{stochastic.order.2}
\end{align}
}
\end{lemma}

\begin{proof}[Proof of Lemma~\ref{lm.1}] 
Recall that under Condition (B6), $\delta^2_{0n}(g_2)=o(g_2^2)$. Arguments similar to those employed to derive \eqref{lm.0.1} can be used to show that
\begin{align}
& w_0 (u)-\hat w_Y(u)  = \frac{ \hat \psi^+_{0,\LL}(u) - \psi^+_0(u)}{\psi_0^+(u)} +  o_P(g_2^2) + O_P(n^{-1/2})     \label{lm.1.1}
\end{align}
uniformly over $u\in [0,1]$, where $\psi_0^+=m_0^+ \, \psi$ is as in (7.1). Similarly to \eqref{LL.stochastic.expansion}, the local linear estimator $\hat{\psi}^+_{0,\LL}$ of $\psi^+_0$ (see definition below (3.3)) admits the following stochastic expansion:
\begin{align}
 \hat \psi^+_{0,\LL}(u) - \psi^+_0(u)= m_0^+ \tilde{B}_{n}(u) g_2^2 + \Lambda_2(u)   + o_P(g_2^2)   \label{ll.exp.2}
\end{align}
uniformly over $u\in [0,1]$, where $\tilde{B}_n$ is as in \eqref{Bn.definition} with $\varphi$ replaced by $\psi$ and
\begin{align*}
	\Lambda_2(u) \equiv  \mathbf{e}_1\T \{ \mathbf{S}_n^*(u) \}^{-1}n^{-1} \sn  L_{g_2}(U_i-u ) \mathbf{w}\{ g_2^{-1}(U_i-u) \} \psi(U_i)( |Y_i| - m^+_0) .
\end{align*}

In view of \eqref{lm.1.1} and \eqref{ll.exp.2}, the arguments that we used earlier to derive \eqref{predictor.bias.dec}, \eqref{predictor.bias.1} and \eqref{predictor.bias.2} can be used to prove that 
\begin{align}
	& n^{-1} \sn K_h(X_i-x)\{w_0 -\hat w_Y(U_i)\} Y_i  \nn  \\ 
	  & = (m_0^+ n^2)^{-1} \sum_{i\neq j} K_h(X_i-x)L_{g_2}(U_j-U_i ) Y_i( | Y_j| - m^+_0) \frac{ \psi(U_j)}{\psi(U_i)} \mathbf{e}_1\T \{ \mathbf{S}_n^*(U_i) \}^{-1} \mathbf{w}\{ g_2^{-1}(U_j-U_i) \} \nn  \\ 
	 & \quad  +   E \{ K_h(X-x)Y \}  E  \{ \tilde{B}_{n}(U)/\psi(U) \} \, g_2^2 +  o_P(g_2^2) + O_P(n^{-1/2})  . \label{lm.1.2}
\end{align}
From there, calculations similar to those used in the proof of  Lemma~\ref{lm.0} show that the first term on the right-hand side of \eqref{lm.1.2} is of order $O_P(n^{-1/2}) + o_P(g_2^2)$, and thus is negligible. Similarly to \eqref{v.exp.1}, $E \{ K_h(X-x)Y \} =   m(x) f_X(x) + O(h^2)$. Further, by \eqref{Bn.definition} with $\varphi$ replaced by $\psi$ we have $E  \{ \tilde{B}_{n}(U)/\psi(U) \}=\frac{1}{2} E\{\psi''(U)/\psi(U)\} \mu_{L,2}+ O(g_2)$. Putting the above calculations together completes the proof of \eqref{stochastic.order.2}.
\end{proof}